\documentclass[11pt]{amsart}

\usepackage{upgreek,amssymb,amsmath,geometry,latexsym,graphics,graphicx,tabularx,shapepar,enumerate,
hyperref,stmaryrd,glossaries,color}
\usepackage{shuffle}
\usepackage[all,cmtip]{xy}

\DeclareSymbolFontAlphabet{\amsmathbb}{AMSb}

\numberwithin{equation}{section}

\newcommand{\ZZ}{\mathbb{Z}}
\newcommand{\FF}{\mathbb{F}}
\newcommand{\F}{\mathbb{F}}
\newcommand{\CC}{\mathbb{C}}
\newcommand{\C}{\mathbb{C}}
\newcommand{\QQ}{\mathbb{Q}}
\newcommand{\RR}{\mathbb{R}}
\newcommand{\R}{\mathbb{R}}
\newcommand{\GG}{\mathbb{G}}
\newcommand{\NN}{\mathbb{N}}
\newcommand{\TT}{\mathbb{T}}
\newcommand{\PP}{\mathbb{P}}

\newtheorem{Theorem}{Theorem}[section]
\newtheorem{Lemma}[Theorem]{Lemma}
\newtheorem{Proposition}[Theorem]{Proposition}
\newtheorem{Corollary}[Theorem]{Corollary}
\theoremstyle{definition}
\newtheorem{Definition}[Theorem]{Definition}
\theoremstyle{definition}
\newtheorem{Remark}[Theorem]{Remark}

\newcommand{\Mtwo}[2]{M_{#1,#2}}
\newcommand{\Hthree}[3]{H^{\langle#1\rangle}_{#2,#3}}
\newcommand{\Htwo}[2]{H_{#1,#2}}

\newcommand{\Deltad}[1]{\boldsymbol{\Delta}^{\langle#1\rangle}}
\newcommand{\partiald}[1]{\boldsymbol{\partial}^{\langle#1\rangle}}
\newcommand{\derd}[1]{\boldsymbol{d}^{\langle#1\rangle}}
\newcommand{\Deltan}{\boldsymbol{\Delta}}
\newcommand{\partialn}{\boldsymbol{\partial}}
\newcommand{\dern}{\boldsymbol{d}}
\newcommand{\leveln}[1]{#1^{\langle d\rangle}}
\newcommand{\dPn}{\mathcal{T}}

\newcommand{\dPs}[2]{\Gamma_{#2}^{\langle #1\rangle}}
\newcommand{\dPsn}[1]{\Gamma_{#1}}
\newcommand{\dPsinv}[2]{\big(\Gamma_{#2}^{\langle #1\rangle}\big)^{-1}}

\newcommand{\inv}{\ensuremath ^{-1}}
\newcommand{\twist}{^{(1)}}
\newcommand{\isom}{\ensuremath \cong}

\newcommand{\twistinv}{^{(-1)}}

\newcommand{\twistj}{^{(j)}}

\newcommand{\twistk}[1]{^{(#1)}}

\newcommand{\tpi}{\widetilde{\pi}}
\newcommand{\bh}{\mathbf{h}}
\newcommand{\bg}{\mathbf{g}}
\newcommand{\be}{\mathbf{e}}
\newcommand{\bz}{\mathbf{z}}
\newcommand{\bb}{\mathbf{b}}
\newcommand{\hM}{\widehat{M}}

\definecolor{ForestGreen}{rgb}{0.0, 0.5, 0.0}

\title[Carlitz operators]{non-commutative factorizations of higher sine functions in positive characteristic}
\date{\today}
\author{N. Green and F. Pellarin}

\address{Nathan Green\\ University of North Carolina at Charlotte\\ 9203 Mary Alexander Rd \\ Charlotte \\ North Carolina, 28223, USA}

\address{Federico Pellarin\\
Dipartimento di Matematica Guido Castelnuovo\\
Universit\`a Sapienza \\
Piazzale Aldo Moro 5 \\
00185 Rome, Italy}


\begin{document}

\maketitle

\begin{abstract}
In this paper we describe new non-commutative factorizations of functions related to $d$-th tensor powers of Carlitz's $\FF_q[\theta]$-module for $d\geq1$, called higher sine functions, related to previous works of the second named author. In \cite{PEL5}  
factorizations of this type have been constructed for operators which are combinations of powers of a Frobenius endomorphism with coefficients ``in $\operatorname{End}(\operatorname{End}(\GG_a^d))$''. In the present paper we succeed in determining factorizations with coefficients ``in $\operatorname{End}(\GG_a^d)$'' which are not easily deducible from \cite{PEL5}. One key ingredient in obtaining this is an application of a ``motivic pairing'' that the first named author introduced in \cite{GRE}. Another key ingredient is the notion of ``$\Delta$-matrix'' which comes into play in the analysis of the coefficients of the factorizations. Our results can be applied to explicitly describe analogues of shuffle $q^n$-powers for multiple polylogarithms, and to multiple zeta values of Thakur. All the identities we prove occur at the finite level. 
\end{abstract}

\tableofcontents

\section{Introduction}

The present paper deals with factorizations of certain ``higher sine functions'' (defined by (\ref{higher-sine}) below). These functions are associated to tensor powers of Carlitz's module (see Anderson and Thakur's \cite{AND&THA}), counterparts of Tate twists $\ZZ(d)$ in positive characteristic function field arithmetic. The choice of the terminology ``sine'' indicates that these are normalizations of exponentials of these modules. The factorizations we are interested in hold in certain non-commutative algebras of formal series in powers of the Frobenius endomorphism, and are called `non-commutative' because the factors do not commute each other. In order to introduce the reader to these themes and to our results, we start with a reminder of classical investigations by Euler, adopting for this and other aspects, as a reference, the book \cite{BUR&FRE} of Fr\'esan and Burgos Gil.

Our results are described starting from \S \ref{results}.

\subsubsection*{Euler's setting}
Euler's factorization of the sine function
\begin{eqnarray*}
\frac{\sin(\pi z)}{\pi z}&=&\prod_{k\geq 1}\left(1-\frac{z^2}{k^2}\right)\\
&=&\sum_{n\geq0}(-1)^n\frac{\pi^{2n}}{(2n+1)!}z^{2n}
\end{eqnarray*}
allows, by comparison of the coefficients of $z^2$, to deduce Euler's formula $$\zeta(2)=\frac{\pi^2}{6},$$
with $\zeta$ Riemann's zeta function. For the coefficients of higher powers of $z$, this formula also gives explicit identities for 
Euler-Zagier multiple zeta values of `parallel weight' 
\begin{equation}\label{identity-mzv-euler}\zeta(\underbrace{2,\ldots,2}_{n\text{ times}})=\frac{\pi^{2n}}{(2n+1)!}.\end{equation} To get explicit formulas for special values of $\zeta$ at even positive integers one uses the cotangent. We have:\begin{eqnarray*}
\cot(\pi z)&=&\pi^{-1}\frac{d}{dz}\log\big(\sin(\pi z)\big)\\
&=&\frac{1}{\pi z}\Big(1-2\sum_{k\geq1}\zeta(2k)z^{2k}\Big)\in\frac{1}{\pi z}-(\pi z)\QQ[[(\pi z)^2]].
\end{eqnarray*}
From this one deduces Euler's 1735 formula involving Bernoulli's numbers $B_n$:
\begin{equation}\label{euler-formulas}
\zeta(2k)=\frac{(-1)^{k-1}B_{2k}}{2(2k)!}(2\pi)^{2k},\quad k\in\NN^*,\quad \frac{x}{e^x-1}=:\sum_{n=0}^\infty B_n\frac{x^n}{n!}
\end{equation}
($\NN^*$ denotes the set of positive integers). In the above discussion, the key point is that $\sin(z)\in\QQ[[z]]$ and that the factors of $\sin(\pi z)/(\pi z)$ are in $\QQ[z]$.

\subsubsection*{Carlitz's setting}
Two centuries after Euler's investigations, Carlitz \cite{CAR} explored parallel structures in the framework of global function fields of positive characteristic. Let $\FF_q$ be the finite field with $q$ elements and characteristic $p>0$. Carlitz discovered formulas analogous of (\ref{euler-formulas}) and from his work it is easy to deduce analogues of (\ref{identity-mzv-euler}). Carlitz was guided by analogies that can be summarized in the following table:

$$\begin{tabular}{c|l} 
$\ZZ$ & $A:=\FF_q[\theta]$\\
$\QQ$ & $K:=\operatorname{Frac}(A)$\\
$\RR$ & $K_\infty:=\widehat{K}_{|\cdot|}=\FF_q((\theta^{-1}))$\\
$\CC$ & $\CC_\infty:=\widehat{K_\infty^{\text{sep}}}$ (completion of a separable closure)\\
$\NN^*$ & $A^+:=\{a\in A:a\text{ monic}\}$\\
\end{tabular}$$
In this table $K$ denotes the fraction field of $A$, the $\FF_q$-algebra of polynomials in an indeterminate $\theta$ with coefficients in the finite field $\FF_q$, $K_\infty$ denotes the local field which is the completion $\widehat{K}_{|\cdot|}$ of $K$ at the infinity place, to which we associated a norm $|\cdot|$ (with the property that $|\theta|>1$ and uniquely determined by this value in $\theta$), that has $\frac{1}{\theta}$ as a uniformizer. Also, $\CC_\infty$ denotes the completion of a separable closure of $K_\infty$ that we fix once and for all; it  plays the role of the field of complex numbers in the paper; it is complete and algebraically closed though not locally compact, unlike $\CC$. 
Finally, $\NN^*$ denotes the set of positive integers, and the choice of the set of monic polynomials
as an analogue of it means that we made the choice of a sign function $K_\infty^\times\rightarrow\FF_q^\times$ (see Goss \cite[Def. 7.2.1]{GOS}).
Further analogies can be noticed observing that both $\ZZ$ and $A$ are euclidean, and $\ZZ$ is discrete and co-compact in $\RR$, while 
$A$ is discrete and co-compact in $K_\infty$ (the infinity place is the only one carrying this property).

Let us write, with $z\in\CC_\infty$,
$$\sin_A(z):=z\prod_{a\in A^+}\Big(1-\frac{z^{q-1}}{a^{q-1}}\Big)=z\prod_{b\in A\setminus\{0\}}\Big(1-\frac{z}{b}\Big).$$
The product converges to an entire $\FF_q$-linear function $\CC_\infty\rightarrow\CC_\infty$, hence is surjective
(in \cite[\S 3.2]{GOS} the reader can find tools to verify these properties). We call this function {\em Carlitz's sine function}. It is the analogue of the factor $\sin_\ZZ$ in 
$$\sin(\pi z):=\pi \underbrace{z\prod_{n\geq1}\Big(1-\frac{z^2}{n^2}\Big)}_{\text{$\sin_\ZZ(x)$}}.$$
The crucial tools that Carlitz introduced are the {\em Carlitz module} $C$ (an analogue of the multiplicative group scheme $\GG_m$ associating to any commutative ring its group of units), and its {\em exponential function} $\exp_C(z)$, which belongs to $K[[z]]$, see \cite[\S 3.2]{GOS}.
Using them we see that the factorization of $\sin_A$ produces identities related to {\em Thakur's multiple zeta values}, introduced in \cite{THA0}, the simplest of which is an analogue of the formula $\zeta(2)=\frac{\pi^2}{6}$, and the use of 
an analogue $\cot_A$ of the classical cotangent allowed Carlitz to discover identities for so-called {\em Carlitz's zeta values}, analogous to Euler's formulas for values of $\zeta$ at positive even integers. We quickly review these basic facts.

\subsubsection*{Carlitz's module} The Carlitz module can be viewed as an analogue in function field arithmetic of the functor $\GG_m$ from commutative rings to abelian groups, see \cite{GOS} and \cite{PAPI}. In this paper, all $A$-algebras are supposed to be commutative unless otherwise specified.
Let $B\subset\CC_\infty$ be an $A$-algebra with the induced algebra structure. Carlitz's module $C(B)$ over $B$ is the $\FF_q$-vector space $B$ with the $A$-module structure determined by the multiplication by $\theta$ given by:
$$C_\theta(b)=\theta b+\tau(b)=\theta b+b^q,\quad b\in B,$$ where $\tau$ is the $\FF_q$-linear endomorphism of $\GG_a(B)$ given by $c\mapsto c^q$ (``$q$-Frobenius endomorphism''). Carlitz's module $C$ over $\CC_\infty$ is {\em uniformizable}. That is, there exists an entire $\FF_q$-linear endomorphism $\exp_C:\CC_\infty\rightarrow\CC_\infty$
such that $\frac{d}{dz}(\exp_C(z))=1$ and, for all $z\in\CC_\infty$:
$$C_\theta\big(\exp_C(z)\big)=\exp_C(\theta z).$$ This is easy to prove because we can actually identify $\exp_C$ with an explicit formal series
$$\exp_C=\sum_{k\geq 0}\frac{1}{D_k}\tau^k\in K[[\tau]],$$ with 
$D_0=1$ and 
\begin{equation}\label{coefficients-D}
D_k=(\theta^{q^k}-\theta)D_{k-1}^q,\quad k\geq1.
\end{equation} 
Here and in all the following, given $B$ an $\FF_q$-algebra, we denote by $B[\tau]$ (resp. $B[[\tau]]$)
the skew ring of finite sums (resp. formal series) $\sum_{i\geq0}b_i\tau^i$ with coefficients $b_i\in B$, with the unique ring structure arising from the commutation rule $\tau b=\tau(b)\tau=b^q\tau$.
Note the fundamental property that
\begin{equation}\label{funda} \frac{1}{D_k}\in K,\quad k\geq0,\end{equation}
which is obvious from the construction. Analogously, we have $$\exp(z)=e^z=\sum_{k\geq0}\frac{z^k}{k!}\in\QQ[[z]],$$ so that $\sin(z)=\frac{e^{iz}-e^{-iz}}{2i}\in\QQ[[z]]$.  A simple computation of Newton polygons (at the infinite place of $K$) shows that there exists $\widetilde{\pi}\in\CC_\infty$ (in fact $\widetilde{\pi}\in\CC_\infty\setminus K_\infty$ if $q\neq 2$) such that $\operatorname{Ker}(\exp_C)=\widetilde{\pi}A$. We deduce right away that, viewing $\sin_A$ no longer as an $\FF_q$-linear, entire function, but as 
a formal series in powers of $\tau$,
\begin{equation}
\sin_A=\widetilde{\pi}^{-1}\exp_C\widetilde{\pi}=\sum_{i\geq 0}\frac{\widetilde{\pi}^{q^i-1}}{D_i}\tau^i\in K_\infty[[\tau]].
\end{equation} 
Note that $\widetilde{\pi}$ is defined up to a factor in $\FF_q^\times$, but $\sin_A$ does not depend on this choice. We can compute $\widetilde{\pi}$ by means of the following formula:
\begin{equation}\label{def-pi}
\widetilde{\pi}=\theta(-\theta)^{\frac{1}{q-1}}\prod_{i\geq 1}\Big(1-\frac{\theta}{\theta^{q^i}}\Big)^{-1}\in K_\infty[(-\theta)^{\frac{1}{q-1}}]^\times,
\end{equation}
which can be applied to prove its transcendence \cite{WAD}: $\widetilde{\pi}\in\CC_\infty\setminus K^{\text{ac}}$ ($K^{\text{ac}}$ denotes the algebraic closure of $K$ in $\CC_\infty$).
 From now on, $(-\theta)^{\frac{1}{q-1}}$ designates a chosen $(q-1)$-th root of $-\theta$ in $\CC_\infty$; this fixes a choice of $\widetilde{\pi}$ in (\ref{def-pi}).
Thanks to (\ref{funda}) one recovers Carlitz's original result (see \cite{CAR}), where 
$$\zeta_A(n):=\sum_{a\in A^+}\frac{1}{a^n}\in K_\infty,\quad n\geq1.$$

\medskip

\noindent{\bf Theorem}  (Carlitz).
{\em For all $k\geq 1$ we have
\begin{equation}\label{carlitz-formulas}
\zeta_A\big(k(q-1)\big)\in K^\times\widetilde{\pi}^{k(q-1)}.\end{equation}
}

\medskip

We mention, for completeness, that in his work, Carlitz discussed certain analogues of Bernoulli numbers in $K$ called {\em Bernoulli-Carlitz elements}. The theory can be further developed along the lines of Kummer's ``ideal numbers'' which can be here encoded in Taelman's class modules and class number formulas that have connections with Bernoulli-Carlitz elements. See for example \cite{GOS,APTR,TAE,TAE2} (non-exhaustive suggestions).

\subsubsection*{Non-commutative factorization of $\sin_A$}

In this paper we are interested in another type of factorization of $\sin_A$.
The {\em non-commutative factorization} of $\sin_A$ no longer sees $\sin_A$ as an entire function but as an element of the non-commutative algebra $K_\infty[[\tau]]$, and can be proved essentially applying Carlitz's theory:
\begin{equation}\label{non-comm-fact-sine}
\sin_A=\prod^{\longleftarrow}_{i\geq 0}\big(1-\mathcal{L}_i\tau\big),
\end{equation}
where $\mathcal{L}_i:=l_{i}^{1-q}$ and the sequence $(l_i)_{i\geq0}$ is given by the recursion
\begin{equation}\label{coeff-log-carlitz}
l_0=1,\quad l_i=(\theta-\theta^{q^i})l_{i-1},\quad i>0
\end{equation}
and the arrow on top of the product sign means that the factors are nested from right to left, so that 
to compute the product, one starts with the first factor $(1-\tau)$, and then multiplies it on the left
by the other factors: $\cdots (1-\mathcal{L}_1\tau)(1-\tau)$. See \cite[Proposition 4.4.9]{PEL5}. There is a direct connection to Thakur's multiple zeta values which is described in 
\cite[Remark 4.4.11]{PEL5}, and this paper contains generalizations of these observations. Incidentally, the sequence (\ref{coeff-log-carlitz}) characterizes the {\em Carlitz's logarithm}
\begin{equation}\label{def-logarithm}
\log_C:=\sum_{i\geq 0}\frac{\tau^i}{l_i}\in K[[\tau]],
\end{equation}
(\footnote{Here $\frac{\tau^i}{l_i}=l_i^{-1}\tau^i$ and we will adopt similar notations in the following.}), which is the formal inverse $\log_C=\exp_C^{-1}$ of $\exp_C$ in $K[[\tau]]$. 
Another way to rewrite (\ref{non-comm-fact-sine}) is by using function composition. On every non-empty bounded subset $U$
of $\CC_\infty$ (e. g. a disk containing $0$) we have that, uniformly,
\begin{equation}\label{sine-composition-factorization}
\sin_A(z)=\lim_{k\rightarrow\infty}\Big((1-\mathcal{L}_{k-1}\tau)\circ \cdots \circ (1-\mathcal{L}_{1}\tau)\circ(1-\tau)(z)\Big),\quad z\in U,
\end{equation}
where $\tau$ acts as the $q$-power Frobenius. Similar non-commutative factorizations, in the framework of rank-one Drinfeld $A$-modules associated to a ring $A$ of regular functions on a smooth projective, geometrically irreducible curve, away from a fixed but arbitrarily chosen closed point, have been obtained in \cite{CHU&NGO&PEL} (with ``power sums'' in the place of the coefficients $\mathcal{L}_i$).

\subsubsection*{Tensor powers of Carlitz's module}

In order to present the results of our paper, we now introduce the main ingredients. Choose $d>0$ and consider the endomorphism
\begin{equation}\label{E:C otimes def}
C^{\otimes d}_\theta:=\begin{pmatrix}\theta & 1 & 0 & \cdots &0 & 0\\ 
0 & \theta & 1  & \cdots &0& 0\\ 
\vdots & \vdots & \vdots &  &\vdots & \vdots\\ 
0 & 0 & 0 & \cdots &1& 0\\
0 & 0 & 0 & \cdots &\theta& 1\\
\tau & 0 & 0 & \cdots &0& \theta
\end{pmatrix}:=\theta+N+e_{d,1}\tau\in\operatorname{End}_{\FF_q}\big(\GG_a^d(A)\big).
\end{equation}
In this expression $e_{i,j}$ denotes the elementary $d\times d$ matrix having its only non-zero coefficient being equal to one, in the line $i$ and column $j$, and
$N=e_{1,2}+e_{2,3}+\cdots+e_{n-1,n}$.
Let $B$ be an $A$-algebra. We define an $A$-module structure $C^{\otimes d}(B)$ on $\GG_a^d(B)$ identified with $B^{d\times1}$ by letting $A$ act via the Carlitz module for $d=1$, and letting $A$ act via $C^{\otimes d}$ for $d>1$. This is the {\em $d$-th tensor power of Carlitz's module} $C^{\otimes d}$, and is an {\em Anderson $A$-module} (or {\em Anderson $t$-module}) of rank one and dimension $d$ \cite[Chapter 5]{GOS} and \cite[\S 3.4]{PAPI}. It was introduced by Anderson and Thakur in \cite{AND&THA}, see also Brownawell and Papanikolas' \cite{BRO&PAP}. It provides an algebro-geometric structure supporting a function field analogue of Tate's twist $\ZZ(d)$, see \cite[\S 1.11]{AND&THA}. In particular
there are isomorphisms of Galois modules $\operatorname{Ker}(C_a^{\otimes d})\cong\operatorname{Ker}(C_a)^{\otimes d}$ for all $d\geq 1$ and $a\in A$ (the second tensor power is taken over $A$, see \cite[Proposition 1.11.1]{AND&THA}). Moreover $C^{\otimes d}(\CC_\infty)$ is {\em uniformizable}. This surprising property (\footnote{There are no signs that similar uniformizations behind $\ZZ(d)$ exist for general choices of $d$.}) was proved by Anderson and Thakur \cite[\S 2.2]{AND&THA}; it means that there exists an exponential function $\exp_{C^{\otimes d}}$ fitting into the exact sequence 
(\ref{exact-sequence}), see section \S \ref{section-carlitz-tensor}.

\subsubsection*{Higher sine functions} Based on the one dimensional theory we can construct higher dimensional analogues of Carlitz's sine function $\sin_A$. Fix a dimension $d>1$. We choose a generator $\Pi_d\in\CC_\infty^{d\times 1}$ of the free rank one $A$-module whose elements are the zeroes of the entire function $\exp_{C^{\otimes d}}$.
We construct, starting from $\Pi_d$, a matrix $\widehat{\Pi}_d$ in the following way: 
\begin{equation}\label{hat-pi}
\Pi_d=:\begin{pmatrix} \widetilde{\pi}_{d-1}\\ \widetilde{\pi}_{d-2}\\ \vdots \\ \widetilde{\pi}_{0}\\ \end{pmatrix},\quad \widehat{\Pi}_d:=\begin{pmatrix} \widetilde{\pi}_{0} & \widetilde{\pi}_{1} & \cdots & \widetilde{\pi}_{d-1}\\ 
0 & \widetilde{\pi}_{0} & \cdots & \widetilde{\pi}_{d-2}\\ \vdots & \vdots & & \vdots \\ 0 & 0 & \cdots &\widetilde{\pi}_{0}\end{pmatrix}\in\operatorname{GL}_d(\CC_\infty).\end{equation}

By \cite[Corollary 2.5.8]{AND&THA} we can choose $\Pi_d$ so that the identity $\widetilde{\pi}_0=\widetilde{\pi}^d$ holds (multiplying by a factor in $\FF_q^\times$).

Finally we define the 
{\em sine function of order $d$} to be 
\begin{equation}\label{higher-sine}
\sin^{\otimes d}_A:=\widehat{\Pi}_d^{-1}\exp_{C^{\otimes d}}\widehat{\Pi}_d\in\operatorname{End}_{\CC_\infty}\big(\GG_a^d(K_\infty)\big)[[\tau]].\end{equation}
More precisely the above formal series in powers of $\tau$ has the coefficients in the ring $\operatorname{End}_{K_\infty}\big(\GG_a^d(K_\infty)\big)$, and they are all invertible. This follows easily from our Theorem B, see below.
This formal series, again independent of the choice of $\Pi_d$, represents a surjective vector-valued $\FF_q$-linear entire function of $d$ variables. From now on we omit the subscript $d$ in $\Pi_d,\widehat{\Pi}_d$ writing more simply $\Pi,\widehat{\Pi}$
in case the value of $d$ is understood; this simplifies our notations.

The kernel of $\sin_A^{\otimes d}$ is:
$$\operatorname{Ker}(\sin_A^{\otimes d})=\left\{\boldsymbol{d}_\theta(a)\begin{pmatrix}0\\ \vdots \\ 0\\ 1\end{pmatrix}:a\in A\right\},$$ 
where $\boldsymbol{d}_\theta(a)$ is the multiplication by 
$a$ for the $A$-module structure induced by $\operatorname{Lie}(C^{\otimes d})$, defined in \S \ref{Papanikolas-matrices}
(see Lemma \ref{computation-kernel-lemma}). It can be easily computed evaluating $a$, seen as an $\FF_q$-linear combination 
of powers of an indeterminate $t$, at $t=\theta+N=\boldsymbol{d}_\theta(\theta)$ that can be denoted by 
$a(\theta+N)$ (evaluation). This kernel, with this structure, is free rank one $A$-module. If $d=1$, we see that 
$\sin^{\otimes d}_A=\sin_A$, Carlitz's sine function.

\subsection{Results}\label{results}

\subsubsection*{Non-commutative factorizations} The first main result that we present is a factorization formula of
certain endomorphisms of $\GG_a^d$ (Carlitz operators). 
To state it we need to introduce a class of partial differential operators that we use in the paper (the basic properties are collected in \S \ref{xyz-formalism}). For a field extension $F/\FF_q$, and $x$ an indeterminate over $F$ we denote by $(\mathcal{D}_{x,j})_{j\geq 0}$ the unique family of $F$-linear higher divided derivatives over $F(x)$ such that
for all $i,j\in\ZZ$, $0\leq i\leq j$, $$\mathcal{D}_{x,j}(x^i)=\binom{i}{j}x^{i-j}.$$
Given now $x,z$ two independent indeterminates over $F$ and $f\in F(x,z)$, we set:
\begin{equation*}
\Deltan_{x,z}(f):=
\begin{pmatrix}\mathcal{D}_{x,d-1}(f) & \mathcal{D}_{x,d-1}\big(\mathcal{D}_{z,1}(f)\big) & \cdots & \mathcal{D}_{x,d-1}\big(\mathcal{D}_{z,d-1}(f)\big)\\
\vdots & \vdots & & \vdots \\
\mathcal{D}_{x,1}(f) & \mathcal{D}_{x,1}\big(\mathcal{D}_{z,1}(f)\big) & \cdots & \mathcal{D}_{x,1}\big(\mathcal{D}_{z,d-1}(f)\big)\\
f & \mathcal{D}_{z,1}(f) & \cdots & \mathcal{D}_{z,d-1}(f)
\end{pmatrix}
\in F(x,z)^{d\times d}.\end{equation*} Consider three independent indeterminates $x,y,z$.
For $d\geq 1$ we set (\footnote{Note that this is a polynomial of $\ZZ[x,y,z]$ reduced modulo $p$ to give 
$f_d\in\FF_p[x,y,z]$.})
\begin{equation}\label{fd}
f_d(x,y,z):=\frac{(z-y)^d-(x-y)^d}{z-x}=\sum_ {j =0}^{d-1}(z-y)^j(x-y)^{d-1-j}\in\FF_p[x,y,z].
\end{equation} We also set $f_0=0$. This sequence was used crucially in \cite{PEL6}, it is denoted by $S_d$ in that paper. In \cite[Theorem E]{PEL6} the second author observed a connection between specializations of these polynomials, the {\em Bernoulli-Carlitz elements} $\operatorname{BC}_{n(q-1)}$ and the {\em Anderson-Thakur polynomials} $H_{(q-1)n-1}$. In the notation of ibid. 
we have the identity in $K[Y^q]$:
$$(\theta-\theta^q)\operatorname{BC}_{n(q-1)}f_n(\theta,Y^q,\theta^q)=H_{(q-1)n-1}(Y).$$ This formula can also be recovered from our results.

Define, for $i\geq 0$,
\begin{equation}\label{def-calL}
\mathcal{L}_i:=\Deltan_{x,z}\left(f_d(x,y,z)\Big(\frac{(x-\theta^q)\cdots(x-\theta^{q^i})}{(z-\theta^{q^2})\cdots(z-\theta^{q^{i+1}})}\Big)^d\right)_{\begin{smallmatrix}x=\theta\\ y=\theta^{q^{i+1}}\\ z=\theta^q\end{smallmatrix}}\in\operatorname{GL}_d(K).\end{equation}
Note that if $d>1$, $\mathcal{L}_0$ is not the identity matrix. To see that $\mathcal{L}_k$ is indeed invertible for all $k$ one can apply our Lemma \ref{first-compatibility} and notice that 
$\Deltan_{x,z}(f_d(x,y,z))$ is invertible (it is easily seen to be a lower triangular matrix with ones on the diagonal).

In \cite{PAP}, Papanikolas considers certain higher dimensional variants of {\em Carlitz's polynomials} that we denote by $E_k$
(in the case $d=1$ those are the polynomials $D_{n}^{-1}f_{<n}(x)$ in \cite[Lemma 5.4.2]{PAPI}, the polynomials 
$D_d^{-1}e_d(x)$ in \cite[Theorem 3.1.5]{GOS}, and the polynomials $\Psi_d$ in \cite[(3.4.1)]{AND&THA}). They can be identified with elements of the algebra $\operatorname{End}_K(\operatorname{End}_{K}(\GG_a^d(K)))[\tau]$ (see also \cite{PEL6}; the definition is recalled in our \S \ref{section-carlitz-operators}).
We further introduce, in \S \ref{section-non-commutative}, certain normalizations $$\boldsymbol{E}_k\in \operatorname{End}_K\big(\GG_a^d(K)\big)[\tau],\quad k\geq 0$$ of {\em Carlitz's operators} $E_k$ restricted over a certain subspace of $\operatorname{Lie}(C^{\otimes d})(\CC_\infty)$ (see (\ref{def-normalization})). The next theorem is restated as Theorem \ref{theorem-factor-complete} in the text.

\medskip

\noindent{\bf Theorem A.} {\em The following factorization holds:
$$\boldsymbol{E}_k=\big(1-\mathcal{L}_{k-1}\tau\big)\big(1-\mathcal{L}_{k-2}\tau\big)\cdots\big(1-\mathcal{L}_{0}\tau\big).$$}

\medskip

From Theorem A we deduce a non-commuta\-tive factorization of $\sin_A^{\otimes d}$ (see Theorem \ref{factorization-of-sine-complete}):

\noindent{\bf Theorem B.} 
{\em The product below converges and the following identity holds:
$$\sin_A^{\otimes d}=\prod_{i\geq 0}^{\longleftarrow}\Big(1-\mathcal{L}_i\tau\Big)\in\operatorname{End}_{K_\infty}\big(\GG_a^d(K_\infty)\big)[[\tau]].$$}

\medskip

Convergence here means that the associated sequence of $\FF_q$-linear functions 
$\CC_\infty^d\rightarrow\CC_\infty^d$ converges uniformly on every bounded subset of $\CC_\infty^d$.
In the case $d=1$ we find (\ref{non-comm-fact-sine}). From both Theorems A and B, choosing $d=1$, we recover corresponding results of \cite{PEL5}, see 
\cite[Proposition 4.4.9]{PEL6} and its proof. 

\subsubsection*{Application to Carlitz multiple polylogarithms}

As previously noticed, the factorization (\ref{non-comm-fact-sine}) (this is the case $d=1$) and its `finite variants' allow to construct
explicit non-trivial $K$-linear dependence relations among multiple zeta values of Thakur, see \cite{PEL5,GEZ&PEL}. At once (\ref{non-comm-fact-sine}) implies that non-trivial linear dependence relations hold among values of Anderson-Thakur polylogarithms. 
In this paper we describe a family of linear dependence relations of 
{\em Carlitz multiple polylogarithms} (\footnote{The extensive use of which has been inaugurated by Chang in \cite{CHA} with the purpose of proving an analogue of Goncharov's conjecture for Thakur's multiple zeta values, rooted on the idea of polylogarithm as the reader can find in \cite[\S 3]{AND&THA}, see also \cite{GOS}, and ultimately, in the work of Carlitz.}), see Theorem C below. While comparing them with the existing literature, we will also recall the above mentioned results.

Consider, for $r\geq0$, positive integers $n_0,\ldots,n_r\in\NN^*$. Consider also elements $a_1,\ldots,a_r\in A$ and a variable $X$. We write:
\begin{equation}\label{definition-L}
\mathcal{L}\begin{pmatrix} X & a_1 &\ldots& a_r\\ n_0 & n_1&\ldots& n_r\end{pmatrix}:=\sum_{i_0>i_1>\cdots>i_r\geq 0}\frac{X^{q^{i_0}}a_1^{q^{i_1}}\cdots a_r^{q^{i_r}}}{l_{i_1}^{n_1}l_{i_2}^{n_2}\cdots l_{i_r}^{n_r}}\in K[[X]].
\end{equation}
The $\FF_q$-span of these formal series is generated by those in which the polynomials $a_i$ are monic monomials in $\theta$.
If $X\in\CC_\infty$, $|X|<|\theta|^{n_0\frac{q}{q-1}}$, and $m_i<n_i\frac{q}{q-1}$ for all $1\leq i\leq r$ where $\deg_\theta(a_i)=m_i$,
it is easy to see that the series 
$$\mathcal{L}\begin{pmatrix} X & a_1 &\ldots& a_r\\ n_0 & n_1&\ldots& n_r\end{pmatrix}$$
converges in $\CC_\infty$. These are special values of {\em Carlitz multiple polylogarithms}, studied in wider generality in Chang's \cite{CHA} (Chang also uses these series to show that the $K$-algebra of Thakur's multiple zeta values is graded by weights).

In Theorem C below, we describe certain explicit families of non-trivial $K$-linear dependence relations among these series. To describe the coefficients of these linear relations expand, for indeterminates $y_1,\ldots,y_r$, the polynomials
$$f_d(\theta,y_1,\theta^{q^r})f_d(\theta^q,y_2,\theta^{q^r})\cdots f_d(\theta^{q^{r-1}},y_r,\theta^{q^r})=\sum_{\underline{m}}c_{\underline{m}}y_1^{m_1}y_2^{m_2}\cdots y_r^{m_r}\in A[y_1,y_2,\ldots,y_r],$$
where the finite sum runs over $r$-tuples $\underline{m}=(m_1,m_2,\ldots,m_r)\in\NN^r$
and the coefficients $c_{\underline{m}}$, uniquely defined, are in $A$. From the definition
(\ref{fd}) it is visible that if $\underline{m}$ is such that there exists $j$ with $m_j\geq d$, then 
$c_{\underline{m}}=0$. So the non-zero coefficients are in natural correspondence with the points of a hypercube and there are at most $d^r$ non-zero such coefficients.

Our next result describes certain linear dependence relations with coefficients in $A$ (see Corollary \ref{finite-version-imply-Theorem-C}; we set $\tau^r(X^i)=X^{iq^r}$):

\medskip

\noindent{\bf Theorem C. }\label{T:Main Theorem C}{\em For any choice of $d,r\geq 1$ the following formula holds:
$$\tau^r\left(\mathcal{L}\binom{X}{d}\right)=(-1)^r(\theta^{q^r}-\theta)^{1-d}D_r^{d}\sum_{\underline{m}}c_{\underline{m}}
\mathcal{L}\begin{pmatrix} X&\theta^{m_1q} & \cdots &\theta^{m_rq^r}\\ d&d(q-1) &\cdots& d(q-1)q^{r-1}\end{pmatrix}.$$}

\medskip

To prove this result we proceed by induction on $r\geq 0$ with a study of the 
projection of the matrix identities of Theorem A on the coefficients situated on the $d$-th line first column of the matrices representing the endomorphisms in the canonical basis. The main result is Theorem \ref{nathan-theorem}, of which Corollary \ref{finite-version-imply-Theorem-C} is a consequence.
Projecting the initial matrix identities of Theorem A to scalar identities of Carlitz multiple polylogarithms is quite a delicate recursive procedure that is described in \S \ref{section-projecting-on-one-coord}.
In the case $d=1$ it is clear that $f_d=1$. Then the sum in the right-hand side of the formula of Theorem C reduces to the unique contribution of $\underline{m}=(0,\ldots,0)$ independently on $r$.
We deduce the formula
$$\tau^r\left(\mathcal{L}\binom{X}{1}\right)=(-1)^rD_r\mathcal{L}\begin{pmatrix} X&1 & \cdots &1\\ 1&q-1 &\cdots& (q-1)q^{r-1}\end{pmatrix}$$ 
which agrees with \cite[Theorem 7.2]{GEZ&PEL} and a formula in \cite{LAR&THA}. The reason is that when $m:=\max\{m_1,\ldots,m_r\}=0$, 
any series as in (\ref{definition-L}) with $X=1$ is a $K$-linear combination of multiple zeta values of Thakur. In fact the unital $K$-algebras of multiple zeta values and of multiple Carlitz polylogarithms `at one' agree as a consequence of Ngo Dac's fundamental result \cite[Theorem A]{NGO} (analogue of Brown's theorem in \cite{BROW}).

In the case $r=1$ and arbitrary $d$ it is not difficult to compute the coefficients $c_{\underline{m}}$. Indeed 
by (\ref{fd-expansion}), the formula of Theorem C becomes, in this case:
$$\tau\left(\mathcal{L}\binom{X}{d}\right)=\sum_{i =0}^{d-1}(-1)^i\binom{d}{i}(\theta^{d-i}-\theta^{q(d-i)})
\mathcal{L}\begin{pmatrix} X&\theta^{iq} \\ d&d(q-1)\end{pmatrix}.$$

\subsubsection*{Linear relations between Carlitz multiple polylogarithms at one} 
The title means that more precisely, we consider evaluations of Carlitz multiple polylogarithms at $(1,\ldots,1)$.
To conclude the introduction of our main results, we review our results in \S \ref{applications-to-MZV-etc}, where the reader can find complete definitions, basic notations and tools. The main result of this part is Theorem D below (see the more precise Theorem \ref{identity-scalar}). In this result we show that the right-hand side of the identity of Theorem C is a $K$-linear combination of Carlitz multiple polylogarithms at one, that is, series as in 
(\ref{definition-L}) with $m_1=\cdots=m_r=0$. It is possible to write this expansion in a completely explicit way. To simplify notations, write $$\mathcal{L}\begin{pmatrix} 1& 1 & \cdots & 1\\ n_1& n_2 &\cdots& n_r\end{pmatrix}=
\mathcal{L}\big(n_1, n_2,\ldots,n_r\big)=\mathcal{L}(\mathfrak{n}),$$ so that we can stress the dependence on the {\em array} $\mathfrak{n}=(n_1,\ldots,n_r)$ (multi-index with positive integers as entries).

\medskip

\noindent{\bf Theorem D.} {\em For any choice of $d,r\geq 1$ Theorem C provides us with a non trivial expansion of 
$\mathcal{L}(dq^r)=\mathcal{L}(d)^{q^r}$ as a linear combination of Carlitz multiple polylogarithms at one of higher depths: it exhibits arrays $\mathfrak{k}_{\underline{h}}$ of weight $dq^r$ and elements $\kappa_{\underline{h}}\in K$,  with $\underline{h}\in\{0,\ldots,d-1\}^r$, such that $$\mathcal{L}\big(dq^r\big)=\sum_{\underline{h}}\kappa_{\underline{h}}
\mathcal{L}\big(\mathfrak{k}_{\underline{h}}\big).$$}

\medskip

Note that there are several independent linear relations as in the above statement. The main point is that our non-commutative factorization selects one of them. 
The statement will be clarified in the paper.
Theorem D is a consequence of Theorem \ref{identity-scalar}, which describes the arrays more precisely using the following ideas. Carlitz multiple polylogarithms at $(1,\ldots,1)$ are series involving multiple sums as in \S \ref{applications-to-MZV-etc}. Analogously, multiple zeta values of Thakur are series involving multiple power sums defined in \S \ref{multiple-power-sums}. Universal identities of such finite sums imply identities for Carlitz multiple polylogarithms at $(1,\ldots,1)$, or multiple zeta values of Thakur. When this occurs, we say that the identities occur at the level of finite sums. For instance, the results in \cite{NGO} all occur at the level of finite sums, and Carlitz's identities (\ref{carlitz-formulas}) do not occur at the level of finite sums. In particular, to get Theorem D it suffices to take a limit in our identities of finite sums. The construction of the arrays $\mathfrak{k}_{\underline{h}}$ is completely explicit but the complete description is postponed to \S \ref{applications-to-MZV-etc}, where the tools that we sketch now will be described. It involves the {\em triangle product} (see \cite{IKNLNDHP,NDNCP}) of arrays which are {\em concatenations} of powers of the array $(q-1)$ of depth one for the {\em stuffle product}, intertwined with other arrays of depth one. These identities can be also transposed into identities for multiple power sums and Thakur's multiple zeta values thanks to the work of Ngo Dac \cite{NGO}, but they are better suited for the formalism of Carlitz multiple polylogarithms at $(1,\ldots,1)$.

\subsection{Motivic viewpoint} We describe a motivic interpretation of our formulas. In particular, we place them in the context of Drinfeld module ``cycle integration" developed by Gekeler, Deligne, Anderson and others (see \cite{GEK}, \cite{GOS94} and \cite{YU}). Framed in this context, we interpret our formulas as providing an analogue of the integral representation for multiple polylogarithms and the resulting integral shuffle relations, restricted to $q^n$-th powers.

We first briefly explain the classical theory. Cycle integration was originally conceived as a way of comparing singular cohomology (also called Betti cohomology) with de Rham cohomology, and as a means of producing periods from these two cohomology theories. The simplest example of this is the cohomology of the punctured complex plane. The first singular homology group is equal to the fundamental group of the punctured plane, which is generated by a single counterclockwise loop about the origin. The first de Rham cohomology group is generated by the differential form $\frac{dt}{t}$. Then the pairing between the two cohomology theories, which recovers the periods of the punctured plane, is given by cycle integration. For example, if we let $\gamma$ denote the counterclockwise, unit-circular path about the origin, then
\[\int_{\gamma} \frac{dt}{t} = 2\pi i.\]

Similarly, other meaningful arithmetic values may be identified as periods arising from cycle integration (see for example \cite{KON&ZAG}). The periods important to this work are values of Carlitz multiple polylogarithms. For our purposes here, we simplify to the iterated integral representation of the single polylogarithm (we refer the reader to \cite{BUR&FRE} for a full account on multiple polylogarithms). For $0<t<1$ we have
\[\sum_{n_1\geq 1} \frac{t^{n_1}}{n_1^{s_1}} = \int_{\Delta(t)} \frac{dt_1}{1-t_1}\cdots \frac{dt_{s-1}}{1-t_{s-1}}\frac{dt_s}{t_s},\]
where $\Delta(t)$ is the simplex $\{(t_1,\dots,t_s)\in \R^s \mid t\geq t_1 \geq t_2 \geq \dots \geq t_s \geq 0\}.$ This formula has a motivic interpretation by viewing the differentials $\frac{dt_i}{1-t_i}$ and $\frac{dt_j}{t_j}$ as being elements of the de Rham cohomology of the twice punctured plane $\PP^1\setminus \{0,1,\infty\}$ and viewing the simplex $\Delta(t)$ as an element of the \textit{relative} homology group of the twice punctured plane (mod a specified subgroup). See \cite[Examples 3.346 and 3.348]{BUR&FRE} for a more thorough analysis.

One key property that is derived from these iterated integral expressions is the integral shuffle relations for multiple polylogarithms. Namely, if we multiply two such iterated integrals together, one can ``shuffle", or rearrange, the simplices that define the area of integration to obtain depth-preserving $\ZZ$-linear relations between multiple polylogarithms (see for example \cite[Example 1.118]{BUR&FRE} and the ensuing discussion). We propose that the formulas of this paper provide a characteristic $p$ analogue of such relations. Of course, other structures emerge beyond integration - which does not seem to be appropriate in this characteristic $p>0$ setting - and our formulas are not proved using integration, rather they arise as a pairing between cohomology theories which mimics cycle integration.

We now wish to make a comparison between the characteristic 0 theory described above and the new positive characteristic function field formulas contained in our paper. There is a heavy amount of theory in the area of cohomology, periods and cycle integration of Drinfeld modules, so in our limited exposition here we will give a light overview following the account of Hartl and Juschka \cite{HAR&JUS}. The first definitions of de Rham and Betti cohomology (these are more appropriately called de Rham and Betti modules - since they are not computed using chains complexes - but we will use the common notation nonetheless) and the period pairing were given by Anderson \cite{AND} and Gekeler \cite{GEK}. Given a Drinfeld module $\phi$ with period lattice $\Lambda$, define the Betti homology of $\phi$ to be
\[H_{\text{B},1} = \Lambda.\]
The definition of the de Rham cohomology of $\phi$ is slightly more involved. First we define the space of biderivations of $\phi$, denoted  $D(\phi)$, which are certain families in $\C_\infty[\tau]\tau$ which satisfy a natural difference equation related to $\phi$ (see \cite[5.43]{HAR&JUS} for a succinct summary). Next, we identify a subspace of ``strictly inner" biderivations, denoted $D_{\text{si}}(\phi)$. The de Rham cohomology of $\phi$ is then defined to be
\[H_{\text{dR}}^1(\phi) = D(\phi)/D_{\text{si}}(\phi),\]
which should remind the reader of the classical definition of de Rham cohomology given by smooth differential forms quotiented by exact forms.

The pairing between Betti homology and de Rham cohomology provides an analogue of cycle integration for Drinfeld modules, and recovers the periods and quasi-periods of a Drinfeld module. Given $\eta \in H_{\text{dR}}^1(\phi)$ and $\gamma \in H_{\text{B},1}$, we define a pairing
\[\int_{\gamma}\eta = -\sum_{n=0}^\infty \eta\left (\exp_\phi\left (\frac{\gamma}{\theta^{n+1}}\right )\right ) \theta^n,\]
where $\exp_\phi$ is the exponential function associated to $\phi$ (see \S \ref{section-carlitz-tensor}). Readers familiar with the theory of Anderson's generating functions will recognize that this pairing consists of a difference operator applied to an Anderson generating function, which is then evaluated at $\theta$ (see \cite{NAM&PAP} and \cite{BCPW}).

In order to apply this theory to the formulas in our paper, we must transport the above constructions to the setting of Anderson's $t$-motives. We comment that Brownawell, Chang, Papanikolas and Wei study a similar theory extensively in the case of dual $t$-motives in the paper \cite{BCPW}; the counterpart of this theory for $t$-motives is not as well studied. For our treatment we follow the definitions and theory of \cite[Sec. 2.3.5]{HAR&JUS}. 
We define the Tate algebra in the variable $t$ to be
\begin{equation}\label{D:Tate algebra}
\TT = \biggl\{ \sum_{i=0}^\infty b_i t^i \in \C_\infty[[t]] : \big\lvert b_i \big\rvert \to 0 \biggr\},
\end{equation}
where we recall that $\mid \cdot \mid$ is the norm of $\CC_\infty$. The Tate algebra admits an action of the $q$th-power Frobenius operator, which we denote by $\tau$, that acts $\FF_q[t]$-linearly. Let $M$ be the $t$-motive associated to an abelian $t$-module $\phi$, which is isomorphic to $\C_\infty[t]^r$, for some $r\in \mathbb Z_+$, as a $\C_\infty[t]$-module, but which also carries a (non-standard) action of the Frobenius, which we denote by $\tau$ (see Definition \ref{D:t-motive}). Thus, $(M\otimes_{\C_\infty[t]} \TT)$ admits a diagonal action of the Frobenius, which we also denote by $\tau$. Define the Betti cohomology of $M$ with coefficients in $\C_\infty$ as the set of $\tau$-invariant elements,
\[H_{\text{B}}^1(M) = (M\otimes_{\C_\infty[t]} \TT)^{\tau} = \{ m \in M\otimes_{\C_\infty[t]} \TT \mid \tau m = m\}.\]
We then define the de Rham cohomology with coefficients in $\C_\infty$ as
\[H_{\text{dR}}^1(M) = M/(t-\theta)M.\]
We define the comparison isomorphism from Betti to de Rham cohomology by first defining the natural multiplication map
\[\alpha: (M\otimes_{\C_\infty[t]} \TT)^{\tau} \otimes_{\F_q[t]} \TT \to M\otimes_{\C_\infty[t]} \TT,\]
defined on simple tensors for $\lambda\in (M\otimes_{\C_\infty[t]} \TT)^{\tau}$ and $z\in \TT$ by
\[\lambda \otimes z \mapsto z \cdot \lambda.\]
Finally, we define the comparison map
\[h_{\text{B},\text{dR}}: H_{\text{B}}^1(M) \to H_{\text{dR}}^1(M)\]
to be the map
\[h_{\text{B},\text{dR}}(\lambda) = \alpha(\lambda)\twist \mod{(t-\theta)},\]
where if $f=\sum_if_it^i\in\TT$, $f^{(1)}=\tau(f)=\sum_if_i^qt^i\in\TT$ (see \S \ref{D:Notation}). In particular, when $M$ is a rank 1 $t$-motive (which is the case we consider in this paper), then we have $M\isom \C_\infty[t]$ as $\C_\infty[t]$-modules, and so $M\otimes_{\C_\infty[t]} \TT \isom \TT$. In this case $M/(t-\theta)M\isom \C_\infty$ and the map $h_{\text{B},\text{dR}}$ is given simply by
\[h_{\text{B},\text{dR}}(\lambda) = \lambda\twist(\theta).\]

We now restrict our attention to the case where $M$ is the $t$-motive associated with the $d$-th tensor power of the Carlitz module, which is the setting of our main theorems. We give an alternate expression for the comparison isomorphism which connects with the formulas of our paper.

Let $\Omega = 1/\omega\twist \in \TT$, defined in \eqref{anderson-Thakur}, which satisfies 
\begin{equation}\label{omega-func-eq}
(t-\theta)^d (\Omega)^d = (\Omega\twistinv)^d
\end{equation}
(see \cite[\S 3.3.5]{PAP0}), and so $\tau((\Omega\twistinv)^d) = (\Omega\twistinv)^d$, which implies $(\Omega\twistinv)^d\in H_{\text{B}}^1(M)$. In fact, a quick calculation shows that that it is a generator, $H_{\text{B}}^1(M) = (\Omega\twistinv)^d \F_q[t]$. In \cite[Definition 2.14]{GRE} the first author defined certain motivic maps $\delta_0^M$ and $\delta_{1,\bz}^M$, which we describe in \eqref{D:delta_0^M} and \eqref{D:delta_1^M}. We now give an equivalent definition of the comparison isomorphism, which is proven in Lemma \ref{L:delta_0 alt}, namely
\[h_{\text{B},\text{dR}}(\lambda) = \left (\delta_0^M\left (\frac{1}{t-\theta}\lambda\right )\right )_d,\]
where $(\cdot)_d$ denotes projection onto the $d$-th coordinate. Since we have a canonical isomorphism of $\C_\infty[t]$-modules $$H^1_{\text{dR}}(M) = M/(t-\theta)M \isom \frac{1}{t-\theta}M/M,$$ we view the above expression as a pairing between the element $\frac{1}{t-\theta}\in H^1_{\text{dR}}(M)$ and $\lambda \in H_{\text{B}}^1(M)$.

For example, for fixed $d$, we have
\[h_{\text{B},\text{dR}}((\Omega\twistinv)^d) = \left (\delta_0^M\left (\frac{1}{t-\theta}(\Omega\twistinv)^d\right )\right )_d = -\frac{1}{\tpi^d}.\]
which agrees with \cite[Example 2.3.39]{HAR&JUS} for $d=1$ and with \cite[Proposition 5.2.3]{BCPW} for the case of dual $t$-motives.

In Section \ref{S:Motivic Identities} we show that the formulas of our paper occur in the motivic setting, meaning that they can be seen as families of identities between elements in $M\otimes_{\C_\infty[t]} \TT$. We also show that we can recover the $\C_\infty$-valued identities described in Theorems A-D using the cycle integration map described above. Altogether, this places our formulas in a familiar context similar to Brown's motivic theory (see \cite{BUR&FRE}). We summarize the basic ideas of section \ref{S:Motivic Identities} in the following theorem.

\medskip

\textbf{Theorem E} \textit{
For arbitrary $d\geq 1$ and for $k\geq 0$, we have the following:
\begin{enumerate}
\item The left hand side of Theorem C can be realized as an element of $M\otimes_{\C_\infty[t]} \TT$ evaluated under the $\delta_{1,\bz}^M$ map. This is Proposition \ref{P:delta1 of Omega}.
\item The right hand side of Theorem C can be realized as an element of $M\otimes_{\C_\infty[t]} \TT$ evaluated under the $\delta_{1,\bz}^M$ map. This is Proposition \ref{P:delta1 of L product}.
\item The identity given in Therem C can be seen as an equality of elements in $M\otimes_{\C_\infty[t]} \TT$.  This is \eqref{E:I_i Q_m expression} and \eqref{E:I_i decomposition}, and see particularly the diagram in Remark \ref{R:diagram}.
\end{enumerate}
}

Thus, our new theorems can be viewed as giving a $q$-power analogue of the integral representation of multiple polylogarithms, arising from an extension of cycle integration, where we extend from evaluating under $\delta_0^M$ to $\delta_{1,\bz}^M$ (in fact, $\delta_0^M$ can be considered a special case of $\delta_{1,\bz}^M$, where we set $\tau=0$). Whether these formulas can be described using cycles from a relative homology class as in the characteristic $0$ case, and how the map $\delta_{1,\bz}^M$ fits into the broader context of cycle integration are natural questions. In \cite{GRE} and \cite{GEZ&GRE} the first author uses the map $\delta_{1,\bz}^M$ as an analogue of integration to construct an algebraic version of the Mellin transform related to zeta and $L$-functions. Thus it seems natural this map would appear in our analogue of cycle integration. These questions are a topic for future work.

The fact that our new formulas arise from this analogue of cycle integration also leads to a natural question: Can the $K$-linear families of relations we generate be viewed as a $q$-analogue of the classical integral shuffle relations? On the one hand, our new formulas do not describe a shuffle product structure. However, they do describe an action of Frobenius on certain polylogarithms. In future work, the authors plan to explore the Frobenius structure more fully. Current numerical computations seem to indicate that the formulas presented in this work are a class of examples of a more general, associative Frobenius structure that may exist on a certain ``Hoffman algebra''. The authors hope to answer these questions definitively in their future work. Additionally, the surprising results in \cite{IKND26} on the existence of $\F_q$-linear relations among Carlitz multiple polylogarithms and multiple zeta values suggest that there might be an additional $\F_q$-shuffle product structure on these spaces. It would be interesting to see how the ideas of this work interact with the new relations discovered in \cite{IKND26}.


\section{Some background}

We begin by presenting the necessary tools to work with the tensor powers $C^{\otimes d}$ of Carlitz's module $C$,
especially its exponential and logarithm functions $\exp_{C^{\otimes d}},\log_{C^{\otimes d}}$.
Then we present some unpublished formulas due to Papanikolas allowing to compute its matrix coefficients.  We will then have at our disposal a toolbox of matrices and relations that will allow us to state and prove our main results.

\subsection{Main notations and conventions}

In all the following:
\begin{itemize}\label{D:Notation}
\item[-] $\NN$ denotes the set of natural integers and $\NN^*$ denotes the subset of positive integers.
\item[-] There is a common notation for all the multiplicative neutral elements of all rings (commutative and non-commutative) considered in this paper. In particular we write $1$ for all identity matrices.
\item[-] $\FF_q$ denotes a finite field with $q$ elements and characteristic $p$.
\item[-] $A=\FF_q[\theta]$ denotes the $\FF_q$-algebra of polynomials in the indeterminate $\theta$ with coefficients in $\FF_q$.
\item[-] $K=\FF_q(\theta)$ denotes the fraction field of $A$.
\item[-] $K_\infty$ denotes the completion of $K$ at the infinity place.
\item[-] $\CC_\infty$ denotes the completion of a separable closure of $K_\infty$.
\item[-] $|\cdot|$ is a fixed norm over $\CC_\infty$ such that $|\theta|>1$.
\item[-] We write $\tau(b)=b^q$ for $b\in \CC_\infty$. At the same time we denote by $\tau$ the $\FF_q$-linear endomorphism
$b\mapsto\tau(b)$. 
\item[-] If $B$ is an $\FF_q$-algebra endowed with an $\FF_q$-linear endomorphism $\tau$, we denote by $B[\tau]$ (resp. $B[[\tau]]$) the skew ring of finite linear combinations with coefficients in $B$ of powers of $\tau$
(resp. the skew ring of formal series in powers of $\tau$ with coefficients in $B$)  with multiplication rule determined by $\tau b=\tau(b)\tau$. We often write $f^{(1)}$ instead of $\tau(f)$. More generally, we write $f^{(k)}$ for $\tau^k(f)$. The notations $\tau^k(f),f^{(k)}$ are extended coefficientwise on matrices with entries in $B$.
\item[-] If $B$ is any ring, $B^{r\times s}$ denotes the set of matrices with $r$ lines and $s$ columns, with its $rs$ entries in $B$.
If $r=s$ we denote by $1$ the identity matrix of $B^{r\times r}$.
\item[-] In this paper $x,x',y,z,t\ldots$ denote independent indeterminates or variables. In general
$t$ is reserved for a central variable, that is such that $\tau t=t\tau$. $x,y,z,\ldots$ are not central, unless otherwise specified.
\item[-] $(\cdot)^\top$ denotes matrix transposition and $(\cdot)^\perp$ antitransposition (see \S \ref{Papanikolas-matrices}).
\item[-] $[m]=\theta^{q^m}-\theta\in A$ if $m>0$. If $m=0$ we set $[0]=1$.
\item[-] Empty products are by convention set to one and empty sums are set to zero.
\item[-] $\binom{i}{j}$ the reduction modulo $p$ of the binomial coefficient when it is well defined.
\item[-] $\TT$ is the Tate algebra completion of $\CC_\infty[t]$ for the Gauss norm.
\end{itemize}

\subsection{The tensor powers of Carlitz's module}\label{section-carlitz-tensor}

We recall the definition of $C^{\otimes d}$ from \eqref{E:C otimes def}. There exists a unique exact sequence of $A$-modules
\begin{equation}\label{exact-sequence}
0\rightarrow \operatorname{Ker}(\exp_{C^{\otimes d}})\hookrightarrow\operatorname{Lie}(C^{\otimes d})(\CC_\infty)\xrightarrow{\exp_{C^{\otimes d}}}C^{\otimes d}(\CC_\infty)\rightarrow0
\end{equation}
that we explain now. (1)
$\operatorname{Lie}(C^{\otimes d})$ is the Lie functor associated to $C^{\otimes d}$, uniquely defined, for $B$ an $A$-algebra, by the multiplication by $\theta$ which is the left multiplication by the endomorphism 
\begin{equation}\label{algebraic-definition}
\dern_\theta(\theta):=\theta\cdot 1+N=\theta+N\in\operatorname{End}_K\big(\GG_a^d(K)\big),\end{equation}
where $1$ is the identity matrix of $K^{d\times d}$ and $N$ is defined in (\ref{E:C otimes def}).
Then, the multiplication by a general element $a\in A$ for the module structure of $\operatorname{Lie}(C^{\otimes d})$ is given by the 
standard evaluation endomorphism given by the polynomial $a(x)$ with the variable $x$ replaced by the matrix $\theta\cdot 1+N=\theta+N\in K^{d\times d}$:
$\dern_\theta(a)=a(\theta+N).$

In this paper we often have families of objects $(\leveln{\operatorname{Obj}})_d$ depending on $d\in\NN$, but we mainly study them for fixed choices of $d$. If we want to stress the dependence on $d$ we use the notation $\leveln{\operatorname{Obj}}$. If not, we simply write $\operatorname{Obj}$. For example, we may write 
$\leveln{N}$ instead of $N$ with the purpose of giving importance to the dependence in $d\in\NN^*$.

(2) The map $\exp_{C^{\otimes d}}$ is entire and surjective, and is represented by a formal series (\footnote{The formal exponential in the terminology of \cite{AND&THA}; in our paper, we often identify exponential functions etc. with the corresponding formal counterparts.})
$$\exp_{C^{\otimes d}}=\sum_{i\geq 0}Q_i\tau^i\in\operatorname{End}_K\big(\GG_a^d(K)\big)[[\tau]],$$
for a family of invertible matrices $Q_i\in\operatorname{GL}_d(K)$ with $Q_0=1$ (this normalization makes 
$\exp_{C^{\otimes d}}$ to be the uniquely determined by (\ref{exact-sequence}), to check invertibility use Proposition \ref{proposition-Papanikolas-PQ}). The reason for which these are automorphisms is explained in \S \ref{papanikolas-formulas}.

(3) The kernel of $\exp_{C^{\otimes d}}$ can be computed explicitly. There exists
$$\Pi=\Pi_d\in K_\infty[(-\theta)^{\frac{1}{q-1}}]^{d\times 1}\setminus\{0\}$$ such that $\operatorname{Ker}(\exp_{C^{\otimes d}})$
is the free rank-one submodule of $\operatorname{Lie}(C^{\otimes d})(\CC_\infty)$ generated by $\Pi$ (see \cite[\S 2.5]{AND&THA}), with the last entry of $\Pi$ in $\FF_q^\times\widetilde{\pi}^d$. If $d$ is prime to the characteristic $p$ of $\FF_q$, the $d$ entries of $\Pi$ are known to be algebraically independent over $K$ (see Maurischat's \cite[Theorem 8.1]{MAU}; Corollary 8.5 ibid. allows to compute the transcendence degree in the general case) (\footnote{We are thankful to Andreas Maurischat for bringing his work to our attention.}). We already fixed a fundamental period $\widetilde{\pi}$ of Carlitz's module, in (\ref{def-pi}). 
From now on, we choose $\Pi$ to be a generator of $\operatorname{Ker}(\exp_{C^{\otimes d}})$ such that 
$$\Pi=\begin{pmatrix}\widetilde{\pi}_{d-1}\\ \vdots \\ \widetilde{\pi}_1\\ \widetilde{\pi}_0\end{pmatrix}=\begin{pmatrix}*\\ \vdots \\ *\\ \widetilde{\pi}^d\end{pmatrix}\in K_\infty[(-\theta)^{\frac{1}{q-1}}]^{d\times 1}.$$
The matrix $\widehat{\Pi}$ introduced in (\ref{hat-pi}) can also be computed easily by using \cite[(3)]{MAU2}:
\begin{equation}\label{maurischat-formula}
\dern_x\big(\tau(\omega)(x)\big)^d_{x=\theta}=(-1)^d\widehat{\Pi},
\end{equation}
($\dern_x$ is defined in the next section) where $\omega$ is {\em Anderson-Thakur function}
\begin{equation}\label{anderson-Thakur}
\omega(x):=(-\theta)^{\frac{1}{q-1}}\prod_{i\geq 0}\Big(1-\frac{x}{\theta^{q^i}}\Big)^{-1},
\end{equation}
with the element $(-\theta)^{\frac{1}{q-1}}$ chosen so that the residue of $\omega$ at $t=\theta$
equals $-\widetilde{\pi}$,
which defines a rigid analytic function $D^\circ_{\CC_\infty}(0,|\theta|)\rightarrow\CC_\infty$, where
$$D^\circ_{\CC_\infty}(0,|\theta|)=\{z\in\CC_\infty:|z|<|\theta|\}$$ and $\tau$ is the $\FF_q(x)$-linear extension of
$c\mapsto c^q$.
An alternative proof of (\ref{maurischat-formula}) is proposed in \cite[Proof of Lemma 4.6]{PEL6}.

\subsection{$d$-matrices, $\partial$-matrices, anti-transposition}
\label{Papanikolas-matrices}

For an $\FF_q$-algebra $B$, the centralizer $\mathcal{Z}(B)$ of 
$N$ in $\operatorname{End}(\GG_a^d(B))$ can be identified with $\GG_a^d(B)$ via the projection
$[\cdot]_d$ on the last column in the canonical basis, determining a structure of $A$-algebra over $\GG_a^d(B)$. We denote by $h$ the inverse of the isomorphism 
$\mathcal{Z}\xrightarrow{[\cdot]_d}\GG_a^d$. Alternatively, if $Z\in\GG_a^d(B)$, we write $\widehat{Z}$
instead of $h(Z)$. Explicitly,
$$Z=:\begin{pmatrix}z_{d-1}\\ \vdots \\ z_1\\ z_0\end{pmatrix},\quad \widehat{Z}:=
\begin{pmatrix}z_{0} & z_1 & \cdots & z_{d-1}\\ 0 & z_0 & \cdots & z_{d-2}\\
\vdots & \vdots & & \vdots \\
0 & 0 & \cdots & z_0\end{pmatrix}.$$
We also note that the elementary commutation rule holds
\begin{equation}\label{commutation}
\widehat{Z}U=\widehat{U}Z,\quad U,Z\in\GG_a^d(B).
\end{equation}

We review from Papanikolas' \cite{PAP} other tools:  the notions of $d$-matrix and $\partial$-matrix, and the operation of anti-transposition, as well as certain formulas for the coefficients of the inverse series 
$$\log_{C^{\otimes d}}:=\exp_{C^{\otimes d}}^{-1}.$$ This reference being in this moment an unpublished monograph, we recall all we need here, with slightly modified notations that are suitable for our purposes. With $F$ any field containing $\FF_q$ and $t$ an indeterminate over $F$, we have the injective ring homomorphism $$F(t)\xrightarrow{\dern_t} \mathcal{Z}(F(t))$$ defined by the truncation at the order $d$ of Taylor series expansions 
in the variable $t$ in the following way:
$$f\mapsto \sum_{i\geq 0}\mathcal{D}_{t,i}(f)N^i\in\mathcal{Z}\big(F(t)\big),$$
where we set $N^0=1$ and we recall that $\mathcal{D}_{t,i}(f)$ denotes the $i$-th higher divided derivatives of $f$ in the variable $t$. A {\em $\dern_t$-matrix} (or, more simply, a {\em $d$-matrix}) is any matrix contained 
in the image of this homomorphism. Note that this definition agrees with (\ref{algebraic-definition}) 
for $f\in\FF_q[t]$.

Papanikolas also defines {\em $\partialn_t$-matrices}, 
or more simply, {\em $\partial$-matrices}. These are variants of Wronskian matrices. Given an $s$-tuple $(f_1,\ldots,f_s)\in F(t)^{1\times s}$, we set
$$\partialn_t(f_1,\ldots,f_s):=\begin{pmatrix} \mathcal{D}_{t,d-1}(f_1) & \mathcal{D}_{t,d-1}(f_2) & \cdots & \mathcal{D}_{t,d-1}(f_s) \\
\mathcal{D}_{t,d-2}(f_1) & \mathcal{D}_{t,d-2}(f_2) & \cdots & \mathcal{D}_{t,d-2}(f_s) \\
\vdots & \vdots &  & \vdots \\
\mathcal{D}_{t,1}(f_1) & \mathcal{D}_{t,1}(f_2) & \cdots & \mathcal{D}_{t,1}(f_s) \\
f_1 & f_2 & \cdots & f_s
 \end{pmatrix}\in F(t)^{d\times s}.$$
Papanikolas noticed that for $g\in F(t)$,
\begin{equation}\label{leibnitz}
\partialn_t(gf_1,\ldots,gf_s)=\dern_t(g)\partialn_t(f_1,\ldots,f_s).
\end{equation}
In other words, the map $$F(t)\xrightarrow{\partialn_t}F(t)^{d\times 1}$$
which sends $f\in F(t)$ to the column matrix $\partialn_t(f)$ (one column) 
is left $F(t)$-linear via $\dern_t$. This follows directly from (\ref{commutation}) and Leibnitz's rule.

Here and in all the following, $(\cdot)^\top$ denotes the matrix transposition. 
{\em Anti-transposition} is another fundamental tool considered in Papanikolas' work \cite{PAP}. to define it, it is practical to 
first introduce the {\em reverse} $(f_1,\ldots,f_s)^{\text{R}}$ of a row sequence $(f_1,\ldots,f_s)$ of objects. This is just 
the row sequence $(f_s,\ldots,f_1)$. Now, given a matrix with entries in a ring $R$,
$$M=(M^1,\ldots,M^d)\in R^{d\times s}$$
so that $M^j$ is the $j$-th column of $M$ (not the $j$-th power),
the anti-transpose $M^\perp$ is defined by 
$$M^\perp=\Big(\Big((M^1,\ldots,M^d)^{\text{R}}\Big)^\top\Big)^{\text{R}}\in R^{s\times d}$$ (there are several alternative definitions
of this operation on matrices). 
We also have $(M^\perp)^\perp=M$ (\footnote{In the case $s=d$ the anti-transposition is the adjunction operator for the symmetric bilinear form which associates to a couple $(f,g)$ of elements of $R^{\oplus d}$ the standard scalar product of $f$ and the reverse of $g$.}). Given matrices $M_1,M_2$ such that $M_1M_2$ is well defined,
$$(M_1M_2)^\perp=M_2^\perp M_1^\perp.$$

\subsubsection{Some formulas due to Papanikolas}\label{papanikolas-formulas}
We describe Papanikolas explicit formulas allowing to compute explicitly the coefficients $Q_i\in\operatorname{End}_K(\GG_a^d(K))$ of $\exp_{C^{\otimes d}}=\sum_iQ_i\tau^i$
as well as the coefficients $P_j\in\operatorname{End}_K(\GG_a^d)$ of its inverse formal series 
$$\log_{C^{\otimes d}}=\exp_{C^{\otimes d}}^{-1}=\sum_jP_j\tau^j\in\operatorname{End}_K\big(\GG_a^d(K)\big)[[\tau]],$$ the associated {\em logarithm}
(uniquely defined by $\exp_{C^{\otimes d}}\log_{C^{\otimes d}}=\log_{C^{\otimes d}}\exp_{C^{\otimes d}}=1$). To describe these formulas we introduce a few more notations. Consider indeterminates $x,y,t$.

We set
$$H_{x,y}:=\partialn_x\big( (x-y)^{d-1},\ldots, x-y,1\big)\in\operatorname{GL}_d(F(x,y)).$$
($H_{x,y}$ is lower triangular with 1's along the diagonal - hence the invertibility). For elements $a,b\in F$, we also set
$$H_{a,b}:=\Big(H_{x,y}\Big)_{\begin{smallmatrix}x=a\\ y=b\end{smallmatrix}}\in\operatorname{GL}_d(F).$$
Since $H_{x,y}$ is the idempotent matrix representing the identity endomorphism over the vector space $F(x,y)^{d\times 1}$
with respect to the bases $(\partial_x(x^i))_{0\leq i\leq d-1}$ (target) and $(\partial_y(y^i))_{0\leq i\leq d-1}$ (source), one sees easily that, given any $a,b,c\in F$, 
\begin{equation}\label{inverse-H}
\big(H_{a,b}\big)^{-1}=H_{b,a}\end{equation}
and
\begin{equation}\label{groupoidH}
H_{a,b}H_{b,c}=H_{a,c}.
\end{equation}
Note also the identity
$$\big(H_{a,b}\big)^{(k)}=H_{a^{(k)},b^{(k)}},\quad k\geq0$$
when $B$ is an $\FF_q$-algebra (we can even choose $k\in\ZZ$ if $\tau$ is invertible on $B$). This is used often in the paper.
These properties have been observed in \cite{PEL6} (\footnote{Another way to check these identities is to first proving them in the case $d=2$, and 
noticing that for general $d$ (here we use the notation $\leveln{(\cdot)}$ because we consider several values for $d$), $\Hthree{d}{x}{y}$ is a symmetric power:
$$\Hthree{d}{x}{y}=\operatorname{Sym}^{(d-1)}(\Hthree{2}{x}{y}).$$}).

We also define, with $x,t$ independent indeterminates, the matrices
$$\dPn_k(x,t):=\dern_x\Big((x-t)\cdots(x-t^{q^{k-1}})\Big)^{-d}\in\mathcal{Z}(F(x,t)),$$
and when it is well defined, for $a,b\in F$, we set
$$\dPn_k(a,b):=\dPn_k(x,t)_{\begin{smallmatrix}x=a\\
t=b\end{smallmatrix}}.$$
Then we have:
\begin{Proposition}[Papanikolas]\label{proposition-Papanikolas-PQ}
\begin{multline*}
Q_i=\left((H_{x,z})^\perp\dPn_i(x,t)\right)_{\begin{smallmatrix}x=\theta^{q^i}\\
z=\theta\\
t=\theta\end{smallmatrix}}=(\Htwo{\theta^{q^i}}{\theta})^\perp\dPn_i(\theta^{q^i},\theta),\\ P_i=\left(\dPn_i(x,t)\Htwo{x}{z}\right)_{\begin{smallmatrix}x=\theta\\
z=\theta^{q^i}\\
t=\theta^q\end{smallmatrix}}=\dPn_i(\theta,\theta^q)\Htwo{\theta}{\theta^{q^i}}.
\end{multline*}
\end{Proposition}
 See \cite[Proposition 4.3.6]{PAP}. Papanikolas' proof of the second formula is also reproduced in  \cite[\S 7.2]{PEL6}).
We are going to use the matrices $\dPn_i(\theta,\theta^q)$ often, so that we introduce an ad hoc notation for them:
\begin{equation}\label{definition-of-Gamma-k}
\dPsn{i}:=\dPn_i(\theta,\theta^q)=\dern_x\Big((x-\theta^q)\cdots(x-\theta^{q^i})\Big)^{-d}_{x=\theta}\in\operatorname{GL}_d(K).
\end{equation}
In particular,
\begin{equation}\label{compact-formula-Pi}
P_i=\dPsn{i}\Htwo{\theta}{\theta^{q^i}}\in\operatorname{GL}_d(K),\quad i\geq 0.
\end{equation}

\section{Carlitz operators, motivic pairings, applications}
The main result in this section establishes the non-commutative factorization of the higher dimensional Carlitz operators described in the introduction, culminating in Proposition \ref{theorem-factor}. This factorization is the main new innovation of this paper and leads directly to the applications on relations between polylogarithms discussed in the introduction. In order to develop the factorization, we study the Carlitz operators (see \S \ref{section-carlitz-operators} below)  from the point of view of $t$-motives. In Section \ref{section-carlitz-operators} we define and study the basic properties of these operators. Then in Section \ref{SS:Motivic Pairings} we review useful properties of $t$- and dual $t$-motives and the theory developed by the first author in \cite{GRE}. Then in Section \ref{SS:Noncom factorization} we use this motivic theory to prove the non-commutative factorization of these higher Carlitz operators. Finally, in Section \ref{section-non-commutative} we modify these operators and their factorization such that they give formulas in $\operatorname{End}_K(\GG_a^d(K))[\tau]$ rather than operators in $\operatorname{End}_K(\operatorname{End}_{K}(\GG_a^d(K)))[\tau]$ as in \cite{PEL5}. 

\subsection{Carlitz's operators}\label{section-carlitz-operators} We review Papanikolas generalization of Carlitz's polynomials in \cite{PAP}. The construction mimics in higher dimensions Carlitz's construction \cite[Theorem 4.1.5]{GOS}. References on Carlitz's polynomials (and Carlitz module) are \cite[Chapter 3]{GOS}, see also \cite[\S 4.4.2]{PEL5}. For $W\in\operatorname{End}_{\CC_\infty}(\GG_a^d(\CC_\infty))$ we expand:
$$\exp_{C^{\otimes d}}W\log_{C^{\otimes d}}=\sum_{k\geq0}E_k(W)\tau^k\in\operatorname{End}_{\CC_\infty}(\GG_a^d(\CC_\infty))[[\tau]].$$ The coefficients of this expansion can be explicitly described:
$$E_k(W)=\sum_{i=0}^kQ_iW^{(i)}P_{k-i}^{(i)}.$$ Note that for $d>1$, $E_k$, $\FF_q$-linear, does not represent an endomorphism of $\GG_a^d(\CC_\infty)$ but rather an element of $\operatorname{End}_{\FF_q}(\operatorname{End}_{\CC_\infty}(\GG_a^d(\CC_\infty)))$. In \cite[\S 2]{PEL6} we write, alternatively
$$E_k=\sum_{i=0}^k(Q_i\otimes P_{k-i}^{(i)})\boldsymbol{\tau}^i\in \operatorname{End}_{\CC_\infty}(\GG_a^d(\CC_\infty))[[\boldsymbol{\tau}]],$$ with the commuting rule
$\boldsymbol{\tau}(U\otimes V)=(U^{(1)}\otimes V^{(1)})\boldsymbol{\tau}$. If $d=1$ we recover the classical definition of Carlitz's polynomials. If $z$ varies in $\CC_\infty$, $E_k(z)$ behaves as
the evaluation of an $\FF_q$-linear polynomial of degree $q^k$ with kernel $A(<k)$, the $\FF_q$-vector space of the elements of $A$ that have degree in $\theta$ which is $<k$.
Recall that Carlitz (see \cite[Proposition 4.4.8 (3)]{PEL5} and \cite[\S 3.2]{GOS}) proved that, in the case $d=1$ and using our notations,
$$\lim_{k\rightarrow\infty}P_k^{-1}E_k=\sin_A.$$ In this simpler situation, 
$P_k^{-1}=l_k\in A$, where the sequence $l_k$ has been introduced in (\ref{coeff-log-carlitz}).
This led the second author to study in \cite{PEL6} the normalization
\begin{equation}\label{E:Federico Normalization}
\mathcal{E}_k(W):=P_k^{-1}E_k(W)
\end{equation}
in the general case $d\geq 1$ (by Proposition \ref{proposition-Papanikolas-PQ}, $P_k$ is invertible for all $k$). He proved (see \cite[Theorem A]{PEL6} that the sequence of functions $(\mathcal{E}_k)_{k\geq 0}$, as $k\rightarrow\infty$, tends to the exponential function $\exp_{\widetilde{\boldsymbol{\phi}}}$ of an Anderson $A$-module $\widetilde{\boldsymbol{\phi}}$ of dimension $d^2$ and rank $d$ which is isomorphic to $(C^{\otimes d})^{\oplus d}$ (direct sum of $d$ copies of $C^{\otimes d}$). The structure of this $A$-module is described in \cite[Theorem 5.2]{PEL6}. In \cite[Theorem B]{PEL6} it is proved that the restriction
$\exp_{\widetilde{\boldsymbol{\phi}}}|_{\mathcal{Z}(\CC_\infty)}$ of $\exp_{\widetilde{\boldsymbol{\phi}}}$ to $\mathcal{Z}(\CC_\infty)$ has a non-commutative factorization extending (\ref{non-comm-fact-sine}) to the case $d\geq1$. The factors in the factorization are of the form
$$\widetilde{\gamma}_k+\widetilde{\delta}_k\boldsymbol{\tau}$$
with $\widetilde{\gamma}_k,\widetilde{\delta}_k$ explicit sequences of elements of $\operatorname{End}_K(\operatorname{End}_{K}(\GG_a^d))$.
 
We now present another type of normalization of the operators $E_k$ that give different properties, and radically simpler factorizations, with factors of the type 
$$1-\mathcal{L}_k\tau,$$ as in the introduction. They will give us our Theorems A and B.

\subsection{Motivic pairings}\label{SS:Motivic Pairings}

In this section we will describe the extensions of the Carlitz operators $E_k$ described above as combinations of evaluations of elements of $t$-motives and dual $t$-motives. This is similar to what was done by the first author in \cite{GRE}. The advantage of this approach is that the setting of $t$-motives has more natural structure --- they are $\C_\infty[t]$-modules with a compatible structure of Frobenius (see definition below). We can see the noncommutative factorizations presented in this paper as arising naturally from the interplay between the $t$ and Frobenius structures.

We briefly recall the definition of the $t$-motive and dual $t$-motive attached to the $d$-th tensor power $C^{\otimes d}$ of the Carlitz module $C$. These objects have similarities so we give the definitions simultaneously. We will not use $t$-motives extensively, so we do not give a full account of their theory here (refer to \cite{BRO&PAP} for a full account), but we do use these specific $t$-motive and dual $t$-motive in our proof of the main theorems in this section.

We denote by $\CC_\infty[t,\tau]$ the non-commutative $\CC_\infty[t]$-algebra generated by formal finite $\CC_\infty[t]$-linear combinations $\sum_ic_i\tau^i$ with the obvious sum and the unique product defined by $\tau c=c^{(1)}\tau$ for $c\in\CC_\infty[t]$ and $\tau^i\tau^j=\tau^{i+j}$ for all $i,j\geq 0$. Note that $\CC_\infty[t,\tau]$ contains 
the $\CC_\infty$-algebra $\CC_\infty[\tau]$.
Similarly, we define $\CC_\infty[t,\sigma]$ by setting $\sigma c=c^{q^{-1}}\sigma$,
for $c\in\CC_\infty[t]$ (it contains $\CC_\infty[\sigma]$).

We choose an integer $d\geq 1$.

\begin{Definition}\label{D:t-motive} 
The {\em $t$-motive $M$} (resp. the {\em dual $t$-motive $N$}) {\em associated to $C^{\otimes d}$} is the left $\C_\infty[t,\tau]$-module (respectively, the left $\C_\infty[t,\sigma]$-module) determined by the free rank one $\CC_\infty[t]$-module $M=N=\C_\infty[t]$ with the left multiplication by $\tau$ (resp. by $\sigma$) given, for $m\in M$ and $n\in N$, by
\[\tau m = (t-\theta)^d m\twist,\quad \sigma n = (t-\theta)^d n\twistinv.\]
We comment that, with suitable restrictions, the category of dual $t$-motives ($t$-motives) is equivalent (antiequivalent) to the category of abelian $t$-modules (see \cite{BRO&PAP}).
\end{Definition}

\begin{Remark}
We comment that with the definition above, we are now using the symbol $\tau$ to mean two things: 1. The $q$-power Frobenius operator, and 2. The $\tau$-action on $t$-motive $M$. Throughout the paper, these two uses never occur simultaneously, so there is little risk of confusion.
\end{Remark}

Both $M$ and $N$ are free $\C_\infty[t]$-modules of rank 1 (by definition) and $M$ is a free $\C_\infty[\tau]$-module of rank $d$, while $N$ is a free $\C_\infty[\sigma]$-module of rank $d$. Indeed both $M$ and $N$ have the  basis
\[\big(1,(t-\theta),\dots,(t-\theta)^{d-1}\big)\]
(for the $\C_\infty[\tau]$-module structure of $M$ and for the $\C_\infty[\sigma]$-module structure of $N$).
In order to maintain consistency with the notation in \cite{GRE} we set $$g_k = (t-\theta)^{k-1},\quad h_k= (t-\theta)^{d-k}$$ for $1\leq k\leq d$ and consider $$\bg = 
\begin{pmatrix} g_1\\ \vdots \\ g_d\end{pmatrix}$$ as a $\C_\infty[\tau]$-basis for $M$ and $$\bh = \begin{pmatrix} h_1\\ \vdots \\ h_d\end{pmatrix}$$ as a $\C_\infty[\sigma]$-basis for $N$ (later we will reuse this notation to define certain column matrices with entries in $A[t]$).

\subsubsection*{The maps $\delta_0^M,\delta_0^N$ and $\delta_{1,\bz}^M$}
Following \cite[\S 2]{GRE}, define $\CC_\infty$-linear maps
$$\delta_0^M:M\to \C_\infty^{d\times 1},\quad\delta_{1,\bz}^M:M \to \C_\infty, \text{ and }\quad \delta_0^N:N\to \C_\infty^{d\times 1}$$
in the following way. First, to define $\delta_0^M$, recall that $M=\CC_\infty[t]$ as a left $\CC_\infty[t]$-module. Every element $m$ of $M$ has a Taylor expansion at $\theta$:
$$m=\sum_{i\geq 0}\big(\mathcal{D}_{t,i}(m)\big)_{t=\theta}(t-\theta)^i$$
(the sum is finite). We set:
\begin{equation}\label{D:delta_0^M}
\delta_0^M(m) =\begin{pmatrix}
m|_{t=\theta}\\ \big(\mathcal{D}_{t,1}(m)\big)|_{t=\theta} \\ \vdots \\ \big(\mathcal{D}_{t,d-1}(m)\big)|_{t=\theta}
\end{pmatrix}=\left(\Big(\big(\boldsymbol{\partial}_t(m)\big)^\top\Big)^{\text{R}}\right)^\top\Bigg |_{t=\theta}\in\CC_\infty^{d\times 1}.
\end{equation}
In other words, $\delta_0^M$ is the column vector determined by the truncation to the order $d$ of the Taylor series of $m$ in $\theta$, or equivalently, the reverse of the column matrix $\boldsymbol{\partial}_t(m)_{t=\theta}$. Note that this map is a particular case of the motivic map described in \cite[Definitions 2.13 and 2.14]{GRE}, the reader can see that it is related to the residue map \cite[(2.5.6)]{AND&THA}.
It can be alternatively computed by expanding $m$ in the $\CC_\infty[\tau]$-basis $\bg$ of $M$: expand $$m=\sum_{j=0}^k\sum_{i=0}^{d-1}a_{i,j} \tau^j(t-\theta)^i,$$ with $a_{i,j}\in\CC_\infty$ for some $k\geq 0$. It is easy to check that 
\begin{equation}\label{D:delta_0^M alt}
\delta^M_0(m)=\begin{pmatrix}
a_{0,0}\\ a_{1,0} \\ \vdots \\ a_{d-1,0}
\end{pmatrix},
\end{equation} 
which agrees with the reverse of the column matrix $\boldsymbol{\partial}_t(m)_{t=\theta}$.
For fixed $\bz\in \C_\infty^d$, we also define a related map $\delta_{1,\bz}^M:M \to \C_\infty$ (which is not used until \S \ref{S:Motivic Identities}) by setting
\begin{equation}\label{D:delta_1^M}
\delta_{1,\bz}^M(m) = \begin{pmatrix}
a_{0,0}\\ a_{1,0} \\ \vdots \\ a_{d-1,0}
\end{pmatrix}^\top\bz + 
\begin{pmatrix}
a_{0,1}\\ a_{1,1} \\ \vdots \\ a_{d-1,1}
\end{pmatrix}^\top\bz\twist + \cdots +
\begin{pmatrix}
a_{0,k}\\ a_{1,k} \\ \vdots \\ a_{d-1,k}
\end{pmatrix}^\top\bz\twistk{k}.
\end{equation}
There exists an extension of $\delta_{1,\bz}^M$ to certain elements of $M\otimes_{\C_\infty[t]} \TT \isom \TT$ which is discussed extensively in \cite[\S 2.6 and \S 4.3]{GEZ&GRE}, and which covers the case discussed here. In an effort to be self contained, we give a description of this construction, but for detailed proofs we refer the reader to that reference. First, note that we have an isomorphism $\varphi:\C_\infty[\tau]^d \isom M \isom \C_\infty[t]$ using the basis $\bg$, in other words, for $m\in M$ we write
$$\varphi(m)=\sum_{j=0}^k\sum_{i=0}^{d-1}a_{i,j} \tau^j(t-\theta)^i \in \C_\infty[t].$$
We give $\C_\infty[t]$ a norm by viewing it as a subspace of $\TT$ and we give $M\isom \C_\infty[\tau]^d$ a norm by setting $\mathfrak{v}_u = q^{d/(q-1) + u-1}$ for $1\leq u \leq d$, then
\[
\left|\left(\sum_{n=0}^{k}a_{n,1}\tau^n,\dots,\sum_{n=0}^{k}a_{n,d}\tau^n\right)\right|:=\max_{n,u}\left\{|a_{n,u}|\mathfrak{v}_u^{q^n}\right\}.
\]
A short calculation shows that $\varphi$ is a bounded map between normed $\C_\infty$-vector spaces. Then set
\[\mathbb{M}:=\left \{\left(\sum_{n=0}^{\infty}a_{n,1}\tau^n,\dots,\sum_{n=0}^{\infty}a_{n,d}\tau^n\right)\bigg| |a_{n,u}|\mathfrak{v}_u^{q^n}\to 0 \text{ as } n\to \infty \right \},\]
with the same norm as above. The space $M$ is dense in $\mathbb{M}$ as normed $\C_\infty$-vector spaces, thus there exists a unique extension of $\varphi:M\to \C_\infty[t]$ to $\varphi:\mathbb{M} \to \TT$. Another calculation shows that $\varphi$ is injective. For $\bz = (z_1,\dots,z_d)$ with $|z_i|<\mathfrak{v}_i$, we define an extension of $\delta_{1,\bz}^M$ to $\mathbb{M}$ by setting
\begin{equation}\label{D:delta_1 extended def}
\delta_{1,\bz}^M(m) = \begin{pmatrix}
a_{0,0}\\ a_{1,0} \\ \vdots \\ a_{d-1,0}
\end{pmatrix}^\top\bz + 
\begin{pmatrix}
a_{0,1}\\ a_{1,1} \\ \vdots \\ a_{d-1,1}
\end{pmatrix}^\top\bz\twist + \cdots .
\end{equation}
Finally, we will consider $\delta_{1,\bz}^M$ to have domain inside $\TT$ using $\varphi \inv$, but we will omit the $\varphi\inv$ from the notation.

The construction is analogous to define $\delta_0^N$, (compare again with \cite[Definitions 2.13 and 2.14]{GRE}), but we must take into account the $\CC_\infty[\sigma]$-basis $\bh$. Therefore, expanding an element $n\in N$ in Taylor series at $\theta$ 
or in the basis $\bh$ writing 
$$n=\sum_{j=0}^k\sum_{i=0}^{d-1}b_{i,j}\sigma^j(t-\theta)^i,$$
we write, with obvious notations taking into consideration the left ideal $\CC_\infty[\sigma]\sigma$ of $\CC_\infty[\sigma]$ (the higher derivatives now occur in reverse order):
\begin{equation}\label{D:delta_0^N}
\delta^N_0(n):=\begin{pmatrix} \big(\mathcal{D}_{t,d-1}(n)\big)(\theta)\\ \big(\mathcal{D}_{t,d-2}(n)\big)(\theta)\\ \vdots \\ n(\theta)\end{pmatrix}=\begin{pmatrix} b_{d-1,0}\\ b_{d-2,0}\\ \vdots \\ b_{0,0}\end{pmatrix}\in\CC_\infty^{d\times 1}.
\end{equation}
This agrees with $\boldsymbol{\partial}_t(n)_{t=\theta}$.

We can see $M$ and $N$ as subsets of $\C_\infty[t,\tau,\tau\inv]$- and $\C_\infty[t,\sigma,\sigma\inv]$-modules, respectively, by setting $M_K = N_K = \C_\infty(t)$ and
\[\tau\inv m = \frac{1}{(t-\theta^{1/q})^d} m\twistinv\in \C_\infty(t),\quad m\in\CC_\infty(t)\]
\[\sigma\inv n = \frac{1}{(t-\theta^{q})^d} n\twist\in \C_\infty(t),\quad n\in\CC_\infty(t).\]
The maps $\delta^M_0,\delta^N_0$ extend uniquely to the subspace of elements in $\C_\infty(t)$ (or $\TT$) which are regular at $\theta$ (the definition is still \eqref{D:delta_0^N}), so in particular we may evaluate them at the elements $\tau^{-k}m$ and $\sigma^{-j}n$ for all $j,k\geq 0$ and all $m\in M$ and $n\in N$. The details of this extension can be found in \cite[Prop. 2.4.2]{CGM20}.

\begin{Lemma}\label{L:delta properties}
The maps $\delta_0^M$ and $\delta_0^N$ have the following properties for $m\in M$ and $n\in N$:
\begin{enumerate}
\item Both $\delta_0^M$ and $\delta_0^N$ are $\C_\infty$-linear
\item $\delta_0^M(tm) = (\theta+N)^\top \delta_0^M(m)$. 
\item $\delta_0^N(tn) = (\theta+N)\delta_0^N(n)$
\item For all $m\in \tau M$ and $n\in \sigma N$, we have $\delta_0^M(m) = 0$ and $\delta_0^N(n) = 0$.
\end{enumerate}
\end{Lemma}

\begin{proof}
These properties follow from the more general \cite[Proposition 2.15]{GRE}. 
\end{proof}

\begin{Lemma}\label{L:delta_0 alt}
For $\lambda\in (\Omega\twistinv)^d \F_q[t]$ we have
\[\left (\delta_0^M\left (\frac{1}{t-\theta}\lambda\right )\right )_d = \lambda\twist(\theta).\]
\end{Lemma}

\begin{proof}
First recall (\ref{omega-func-eq}). Thus
\[\frac{1}{t-\theta} (\Omega\twistinv)^d = (t-\theta)^{d-1} (\Omega)^d,\]
Then, since $\Omega(\theta) = -1/\tpi$, we apply the alternate expression for $\delta_0^M$ from \eqref{D:delta_0^M alt} and find that the bottom coordinate of $\delta_0^M\left ((t-\theta)^{d-1} (\Omega)^d\right )$ is $(-1/\tpi)^d$. This shows that the theorem holds for $\lambda = (\Omega\twistinv)^d$. We conclude that it holds for any $\lambda\in (\Omega\twistinv)^d \F_q[t]$ by observing that
\[\left (\delta_0^M\left ((t-\theta)\inv a(t)(\Omega\twistinv)^d\right )\right )_d =  \left (\delta_0^M\left (a(t)(t-\theta)^{d-1}(\Omega)^d\right )\right )_d= a(\theta) \left (\delta_0^M\left ((\Omega\twistinv)^d\right )\right )_d,\]
for any $a(t)\in \F_q[t]$. Finally, $\F_q[t]$ is fixed under Frobenius twisting, so $a(\theta) = a\twist(\theta)$, finishing the proof.
\end{proof}

\subsubsection{Pairings}
Recall the centralizer $\mathcal{Z}(\CC_\infty)\subset\CC_\infty^{d\times d}$ of $N$ introduced at the beginning of \S \ref{Papanikolas-matrices}.
\begin{Definition}\label{D:E_l pairing}
Fix $\ell\geq 0$ and $W\in\mathcal{Z}(\CC_\infty)$. For $x\in \C_\infty[t,\tau]$ and $y\in \C_\infty[t,\sigma]$, we define the pairing
\[E_\ell: \C_\infty[t,\tau]\times \C_\infty[t,\sigma] \to \CC_\infty^{d\times d},\]
\[E_\ell(x,y;W) = \left(\sum_{j=0}^\ell (\delta_0^M(\tau^{-j}(x g_k))\twistj)^\top W\twistj \delta_0^N(\sigma^{j-\ell}(yh_m))\twistj\right )_{1\leq k,m \leq d}\in\CC_\infty^{d\times d}.\]
\end{Definition}

\begin{Proposition}\label{P:E_ell properties}
For all $\ell\geq 1$ the pairing $E_\ell$ satisfies the following properties:
\begin{enumerate}
\item $E_\ell(tx,y;W) = E_\ell(x,ty;W)$
\item $E_\ell(\theta^{q^\ell} x,y;W) = E_\ell(x,\theta y;W)$
\item $E_\ell(x,\sigma y;W) = E_{\ell-1}(x,y;W)$
\item $E_\ell(\tau x, y;W) = E_{\ell-1}(x,y;W)\twist$
\item $E_\ell(1,1;W) =   \sum_{j=0}^\ell Q_j W\twistj P_{\ell-j}\twistj$.
\end{enumerate}
\end{Proposition}

\begin{proof}
Part (1) follows from Lemma \ref{L:delta properties} Parts (2) and (3), from the fact that $t$ commutes with $\tau$ and $\sigma$ and the fact that $W$ is in $\mathcal{Z}(\CC_\infty)$. The proof of Part (2) is a simple computation using the fact that $\tau z = z^q \tau$ and $\sigma z = z^{1/q} \sigma$ for all $z\in \C_\infty$. Parts (3) and (4) are a direct computation using Lemma \ref{L:delta properties} Parts (1) and (4). Part (5) follows from \cite[Corollaries 3.8 and 4.5]{GRE}.
\end{proof}

\begin{Definition} 
We also introduce a modified version of $E_\ell$ which satisfies simpler properties, by setting
\[E_\ell'(x,y;W) = \left(\sum_{j=0}^\ell (\delta_0^M(\tau^{-j}(x g_k))\twistj)^\top W\twistj \delta_0^N\left (\sigma^{j-\ell}\left ((yh_m)\twistk{-\ell}\right )\right )\twistj\right )_{k,m = 1}^d.\]
\end{Definition}

\begin{Proposition}\label{P:E_ell prime properties}
$E_\ell'$ satisfies the following properties, for $W\in\mathcal{Z}(\CC_\infty)$.
\begin{enumerate}
\item $E_\ell'(tx,y;W) = E_\ell'(x,ty;W)$
\item $E_\ell'(\theta x,y;W) = E_\ell'(x,\theta y;W)= \theta E_\ell'(x,y;W)$
\item $E_\ell'(1,1,W) = E_\ell(1,1,W)H_{\theta^{q^\ell},\theta}$
\end{enumerate}
\end{Proposition}

\begin{proof}
Parts (1) and (2) follow similarly as in the proof of Proposition \ref{P:E_ell properties}. For part (3), we recall from Proposition \ref{P:E_ell properties} that $E_\ell(1,1;W) = \sum_{j=0}^\ell Q_j W\twistj P_{\ell-j}\twistj$. Additionally, from Proposition \ref{proposition-Papanikolas-PQ} we have $P_\ell=\dPsn{\ell}\Htwo{\theta}{\theta^{q^\ell}}$. Then, from \eqref{inverse-H} and 
\eqref{groupoidH} we see that
\begin{align*}
P_{\ell-j}\twistj H_{\theta^{q^\ell},\theta} &= (\dPsn{\ell-j})\twistj(\Htwo{\theta}{\theta^{q^{\ell-j}}})\twistj H_{\theta^{q^\ell},\theta}\\
&= (\dPsn{\ell-j})\twistj(\Htwo{\theta^{q^j}}{\theta^{q^{\ell}}}) H_{\theta^{q^\ell},\theta}\\
&= (\dPsn{\ell-j})\twistj\Htwo{\theta^{q^j}}{\theta}\\
&= (\dPsn{\ell-j})\twistj(\Htwo{\theta}{\theta^{q^{-j}}})\twistj.
\end{align*}
Finally, a careful comparison with definitions shows that this is the same as 
\[\left (\delta_0^N(\sigma^{j-\ell}(h_1)\twistk{-\ell})\twistj,\dots,\delta_0^N(\sigma^{j-\ell}(h_d)\twistk{-\ell})\twistj\right ).\]
\end{proof}

\begin{Proposition}
For all $W\in\mathcal{Z}(\CC_\infty)$ we have that $E_\ell'(1,1;W)^\perp = E_\ell'(1,1;W)$.
\end{Proposition}

\begin{proof}
By Proposition \ref{P:E_ell prime properties} we see that $E_\ell'((t-\theta),1;W) = E_\ell'(1,(t-\theta);W)$. Recalling the definitions of $g_k$ and $h_k$, we see that $(t-\theta)g_k = g_{k+1}$ for $1\leq k\leq d-1$ and $(t-\theta)h_k = h_{k-1}$ for $2\leq k\leq d$. The proposition then follows from the definition of $E_\ell'$.
\end{proof}

\begin{Definition}\label{D:E ell def}
 Set, for $W\in\mathcal{Z}(\CC_\infty)$,
\[F_\ell(x,y;W) = \Gamma_\ell\inv E_\ell'(x,y;W).\]
Going forward we will often set $x=y=1$ in the above pairings. We will thus suppress that notation, and simply write
$E_\ell(W) := E_\ell(1,1;W)$ or simply $E_\ell$ (if the value of $W$ is understood from the context) and similarly for $E'_\ell$ and $F_\ell$. We record here that
\begin{equation}\label{E:F E relation}
F_\ell(W) = \Gamma_\ell\inv E_\ell(W) H_{\theta^{q^\ell},\theta}.
\end{equation}

\end{Definition}

\begin{Remark}
We make a brief comment about the utility of the three pairings $E_\ell$, $E_\ell'$ and $F_\ell$. Each is natural from a certain viewpoint, and thus it is worthwhile to study each of them in turn. First, $E_\ell$ is the most natural generalization of Carlitz's polynomials, in view of \ref{P:E_ell properties} (5). On the other hand, $E_\ell'$ provides the normalization allowing us to obtain the non-commutative factorization of Theorem \ref{T:non-commutative factorization E_ell'}. Finally, $F_\ell$ is a final normalization such that the limit $\ell\to \infty$ exists (see Theorem \ref{factorization-of-sine-complete}).

\end{Remark}



\subsection{A recursive formula for $E_\ell'$}\label{SS:Noncom factorization}
In this section we derive a recursive formula (a certain non-commu\-tative factorization) for $E_\ell'$ which results in our Theorem A (see Theorem \ref{theorem-factor-complete}) giving new families of linear relations between Carlitz multiple polylogarithms (see \S \ref{section-projecting-on-one-coord} for details on Carlitz multiple polylogarithms, the main result of this subsection is Theorem \ref{T:non-commutative factorization E_ell'}). 
Such recursive formulas are not unique in general; the second author considers an infinite family of similar recursive formulas in \cite{PEL6}. The relationship between these different non-commutative factorizations will be discussed in Remark \ref{R:Factorization Comparison}.

\begin{Definition}\label{defi-Ml}
For $1\leq \ell\leq d$
we define matrices $M_\ell\in\operatorname{GL}_d(K)$ as follows. First, we expand in unique way for $1\leq \ell\leq d$
\[(t-\theta^{q^\ell})^d = (\theta-\theta^{q^\ell})^d + a_1(t-\theta) + a_2(t-\theta)^2 + \dots + a_{d-1}(t-\theta)^{d-1} + (t-\theta)^d,\]
for coefficients $a_i\in K$ depending on $\ell$ (note that the coefficients $a_i$ can be written as evaluations of hyperderivatives). We then set
\[M_\ell = \begin{pmatrix}
1 & 0 & 0 & 0 & \cdots & 0\\
a_{d-1} & 1 & 0 & 0 & \cdots & 0\\
a_{d-2} & a_{d-1} & 1  & 0 & \cdots & 0\\
a_{d-3} & a_{d-2} & a_{d-1} & 1   & \cdots & 0\\
\vdots & \vdots & \vdots & \vdots & \ddots & \vdots\\
a_1 & a_2 & a_3 & a_4 &  \cdots & 1
\end{pmatrix}.\]

\end{Definition}

\begin{Remark} 
The matrices $M_\ell$ can be equivalently defined as follows. Write $$\bg = (g_1,\dots,g_d)^\top=\begin{pmatrix} 1\\ t-\theta \\ \vdots \\ (t-\theta)^{d-1}\end{pmatrix}\in\CC_\infty[t]^{d\times 1}.$$ Then 
\begin{equation}\label{E:t-theta ql expansion}
(t-\theta^{q^\ell})^d \cdot\bg = \Big(\dern_t((t-\theta^{q^\ell})^d)\Big)_{t=\theta} \bg + (t-\theta)^d\cdot M_\ell\cdot \bg.
\end{equation}

\end{Remark}

In all the following we also write 
\begin{equation}\label{defi-H}
H := \Htwo{\theta^q}{\theta},\end{equation} 
to simplify our formulas.

\begin{Theorem}\label{T:non-commutative factorization E_ell'}
We have the following recursive formula for $E_\ell'(W)$, with $W\in\mathcal{Z}(\CC_\infty)$:
\[E_\ell'(W) = -\dern_t\left((t-\theta^{q^\ell})^{-d}\right )_{t=\theta} \left( M_\ell \cdot H^\perp  (E_{\ell-1}'(W))\twist  H  - E_{\ell-1}'(W)\right ).\]
\end{Theorem}

\begin{proof}
Our starting point is the identity
\begin{equation}\label{E:E_ell' t-theta}
E_\ell'((t-\theta^{q^\ell})^d,1;W) = E_\ell'(1,(t-\theta^{q^\ell})^d;W),
\end{equation}
which follows from Proposition \ref{P:E_ell prime properties} (where it is crucial that $W\in\mathcal{Z}(\CC_\infty)$). The right-hand side of \eqref{E:E_ell' t-theta} gives
\begin{align*}
&\left(\sum_{j=0}^\ell (\delta_0^M(\tau^{-j}(g_k))\twistj)^\top W\twistj \delta_0^N(\sigma^{j-\ell}((t-\theta^{q^\ell})^dh_m)\twistk{-\ell})\twistj\right )_{k,m = 1}^d\\
=& \left(\sum_{j=0}^\ell (\delta_0^M(\tau^{-j}(g_k))\twistj)^\top W\twistj \delta_0^N(\sigma^{j-\ell}(t-\theta)^d(h_m)\twistk{-\ell})\twistj\right )_{k,m = 1}^d\\
=& \left(\sum_{j=0}^\ell (\delta_0^M(\tau^{-j}(g_k))\twistj)^\top W\twistj \delta_0^N(\sigma^{j-\ell}\circ \sigma(h_m)\twistk{1-\ell})\twistj\right )_{k,m = 1}^d\\
=& \left(\sum_{j=0}^\ell (\delta_0^M(\tau^{-j}(g_k))\twistj)^\top W\twistj \delta_0^N(\sigma^{j-\ell+1}(h_m)\twistk{1-\ell})\twistj\right )_{k,m = 1}^d\\
=& E'_{\ell-1}(1,1;W).
\end{align*}
We recall that $\sigma n = (t-\theta)^d n\twistinv$ for $n\in\CC_\infty$ and that $\delta_0^N(\sigma(h_i)) = 0$, which facts we use in the last two lines.

On the other hand, the left-hand side of \eqref{E:E_ell' t-theta} gives ($\be_k$ is the $k$-th standard basis vector of $\CC_\infty^{d\times 1}$)
\begin{align*}
&\left(\sum_{j=0}^\ell (\delta_0^M(\tau^{-j}((t-\theta^{q^\ell})^dg_k))\twistj)^\top W\twistj \delta_0^N(\sigma^{j-\ell}(h_m)\twistk{-\ell})\twistj\right )_{k,m = 1}^d\\
=&\left(\sum_{j=0}^\ell (\delta_0^M(\tau^{-j}(\be_k^\top(t-\theta^{q^\ell})^d\bg))\twistj)^\top W\twistj \delta_0^N(\sigma^{j-\ell}(h_m)\twistk{-\ell})\twistj\right )_{k,m = 1}^d.
\end{align*}
We then use the expansion for $(t-\theta^{q^\ell})$ from \eqref{E:t-theta ql expansion} and Proposition \ref{P:E_ell prime properties} Part (2) to write this as
\begin{align*}
=&\left(\sum_{j=0}^\ell (\delta_0^M(\tau^{-j}(\be_k\dern_t\left((t-\theta^{q^\ell})^{d}\right )_{t=\theta} \bg + (t-\theta)^d\cdot M_\ell\cdot \bg))\twistj)^\top W\twistj \delta_0^N(\sigma^{j-\ell}(h_m)\twistk{-\ell})\twistj\right )_{k,m = 1}^d\\
=& \dern_t\left((t-\theta^{q^\ell})^{d}\right )_{t=\theta} E_\ell'(1,1;W) +  M_\ell E'_\ell((t-\theta)^d,1;W).
\end{align*}
The first term of the above equality will become the second term in the right-hand side of the statement of the proposition. We must deal with the second term. We first observe that for $k\geq 0$
\begin{equation}\label{E:t-thetainv eq}
(t-\theta\twistinv)^k = (\theta-\theta\twistinv)^k + \mathcal{D}_{t,1}(t-\theta\twistinv)^k|_{t=\theta}\cdot (t-\theta) + \dots + \mathcal{D}_{t,k-1}(t-\theta\twistinv)^k|_{t=\theta}\cdot (t-\theta)^{k-1} + (t-\theta)^k.
\end{equation}
We observe that the coefficients of the powers of $(t-\theta)$ in \eqref{E:t-thetainv eq} twisted give the columns of $H$. Then using the fact that $g_k = (t-\theta)^{k-1}$, that $h_m = (t-\theta)^{d-m}$ and \eqref{E:t-thetainv eq} we write, with
$$\bh = (h_1,\dots,h_d)^\top=\begin{pmatrix} (t-\theta)^{d-1}\\ (t-\theta)^{d-2} \\ \vdots \\ 1 \end{pmatrix}\in\CC_\infty[t]^{d\times 1}$$ and recalled that $H$ is defined in (\ref{defi-H}),
\begin{align*}
E_\ell((t-\theta)^d,1;W)&= \left(\sum_{j=0}^\ell (\delta_0^M(\tau^{-j}\circ \tau(g_k\twistinv))\twistj)^\top W\twistj \delta_0^N(\sigma^{j-\ell}(h_m)\twistk{-\ell})\twistj\right )_{k,m = 1}^d\\
&= \left(\sum_{j=0}^\ell (\delta_0^M(\tau^{1-j}(g_k\twistinv))\twistj)^\top W\twistj \delta_0^N(\sigma^{j-\ell}(h_m\twistinv)\twistk{1-\ell})\twistj\right )_{k,m = 1}^d\\
&= \left(\sum_{j=0}^\ell (\delta_0^M(\tau^{1-j}(\be_k^\top (H\twistinv)^\perp \bg)\twistj)^\top W\twistj \delta_0^N(\sigma^{j-\ell}(\bh^\top H\twistinv \be_m)\twistk{1-\ell})\twistj\right )_{k,m = 1}^d\\
&=H^\perp \cdot E_{\ell-1}'(1,1;W)\twist\cdot H\\
\end{align*}
Putting all these calculations together yields the theorem.
\end{proof}

\begin{Corollary}\label{C:E ell recursive}
We have the following recursive formula for $F_\ell(W)$, for all $W\in\mathcal{Z}(\CC_\infty)$
\[F_\ell(W) = -\Gamma_{\ell-1}\inv \cdot M_\ell\cdot H^\perp \cdot(\Gamma_{\ell-1} F_{\ell-1}(W))\twist \cdot H + F_{\ell-1}(W).\]
\end{Corollary}

\begin{proof}
This follows from Theorem \ref{T:non-commutative factorization E_ell'} and the Definition \ref{D:E ell def} of $F_\ell$ after noticing that $\Gamma_\ell\inv \cdot \dern_t\left((t-\theta^{q^\ell})^{-d}\right )_{t=\theta} = \Gamma_{\ell-1}\inv$.
\end{proof}

\begin{Definition}\label{D:L_k def}
In order to simplify notation, we define a map for $k\geq 0$
\[L_k:\operatorname{End}_{\CC_\infty}(\GG_a^d(\C_\infty))\to \operatorname{End}_{\CC_\infty}(\GG_a^d(\C_\infty)),\]
\[L_k(W) = \Gamma_{k}\inv \cdot M_{k+1}\cdot H^\perp \cdot \Gamma_{k}\twist\cdot W\twist \cdot H.\]
In this notation we have $F_\ell(W) = (1-L_{\ell-1})(F_{\ell-1}(W))$.

\end{Definition}
Using the notation above, we have the following non-commutative factorization of the operator $F_\ell$.
\begin{Corollary}\label{C:non-commutative factorization}
For all $W\in\mathcal{Z}(\CC_\infty)$ we have $F_0(W) = W$ and for $\ell\geq 1$:
\[F_\ell(W) =\Big( (1-L_{\ell-1})\circ (1-L_{\ell-2})\circ \dots \circ (1-L_1)\circ(1-L_0)\Big)(W).\]
\end{Corollary}

\begin{Remark}
Corollary \ref{C:non-commutative factorization} is very similar to the non-commutative factorization given for the Carlitz sine function, as described in \eqref{sine-composition-factorization} (when one lets the parameter tend to infinity).
We remark that one key difference is that our $L_0$ is not equal to the identity if $d>1$.

\end{Remark}

We deduce the following recursive formula for Papanikolas' generalization of Carlitz polynomials, restricted to 
$\mathcal{Z}(\CC_\infty)$.

\begin{Corollary}\label{corollary-nathan}
Consider $W\in\mathcal{Z}(\CC_\infty)$. Then
\begin{equation}\label{eq-nathan}
E_\ell(W)=\sum_{j=0}^\ell Q_jW^{(j)}P_{\ell-j}^{(j)}=\dPsn{\ell}\Big((1-L_{\ell-1})\cdots(1-L_0)(W)\Big)H_{\theta,\theta^{q^\ell}}.
\end{equation}
\end{Corollary}

\subsection{One sided non-commutative factorizations}\label{section-non-commutative}
In the recursive formulas we have obtained so far (Theorem \ref{T:non-commutative factorization E_ell'}, Corollaries \ref{C:E ell recursive}, \ref{C:non-commutative factorization}, \ref{corollary-nathan}), there is always multiplication on the left and right by nonzero matrices. We can say that the recursive formulas
take place in $\operatorname{End}_{\FF_q}(\operatorname{End}_{K}(\GG_a^d(K)))$.

An important observation we make in this subsection is that, with an appropriate normalization, it is possible to 
neutralize the factors which multiply on the right so that we can obtain recursive formulas, and therefore non-commutative factorizations, in $\operatorname{End}_{K}(\GG_a^d(K))[\tau]$. This is made possible by the formulas (\ref{groupoidH}).

Define

\begin{equation}\label{eq-cal-L}
\widetilde{\mathcal{L}}_k:=\Gamma_k^{-1}M_{k+1}H^\perp (\dPsn{k})^{(1)}\in\operatorname{GL}_d(K).
\end{equation}
Write:
\begin{equation}\label{def-normalization}
\boldsymbol{E}_k:=\Gamma_k^{-1}\sum_{j=0}^kQ_j\big(\dPsn{k-j}\big)^{(j)}\tau^j\in\operatorname{End}_K(\GG_a^d(K))[\tau].\end{equation}
We have:
\begin{Proposition}\label{theorem-factor} The following identity holds in $\operatorname{End}_{K}(\GG_a^d(K))[\tau]$:
$$\boldsymbol{E}_k=(1-\widetilde{\mathcal{L}}_{k-1}\tau)\cdots
(1-\widetilde{\mathcal{L}}_{1}\tau)(1-\widetilde{\mathcal{L}}_{0}\tau).$$
\end{Proposition}

\begin{proof} Write $W=\widehat{Z}\in\mathcal{Z}(\CC_\infty)$ with $Z\in\GG_a^d(\CC_\infty)$.
We note that
\begin{equation}\label{identity-Nathan}
F_k(W)=\Gamma_k^{-1}E_k(W)H_{\theta^{q^k},\theta}=\Gamma_k^{-1}\sum_{j=0}^kQ_j\big(\dPsn{k-j}\big)^{(j)}
 W^{(j)}\Htwo{\theta^{q^j}}{\theta}.
\end{equation}
The first identity is clear by \eqref{E:F E relation}. The second identity follows easily combining
(\ref{groupoidH}) and (\ref{inverse-H}), as well as Proposition \ref{proposition-Papanikolas-PQ}, which implies  
$$P_{k-j}^{(j)}=\big(\dPsn{k-j}\big)^{(j)}H_{\theta^{q^j},\theta^{q^k}}$$ (note that $W$ and $\big(\dPsn{k}\big)^{(j)}$ commute).

From Corollary \ref{C:non-commutative factorization} we get that
\begin{align}\label{id-Nathan-E}
F_k(W)&=(1-L_{k-1})\circ (1-L_{k-2})\circ \dots \circ (1-L_1)\circ(1-L_0)(W)\\
&=\sum_{j=0}^{k-1}(-1)^{j}\sum_{0\leq i_1<\cdots <i_j\leq k-1}  (L_{i_j}\circ \cdots \circ L_{i_1})(W).\nonumber
\end{align}
Note that $$L_k(W)=\widetilde{\mathcal{L}_k}\cdot W^{(1)}\cdot H_{\theta^q,\theta}$$ (this can be deduced directly from Definition \ref{D:L_k def}). 
For the term in $W^{(j)}$ we get, by iterating (\ref{groupoidH}),
\begin{align}\label{id-Nathan-E2}(L_{i_j}\circ \cdots \circ L_{i_1})(W)=&\widetilde{\mathcal{L}}_{i_1}\widetilde{\mathcal{L}}_{i_2}^{(1)}\cdots\widetilde{\mathcal{L}}^{(j-1)}_{i_j}W^{(j)}H_{\theta^{q^j},\theta^{q^{j-1}}}\cdots H_{\theta^{q^2},\theta^{q}}H_{\theta^{q},\theta}\nonumber\\ &=
\widetilde{\mathcal{L}}_{i_1}\widetilde{\mathcal{L}}_{i_2}^{(1)}\cdots\widetilde{\mathcal{L}}^{(j-1)}_{i_j}W^{(j)}H_{\theta^{q^j},\theta}\nonumber\\
&= \widetilde{\mathcal{L}}_{i_1}\widetilde{\mathcal{L}}_{i_2}^{(1)}\cdots\widetilde{\mathcal{L}}^{(j-1)}_{i_j}\big(WH_{\theta,0}\big)^{(j)}H_{0,\theta}.\end{align}

We compute the projection on the last column of $F_k(W)$, which we denote $[F_k(W)]_d$. 
The projection $[\cdot]_d$ on the last column $R^{d\times d}\rightarrow R^{d\times1}$ ($R$ any ring) induces an isomorphism from $\mathcal{Z}(\CC_\infty)$ to 
$\CC_\infty^{d\times 1}$. We have $[W]_d=Z$.
Now recall that if $M$ and $H=(*,\ldots,*,H_d)$ are any two matrices of $R^{d\times d}$ with $H$ having $H_d$ as last column, 
$[MH]_d=MH_d$.
With $\be_d$ denoting the $d$-th component of the canonical basis of $\GG_a^d(\C_\infty)$, we further observe that 
if $M=(*,\ldots,*,M_d), H=(*,\ldots,*,\boldsymbol{e}_d)$ are two matrices of $R^{d\times d}$ having as last columns $M_d$ and $\boldsymbol{e}_d$, then $$[MH]_d=M_d;$$ its last column is $M_d$. Note that we can here choose $H=H_{a,b}$ for any $a,b$,
the last column is $\be_d$.
 Applying this isomorphism to \eqref{identity-Nathan} yields 
$$[F_k(W)]_d=\Gamma_k^{-1}\sum_{j=0}^kQ_j\big(\dPsn{k-j}\big)^{(j)}Z^{(j)}.$$
 On the other hand from \eqref{id-Nathan-E} and \eqref{id-Nathan-E2} we deduce
 $$ [F_k(W)]_d=\sum_{j=0}^{k-1}(-1)^{j}\sum_{0\leq i_1<\cdots <i_j\leq k-1}  \widetilde{\mathcal{L}}_{i_1}\widetilde{\mathcal{L}}_{i_2}^{(1)}\cdots\widetilde{\mathcal{L}}^{(j-1)}_{i_j}Z^{(j)}.$$
We deduce the requested identity of 
endomorphisms in $\operatorname{End}_{K}(\GG_a^d(K))[\tau]$.
\end{proof}

\begin{Remark}\label{R:Factorization Comparison}
We discuss the relation between our factorization of $E_\ell$ and that of $E_{\phi,k}$ of \cite{PEL6}. 
By \cite[Theorem 5.2 and the exact sequence (5.3)]{PEL6}, just like in the case $d=1$, the function $\sin_A^{\otimes d}$ is the exponential function of the Anderson module of rank one and dimension $d$ denoted by $\widetilde{\phi}$ in ibid., defined by $$\widetilde{\phi}_\theta=\widehat{\Pi}^{-1}C^{\otimes d}_\theta\widehat{\Pi}.$$ 

Observe that, with the notations of ibid., $E_k = E_{\phi,k}$, but we choose different normalizations of these operators, as is seen by comparing \eqref{E:Federico Normalization} with \eqref{E:F E relation}. Besides that, the factorizations are obtained using fundamentally different methods. Namely, the factorization of $E_{\phi,k}$ is obtained by the second author by noting that for $W\in\mathcal{Z}(\CC_\infty)$ we have
\[C^{\otimes d}_\theta \exp_{C^{\otimes d}}(W \log_{C^{\otimes d}}) = \exp_{C^{\otimes d}}(W \log_{C^{\otimes d}}) C^{\otimes d}_\theta.\]
One then compares coefficients of powers of $\tau$ and gets identities between $E_{\phi,k}$ and $E_{\phi,k-1}$. In the language of our motivic constructions, the above identity is equivalent to observing that
\[E_k(t,1;W) = E_k(1,t;W).\]
Indeed, it is shown in \cite[Lemma 2.10(2) and Example 3.13]{GRE} that $t\bg = C^{\otimes d}_\theta \bg$. Thus
\begin{align*}
E_k(t,1;W) &= \left(\sum_{j=0}^\ell (\delta_0^M(\tau^{-j}(t g_k))\twistj)^\top W\twistj \delta_0^N(\sigma^{j-\ell}(h_m))\twistj\right )_{k,m = 1}^d\\
&= \left(\sum_{j=0}^\ell (\delta_0^M(\tau^{-j}(\be_k^\top C^{\otimes d}_t \bg))\twistj)^\top W\twistj \delta_0^N(\sigma^{j-\ell}(h_m))\twistj\right )_{k,m = 1}^d\\
&= \left(\sum_{j=0}^\ell (\delta_0^M(\tau^{-j}(\be_k^\top (\theta I_d + N + \leveln{e}_{d,1}\tau) \bg))\twistj)^\top W\twistj \delta_0^N(\sigma^{j-\ell}(h_m))\twistj\right )_{k,m = 1}^d\\
&= (\theta + N) E_k(1,1;W) + e_{d,1} E_k(\tau,1;W)\\
&= (\theta + N) E_k(1,1;W) + e_{d,1} E_{k-1}(1,1;W)\twist,
\end{align*}
where we remind the reader of the definitions of the matrices $N$ and $e_{d,1}$ 
in (\ref{E:C otimes def}) and just after. Similarly, 
\[E_k(1,t;W) =  E_k(1,1;W)(\theta^{q^k} + N) +  E_{k-1}(1,1;W)e_{d,1}.\]
Putting these two identities together gives \cite[(6.2)]{PEL6}.

The second author also describes alternate factorizations of $E_{\phi,k}$, parametrized by $a\in A$. These identities are similarly obtained from the equality
\[E_k(a,1;W) = E_k(1,a;W).\]
The final connection then comes from noticing that the starting point for our factorization in Theorem \ref{T:non-commutative factorization E_ell'} is the identity (which fundamentally uses the different normalization)
\[E_k'((t-\theta^{q^\ell})^d,1;W) = E_k'(1,(t-\theta^{q^\ell})^d;W),\]
which by the $\C_\infty$-bilinearity of $E_k'$ (see Prop. \ref{P:E_ell prime properties} (2)) may be written as
\[\sum_{m=0}^d (-1)^m \binom{d}{m} \theta^{q^\ell} E_k'(t^{d-m},1;W).\]
Thus, our factorization can be viewed as an $A$-linear combination of a different normalization of the second author's factorizations from \cite{PEL6}. However, it is not obvious that such linear combinations of factorizations of $E_{\phi,k}$ is again a factorization of $E_{\phi,k}$. The methods we use here show that this is true.
\end{Remark}

We also add, for completeness, as it was claimed in the introduction:

\begin{Lemma}\label{computation-kernel-lemma} We have the following identity of $A$-modules free of rank one induced by $\operatorname{Lie}(C^{\otimes d})$:
$$\operatorname{Ker}(\sin_A^{\otimes d})=\left\{\dern_\theta(a)\begin{pmatrix}0\\ \vdots \\ 0\\ 1\end{pmatrix}:a\in A\right\}.$$
\end{Lemma}

\begin{proof}
Recall that $\sin^{\otimes d}_A(Z)=\widehat{\Pi}^{-1}\exp_{C^{\otimes d}}\widehat{\Pi}Z$. By \cite[Corollary 2.5.9]{AND&THA} $\exp_{C^{\otimes d}}(\widehat{\Pi}Z)=0$ if and only if 
$\widehat{\Pi}Z\in\{d_\theta(a):a\in A\}\Pi$. By (\ref{commutation}) we have $\widehat{\Pi}Z=\widehat{Z}\Pi$ so that $Z\in\operatorname{Ker}(\sin_A^{\otimes d})$ if and only if there exists $a\in A$ such that
$\widehat{Z}=\dern_\theta(a)$. Now, recalling that $[\cdot]_d$ denotes the projection on the last column,
$Z=[\widehat{Z}]_d=\partialn_\theta(a)=\dern_{\theta}(a)\partialn_\theta(1)$.
\end{proof}

\section{$\Delta$-matrices}\label{xyz-formalism}

In order to complete the proof of Theorem A in the introduction we use Proposition \ref{theorem-factor}, but we still need to determine alternative expressions for the coefficients $\widetilde{\mathcal{L}}_k$. In this section we discuss the main properties of the matrix valued differential operators introduced in the introduction, that provide the appropriate tools in view of our results.
The formalism that we introduce needs at least three independent indeterminates. This has been already partly employed in \cite{PEL6}, but the use of $\Delta$-matrices is new.

Let $F$ be a field. 
Most of the results we are going to state and prove in this section are identities of matrices in $\operatorname{GL}_{d}(\ZZ[x,y,z,\ldots])$ that reduce modulo $p$, the characteristic of $\FF_q$, and we can set $F=\FF_q$ or $F=K$. We recall the $\Delta$-operator of the introduction.
\begin{Definition}
For $f\in F(x,z)$ we set
\begin{multline*}
\Deltad{d}_{x,z}(f):=\begin{pmatrix}\mathcal{D}_{x,d-1}(f) & \mathcal{D}_{x,d-1}(\mathcal{D}_{z,1}(f)) & \cdots & \mathcal{D}_{x,d-1}(\mathcal{D}_{z,d-1}(f))\\
\mathcal{D}_{x,d-2}(f) & \mathcal{D}_{x,d-2}(\mathcal{D}_{z,d-1}(f)) & \cdots & \mathcal{D}_{x,d-2}(\mathcal{D}_{z,1}(f))\\
\vdots & \vdots & & \vdots\\
f & \mathcal{D}_{z,1}(f) & \cdots & \mathcal{D}_{z,d-1}(f)
\end{pmatrix}=\\ =\Big(\mathcal{D}_{x,d-i}(\mathcal{D}_{z,j-1}(f))\Big)_{1\leq i,j\leq d}\in F(x,z)^{d\times d}.
\end{multline*} 
We call it the {\em $\Deltan_{x,z}$-matrix associated to $f$}. Loosely, we may speak about {\em $\Delta$-matrices}. A $\Delta$-matrix $\Deltan_{x,z}(f)$ is uniquely determined by its {\em root} $f$, which is 
the coefficient in the $d$-th row and 1st column.
\end{Definition}

Observe that, by the fact that the divided higher derivatives in $x$ and $z$ commute,
$$\Deltan_{x,z}(f)^\perp=\Deltan_{z,x}(f).$$
The following is the basic elementary criterion to recognize $\Delta$-matrices.

\begin{Lemma}\label{criterion-delta}
A matrix $M\in F(x,z)^{d\times d}$ is a $\Delta_{x,z}$-matrix if and only if $M$ is 
a $\partial_x$-matrix and its antitranspose $M^\perp$ is a
$\partial_z$-matrix.
\end{Lemma}
\begin{proof}
 One implication is obvious. Suppose that $M$ is a $\partial_x$-matrix and at the same time $M^\perp$ is a
$\partial_z$-matrix. We can write (a) $M=\partialn_x(m_0,\ldots,m_{d-1})$ and (b) $M^\perp=\partialn_z(m'_0,\ldots,m'_{d-1})$
for elements $m_0,\ldots,m_{d-1},m'_0,\ldots,m'_{d-1}\in F(x,z)$. We deduce that $m_0=m'_0$. We are going to show that $f:=m_0=m'_0$ is the root of $M$ as a $\Delta$-matrix. From (a) we see that $m'_i=\mathcal{D}_{x,i}(f)$ and from (b) that $m_j=\mathcal{D}_{z,j}(f)$, and this for all $0\leq i,j\leq d-1$. Now given two elements
$i,j\in\{1,\ldots,d\}$, the $(i,j)$-entry $m_{i,j}$ of $M$ is, by (a), 
$$m_{i,j}=\mathcal{D}_{x,d-i}(m_{j-1})=\mathcal{D}_{x,d-i}(\mathcal{D}_{z,j-1}(f)).$$
From (b) we also see that
$$m_{i,j}=\mathcal{D}_{z,j-1}(m'_{d-i})=\mathcal{D}_{z,j-1}(\mathcal{D}_{x,d-i}(f)).$$
The above expressions agree thanks to the fact that for all $i,j$, the operators $\mathcal{D}_{x,i}$ and $\mathcal{D}_{z,j}$ commute.\end{proof}

Our next task is to describe basic compatibility properties among $\Delta$-matrices.

\begin{Lemma}\label{first-compatibility} Let us consider two indeterminates $x,z$ independent over $F$. 
Suppose that $f\in F(x)$, $g\in F(x,z)$ and $h\in F(z)$. Then
$$\dern_x(f)\Deltan_{x,z}(g)\dern_z(h)=\Deltan_{x,z}(fgh).$$
\end{Lemma}

\begin{proof}
By Leibnitz's formula, $\dern_x(f)\Deltan_{x,z}(g)\dern_z(h)=\Deltan_{x,z}(fg)\dern_z(h)$. Taking the anti-transpose of
 $\Deltan_{x,z}(fg)\dern_z(h)$ we get
 $$\big(\Deltan_{x,z}(fg)\dern_z(h)\big)^\perp=\dern_z(h)^\perp\Deltan_{x,z}(fg)^\perp=\dern_z(h)\Deltan_{x,z}(fg)^\perp=\Deltan_{x,z}(fgh)^\perp$$
because $\Deltan_{x,z}(fg)^\perp$ is a $\partial_z$-matrix.
\end{proof}

Hence the map $$F(x,z)\xrightarrow{\Deltan_{x,z}}F(x,z)^{d\times d}$$
is left $F(x)$-linear and right $F(z)$-linear via $\dern_x$ and $\dern_z$.
It defines an $F$-linear map 
$$F(x,z)\rightarrow\operatorname{Bil}^+\Big(F(x)\times F(z)\rightarrow F(x,z)^{d\times d}\Big),$$
into bilinear maps (for $\dern_x$ and $\dern_z$) which are symmetric in the sense that 
if $\iota_{x,z}$ is the involution of $F(x,z)$ that exchanges the indeterminates $x,z$, $B$ is in the target space if and only if $\iota_{x,z}(B^\perp)=B$ (this is condensed in the notation $\operatorname{Bil}^+$ explaining the $+$ sign).

There is an analogue of (\ref{groupoidH}) for the operators $\Deltan$.

\begin{Lemma}\label{delta-groupoid}
Let us consider three independent indeterminates $x,y,z$ over $F$. We consider 
$f,g\in F(x,y,z)$. Suppose moreover that $f\in F(x,y)$ and $g\in F(y,z)$.
Then 
$$\Deltan_{x,y}(f)\Deltan_{y,z}(g)=
\Deltan_{x,z}\Big(\mathcal{D}_{y,d-1}(fg)\Big).$$
\end{Lemma}

\begin{proof}
If $M$ is a $\partial_x$-matrix and $M'$ is a matrix whose entries do not depend on the variable $x$, then
$MM'$, if well defined, is a $\partial_x$-matrix. Hence both left- and right-hand sides of the identity are $\partial_x$-matrices.
Both the matrices $(\Deltan_{x,y}(f)\Deltan_{y,z}(g))^\perp=\Deltan_{z,y}(g)\Deltan_{y,x}(f)$ and 
$\Deltan_{u,w}\Big(\mathcal{D}_{v,d-1}(fg)\Big)^\perp=
\Deltan_{w,u}\Big(\mathcal{D}_{v,d-1}(fg)\Big)$ are 
$\partial_z$-matrices. But by Leibnitz's rule, the bottom left coefficients of the identity of the lemma agree, and therefore the identity is verified.
\end{proof}

\subsection{How to construct $\Delta$-matrices}

Suppose that, over a field $F$, we have two independent indeterminates $x,z$ and, for any $d$, two $d$-tuples of elements 
$a_1,\ldots,a_d\in F(x)$ and $b_1,\ldots,b_d\in F(z)$. We set 
$$\varphi:=\sum_{i=1}^da_ib_{d-i+1}\in F(x,z).$$

The first way to construct $\Delta$-matrices is given by the next lemma.

\begin{Lemma}\label{generic-delta}
The following formula holds:
$$\partialn_x(a_1,\ldots,a_d)\big(\partialn_z(b_1,\ldots,b_d)\big)^\perp=\Deltan_{x,z}\big(\varphi\big).$$
\end{Lemma}

\begin{proof}
Set $U:=\partialn_x(a_1,\ldots,a_d)$ and $V=\partialn_z(b_1,\ldots,b_d)$. Then $UV^\perp$ is a 
$\partial_x$-matrix and $(UV^\perp)^\perp=VU^\perp$ is a
$\partial_z$-matrix so that $UV^\perp$ is a $\Deltan_{x,z}$-matrix and, in the two matrices 
$UV^\perp$ and $\Deltan_{x,z}(\varphi)$, the entries in the $d$-th row and first column.
\end{proof}

In the opposite direction, a $\Deltan_{x,z}$-matrix needs not be the left product of the anti-transpose of a  $\partial_z$-matrix by a $\partial_x$-matrix. Already if $d=1$, there are examples of elements of $F(x,z)$ which are not products of elements in $F(x)$ and elements in $F(z)$. By the fact that $F[x,z]=F[x]\otimes_FF[z]$, every $\Deltan_{x,z}$-matrix with polynomial root is a finite linear combination of such matrices.

From Lemma \ref{generic-delta} we deduce the next formula, where $x,y,z,t$ are independent variables over a field $F$:
\begin{equation}\label{delta-xyzt}
\Htwo{x}{y}(\Htwo{z}{t})^\perp=\Deltan_{x,z}\Big(\sum_ {i = 1}^d (x - y)^{d - i} (z - t)^{i - 1}\Big).\end{equation}

Recall the sequence of polynomials $(f_d)_d$ in (\ref{fd})
and set $$f'_d:=\sum_ {j = 1}^d \binom {d} {j} (x - y)^{d - j} (z - x)^{j - 1}\in\FF_q[x,y,z].$$
We also set $f'_0=0$.
We have:
\begin{Lemma}\label{lemmafS}
For all $d\geq 0$, $f'_d(z,y,z)=f_d(x,y,z)=\sum_{i=1}^d(x-y)^{d-i}(z-y)^{i-1}$.
\end{Lemma}

\begin{proof}
It suffices to show the recursive formula
$f'_d=(x-y)f'_{d-1}+(z-y)^{d-1}$ for $d\geq1$ because the lemma follows easily from it by induction.
But
\begin{eqnarray*}
\lefteqn{f'_d-(x-y)f'_{d-1}=}\\
&=&\sum_{j=1}^d\Big(\binom{d}{j}-\binom{d-1}{j}\Big)(x-y)^{d-j}(z-x)^{j-1}\\
&=&\sum_{j=1}^d\binom{d-1}{j-1}(x-y)^{d-j}(z-x)^{j-1}\\
&=&\sum_{i=0}^{d-1}\binom{d-1}{i}(x-y)^{d-1-i}(z-x)^{i}\\
&=&(x-y+z-x)^{d-1}=(z-y)^{d-1}.
\end{eqnarray*}
\end{proof}
We define
$$\Mtwo{x}{y}:=\big(1+(x-y)N^\top\big)^d\in\operatorname{GL}_d(\FF_p[x,y]).$$
For elements $a,b\in K$, we set $\Mtwo{a}{b}:=\big(\Mtwo{x}{y}\big)_{\begin{smallmatrix}x=a\\
y=b\end{smallmatrix}}\in\operatorname{GL}_d(K)$.
Note that 
\begin{equation}\label{twoMs}
\Mtwo{\theta}{\theta^{q^\ell}}=M_\ell,\end{equation} with the family of matrices $(M_\ell)_\ell$ defined in (\ref{defi-Ml}).
\begin{Lemma}\label{lemma-MH}
 $\Mtwo{x}{y}\Htwo{z}{x}^\perp=\Deltan_{x,z}(f_d).$
\end{Lemma}
\begin{proof}
By construction, the lower left coefficient of $\Mtwo{x}{y}(\Htwo{z}{x})^\perp$ is $f'_d$ that we know being equal to $f_d$ thanks to Lemma \ref{lemmafS}. Since $\Mtwo{x}{y}\in\FF_p[x,y]^{d\times d}$
does not depend on $z$, by the fact that $\Htwo{z}{x}$ is a $\partial_z$-matrix it is easy to see that 
$(\Mtwo{x}{y}(\Htwo{z}{x})^\perp)^\perp=\Htwo{z}{x}(\Mtwo{x}{y})^\perp=\Htwo{z}{x}\Mtwo{x}{y}$ is a 
$\partial_z$-matrix. Now, to show that $\Mtwo{x}{y}(\Htwo{z}{x})^\perp$ is a $\partial_x$-matrix, it suffices to show that its first column $C$ is, because the operators $\mathcal{D}_{x,i}$ and $\mathcal{D}_{z,j}$ commute for all $i,j$. This follows from the following identities which can be proved by induction on $d$ for all $i$,
$$\mathcal{D}_{x,d-i}(f_d)=\sum_{j=1}^d\binom{d}{d-i+j}(x-y)^{i-j}(z-x)^{j-1},$$
noticing that the right-hand side coincides with the $i$-th coefficient of $C$.
\end{proof}

By a computation of degrees in the variables involved in $f_d$ both $\Mtwo{x}{y}(\Htwo{z}{x})^\perp$ and $\Deltan_{x,z}(f_d)$ are lower triangular with ones over the diagonal. 
We easily deduce from Lemmas \ref{first-compatibility}, \ref{lemma-MH} and the identity (\ref{twoMs}):

\begin{Corollary}\label{Corollary-Metc}
For all $k\geq 0$ the matrix $\widetilde{\mathcal{L}}_k=\Gamma_k^{-1}M_{k+1}H^\perp \Gamma_k^{(1)}$ is the evaluation at $x=\theta,y=\theta^{q^{k+1}},z=\theta^q,t=\theta$ of 
the $\Delta$-matrix $\Deltan_{x,z}(g_{d,k})$ with
$$g_{d,k}=\frac{c_{d,k}(x,t)}{c_{d,k}(z,t^q)}f_{d}(x,y,z)\in\FF_p(x,y,z,t),$$ 
where 
$$c_{d,k}(x,t):=\prod_{i=1}^k(x-t^{q^i})^{d}.$$ \end{Corollary}

We can finally state and prove:

\begin{Theorem}\label{theorem-factor-complete} The following identity holds in $\operatorname{End}_{K}(\GG_a^d(K))[\tau]$:
$$\boldsymbol{E}_k=(1-\mathcal{L}_{k-1}\tau)\cdots
(1-\mathcal{L}_{1}\tau)(1-\mathcal{L}_{0}\tau),$$
where
\begin{equation}\label{explicit-formula-forL}\mathcal{L}_i=\Delta_{x,z}\left(f_d(x,y,z)\Big(\frac{(x-\theta^q)\cdots(x-\theta^{q^i})}{(z-\theta^{q^2})\cdots(z-\theta^{q^{i+1}})}\Big)^d\right)_{\begin{smallmatrix}x=\theta\\ y=\theta^{q^{i+1}}\\ z=\theta^q\end{smallmatrix}}.
\end{equation}
\end{Theorem}

This is Theorem A in the introduction. 
\begin{proof}
This follows directly from a combination of Proposition \ref{theorem-factor} and Corollary \ref{Corollary-Metc}.\end{proof}

\begin{Remark} Theorem A implies that
$$\boldsymbol{E}_k=\big(1-\mathcal{L}_{k-1}\tau\big)\boldsymbol{E}_{k-1},\quad k\geq 1.$$
By using the formula (\ref{def-normalization}) and comparing the coefficients of the different powers of $\tau$ we deduce recursive formulas for the coefficients $Q_i$ of the exponential $\exp_{C^{\otimes d}}=\sum_{i\geq 0}Q_i\tau^i$ (reviewed in \S \ref{section-carlitz-tensor}). These are:
\begin{multline*}
Q_j\Big(\boldsymbol{d}_t\big((t-\theta^{q^{k-j}})^d\big)_{t=\theta}\Big)^{(j)}-\boldsymbol{d}_t\big((t-\theta^{q^k})^d\big)_{t=\theta}Q_j=\\=-\boldsymbol{\Delta}_{x,z}\big(f_d(x,y,z)\big)_{\begin{smallmatrix}x=\theta\\ y=\theta^{q^k}\\ z=\theta^q\end{smallmatrix}}Q_{j-1}^{(1)},\quad 1\leq j\leq k
\end{multline*}
($\boldsymbol{d}_t$ is defined in \S \ref{Papanikolas-matrices}). Compare with the formula
\cite[(2.2.2)]{AND&THA} by Anderson and Thakur:
$$Q_i\boldsymbol{d}_t(t)_{t=\theta^{q^i}}-\boldsymbol{d}_t(t)_{t=\theta}Q_i=-\boldsymbol{\Delta}_{x,z}(1)Q_{i-1}^{(1)},\quad i\geq 1.$$
If $j=k$ one term of our formula vanishes and we get 

\begin{Corollary}
\begin{equation*}
\boldsymbol{d}_t\big((t-\theta^{q^k})^d\big)_{t=\theta}Q_j=-\left(\boldsymbol{\Delta}_{x,z}\big(f_d(x,y,z)\big)\right)_{\begin{smallmatrix}x=\theta\\ y=\theta^{q^k}\\ z=\theta^q\end{smallmatrix}}Q_{k-1}^{(1)},\quad 0\leq j\leq k
\end{equation*}
\end{Corollary}
It is easy to deduce that $Q_i$ is invertible for all $i$. This property is not easy to deduce directly from Anderson and Thakur formula and our formula can be compared with a formula by Papanikolas, reproduced in our Proposition \ref{proposition-Papanikolas-PQ}.
\end{Remark}

\subsection{Proof of Theorem B}

Finally, we reach a proof of Theorem B of the introduction. We recall that $\mathcal{L}_i$ is given in (\ref{explicit-formula-forL}). 

\begin{Theorem}\label{factorization-of-sine-complete}
We have the following factorization:
$$\sin_A^{\otimes d}=\prod_{i\geq 0}^{\longleftarrow}\Big(1-\mathcal{L}_i\tau\Big)\in\operatorname{End}_{K_\infty}(\GG_a^d(K_\infty))[[\tau]].$$
The convergence of the product holds uniformly on every bounded subset of $\CC_\infty^d$. In particular, convergence holds coefficient by coefficient in the $\tau$-expansions of both the left and right hand sides.
\end{Theorem}
\begin{proof}
For $R\in|\CC_\infty^\times|$, denote by $D(0,R)$ the set of $x\in\CC_\infty$ such that $|x|<R$.
By Theorem \ref{theorem-factor-complete} it suffices to show that for every $z\in D(0,R)^d$,
$$\left|\sin_A^{\otimes d}(z)-\boldsymbol{E}_k(z)\right|\leq c(R,k)$$
where $c(R,k)$ is a constant depending on $R,k$, with $\lim_{k\rightarrow\infty}C(R,k)=0$ for every $R$.
This is quite a standard verification. Recall from \eqref{definition-of-Gamma-k}, for $k\geq 0$:
\begin{eqnarray*}
\Gamma_k&=&\dern_x\Big((x-\theta^q)\cdots(x-\theta^{q^k})\Big)_{x=\theta}^{-d}\\
&=&\dern_x\Big((-\theta)^{-d(q+\cdots+q^k)}\Big(\big(1-\frac{x}{\theta^q}\big)\cdots\big(1-\frac{x}{\theta^{q^k}}\big)\Big)^{-d}\Big)_{x=\theta}\\
&=&(-1)^d(-\theta)^{-\frac{q^{k+1}}{q-1}}\dern_x\Big((-1)^d(-\theta)^{\frac{dq}{q-1}}\Big(\big(1-\frac{x}{\theta^q}\big)\cdots\big(1-\frac{x}{\theta^{q^k}}\big)\Big)^{-d}\Big)_{x=\theta}.
\end{eqnarray*}
Write  
$$\widehat{\Pi}_k=\dern_x\Big((-1)^d(-\theta)^{\frac{dq}{q-1}}\Big(\big(1-\frac{x}{\theta^q}\big)\cdots\big(1-\frac{x}{\theta^{q^k}}\big)\Big)^{-d}\Big)_{x=\theta},$$
so that 
$$\Gamma_k=(-1)^d(-\theta)^{-\frac{q^{k+1}}{q-1}}\widehat{\Pi}_k.$$
We deduce:
$$\Gamma_k^{-1}Q_j\Gamma_{k-j}^{(j)}=\widehat{\Pi}_k^{-1}Q_j\widehat{\Pi}_{k-j}^{(j)},\quad k\geq j\geq0.$$
By \eqref{maurischat-formula}:
$$\widehat{\Pi}-\widehat{\Pi}_k=\widehat{\Pi}_k\dern_x\left(\prod_{l>k}\Big(1-\frac{x}{\theta^{q^l}}\Big)^{-d}-1\right)_{x=\theta}.$$
Similarly, for every $k$,
$$\widehat{\Pi}^{-1}-\widehat{\Pi}^{-1}_k=\widehat{\Pi}_k^{-1}\dern_x\left(\prod_{l>k}\Big(1-\frac{x}{\theta^{q^l}}\Big)^{d}-1\right)_{x=\theta}.$$
Denote by $\mathcal{M}_{K_\infty}$ the maximal ideal of $K_\infty$ (that is, $\mathcal{M}_{K_\infty}=\theta^{-1}\FF_q[[\theta^{-1}]]$). As
$$\dern_x\Big(1-\frac{x}{\theta^{q^k}}\Big)_{x=\theta}\in 1+(\mathcal{M}^{q^k}_{K_\infty})^{d\times d}$$
where $1$ denotes the identity matrix of $K_\infty^{d\times d}$, we get that
$$\widehat{\Pi}-\widehat{\Pi}_k\in(-1)^d(-\theta)^{\frac{dq}{q-1}}(\mathcal{M}^{q^{k+1}}_{K_\infty})^{d\times d},\quad\widehat{\Pi}^{-1}-\widehat{\Pi}^{-1}_k\in(-1)^d(-\theta)^{-\frac{dq}{q-1}}(\mathcal{M}^{q^{k+1}}_{K_\infty})^{d\times d}.$$
Now, for $k\geq j\geq0$,
$$\widehat{\Pi}^{-1}Q_j\widehat{\Pi}^{(j)}\in\Big(\widehat{\Pi}^{-1}_k+(-1)^d(-\theta)^{-\frac{dq}{q-1}}(\mathcal{M}^{q^{k+1}}_{K_\infty})^{d\times d}\Big)Q_j\Big(\widehat{\Pi}_{k-j}^{(j)}+(-1)^d(-\theta)^{\frac{dq^{j+1}}{q-1}}(\mathcal{M}^{q^{k+1}}_{K_\infty})^{d\times d}\Big).$$
As $|Q_j|\leq |\theta|^{-jq^j}$ for all $j$ by \cite[Proposition 2.4.2]{AND&THA}, we find the following inequality
$$
\Big|\widehat{\Pi}^{-1}Q_j\widehat{\Pi}^{(j)}-\widehat{\Pi}_k^{-1}Q_j\widehat{\Pi}_{k-j}^{(j)}\Big|\leq |\theta|^{-q^{k+1}}|\theta|^{\frac{dq^{j+1}}{q-1}-jq^j}.
$$
Let $[\sin_A^{\otimes d}]_k\in K_\infty^{d\times d}[\tau]$ be the truncation of $\sin_A^{\otimes d}\in K_\infty^{d\times d}[[\tau]]$
to the order $k+1$ so that all the entries in $K_\infty[\tau]$ have degree in $\tau$ which is $\leq k$. Choose 
$z\in D(0,R)^d$ with $R\in|\CC_\infty^\times|$. Then
$$[\sin_A^{\otimes d}]_k(z)-\boldsymbol{E}_k(z)=\sum_{j=0}^k\Big(\widehat{\Pi}^{-1}Q_j\widehat{\Pi}^{(j)}-\widehat{\Pi}_k^{-1}Q_j\widehat{\Pi}_{k-j}^{(j)}\Big)z^{q^j}$$
so that
$$\left|[\sin_A^{\otimes d}]_k(z)-\boldsymbol{E}_k(z)\right|\leq|\theta|^{-q^{k+1}}\max_{0\leq j\leq k}\{R'{}^{q^j}|\theta|^{-jq^j}\},$$
where $R'=R|\theta|^{\frac{dq}{q-1}}$. Elementary theory of functions of a real variable ensures now the existence of 
a constant $c_1(R')$ (effective and depending on $R'$ only) such that 
\begin{equation}\label{inequality}
\left|[\sin_A^{\otimes d}]_k(z)-\boldsymbol{E}_k(z)\right|\leq|\theta|^{-q^{k+1}}c_1(R').
\end{equation}
With the same choice of $z$,
\begin{eqnarray*}
\left|\sin_A^{\otimes d}(z)-\boldsymbol{E}_k(z)\right|&\leq&\max\left\{\left|\sin_A^{\otimes d}(z)-[\sin_A^{\otimes d}]_k(z)\right|,\left|[\sin_A^{\otimes d}]_k(z)-\boldsymbol{E}_k(z)\right|\right\}
\end{eqnarray*}
which tends to $0$ by \eqref{inequality} and the fact that $\sin_A^{\otimes d}-[\sin_A^{\otimes d}]_k\in K_\infty^{d\times d}[[\tau]]\tau^{k+1}$ is the tail series of
a formal series representing an entire function $\CC_\infty^d\rightarrow\CC_\infty^d$. This concludes the proof 
of the uniform convergence of the sequence $\boldsymbol{E}_k$ to $\sin_A^{\otimes d}$ on bounded subsets of $\CC_\infty^d$.
The convergence of the corresponding sequences of coefficients in the expansions in $K_\infty^{d\times d}[[\tau]]$ follows easily.
It is in fact easier to get it directly without the above uniform convergence considerations, we omit the details.
\end{proof}

\section{Application to new identities of Carlitz multiple polylogarithms}\label{section-projecting-on-one-coord}

In this section we show our Theorem C as a consequence of Theorem \ref{nathan-theorem} (the main result of this section), through Corollary \ref{finite-version-imply-Theorem-C}.
Recall that we wrote:
\begin{equation}\label{D:Lambda_k}
\mathcal{L}_k=\Gamma_k^{-1}\Deltan_{x,z}(f_d(x,y,z))_{\begin{smallmatrix}x=\theta\\ y=\theta^{q^{k+1}} \\ z=\theta^q\end{smallmatrix}}\dPsn{k}^{(1)},\quad k\geq 0.
\end{equation}
\subsubsection*{Matrix multiple sums}
We define, for $k\geq m\geq1$,
$$\mathcal{L}_{<k}(m):=\sum_{k>i_1>\cdots>i_m\geq0}\mathcal{L}_{i_1}\mathcal{L}_{i_2}^{(1)}\cdots\mathcal{L}_{i_m}^{(m-1)}\in K^{d\times d}.$$
Note that for all $k\geq0$,
\begin{eqnarray}
\mathcal{L}_{<k}(m)&=&\mathcal{L}_{k-1}\mathcal{L}_{<k-1}(m-1)^{(1)}+\mathcal{L}_{<k-1}(m),\quad m>1\label{Reclambda1}\\
\mathcal{L}_{<m}(m)&=&\mathcal{L}_{m-1}\mathcal{L}_{m-2}^{(1)}\cdots\mathcal{L}_0^{(m-1)},\quad m>1.\label{Reclambda2}
\end{eqnarray}
Recall the main identity of Theorem \ref{theorem-factor-complete}. We deduce the following identity.
\begin{equation}\label{identity}
\mathcal{L}_{<k}(m)=(-1)^m\Gamma_k^{-1}Q_m\dPsn{k-m}^{(m)},\quad k\geq m\geq1.
\end{equation}
We recall that $$l_n=(\theta-\theta^q)(\theta-\theta^{q^2})\cdots(\theta-\theta^{q^n}),\quad n\geq 0.$$
For all $M\in K^{d\times e}$ we denote by $[M]_1$ the first column
of $M$ and by $[M]_{d,1}$ the last coefficient of $[M]_1$.
The identity (\ref{identity}) implies
$$[\mathcal{L}_{<k}(m)]_1=(-1)^m\Gamma_k^{-1}[Q_m]_1l_{k-m}^{-dq^m}.$$ Recall from \cite[Corollary 4.2.4]{PAP}:
\begin{equation}\label{papanikolas}
[Q_m]_1=D_m^{-d}\begin{pmatrix}1 \\ [m]\\ \vdots \\ [m]^{d-1}\end{pmatrix}
\end{equation}
where $[m]=\theta^{q^m}-\theta$ for $m\geq 1$ and 
$$D_m=[m][m-1]^q\cdots[1]^{q^{m-1}},\quad m\geq 0.$$
This also follows directly from Proposition \ref{proposition-Papanikolas-PQ}. We also set $[0]=1$.
The column vector $[\mathcal{L}_{<k}(m)]_1$ is determined by the 
coefficient $[\mathcal{L}_{<k}(m)]_{d,1}$ in the following sense.

Define, for $m\geq 1$ and $k\geq 0$:
$$U_{k,m}:=[m]^{1-d}l_{k}^{-d}\Gamma_k^{-1}\begin{pmatrix}1 \\ [m]\\ \vdots \\ [m]^{d-1}
\end{pmatrix}=\begin{pmatrix}* \\ \vdots \\ * \\1\end{pmatrix}\in K^{d\times 1}.
$$
Then 
$$\Gamma_k^{-1}[Q_m]_1=l_k^{d}[m]^{d-1}D_m^{-d}U_{k,m}$$
and we deduce the following two identities, corresponding to (\ref{identity}):
\begin{eqnarray}
[\mathcal{L}_{<k}(m)]_1 & =& [\mathcal{L}_{<k}(m)]_{d,1}U_{k,m}\label{id1}\\
& =& [m]^{d-1}D_m^{-d}l_k^dl_{k-m}^{-dq^m}U_{k,m}.\label{id2}
\end{eqnarray}

\subsubsection*{Scalar multiple sums}

Now set
\begin{equation}\label{lambdaij}
\lambda_{i,j}:=[j-1]^{(1-d)q}l_i^{d(1-q)}f_d(\theta,\theta^{q^{i+1}},\theta^{q^j})\in K,\quad i\geq 0,\quad j\geq 1,
\end{equation} and the multiple sum
\begin{equation}\label{defi-lambda<}
\lambda_{<k}(m):=\sum_{k>i_1>\cdots>i_m\geq0}\lambda_{i_1,m}\lambda_{i_2,m-1}^q\cdots\lambda_{i_m,1}^{q^{m-1}},\quad k\geq m\geq 1.
\end{equation}
Note the identities 
\begin{eqnarray}
\lambda_{<k}(m)&=&\lambda_{k-1,m}\Big(\lambda_{<k-1}(m-1)\Big)^{(1)}+\lambda_{<k-1}(m),\quad m>1\label{reclambda1}\\
\lambda_{<m}(m)&=&\lambda_{m-1,m}\Big(\lambda_{<m-1}(m-1)\Big)^{(1)}\label{reclambda2}\\
&=&\lambda_{m-1,m}\lambda_{m-2,m-1}^{(1)}\cdots\lambda_{0,1}^{(m-1)},\quad m>1\nonumber.
\end{eqnarray}

\begin{Theorem}\label{nathan-theorem} For $k\geq m\geq 1$ we have
\begin{eqnarray}
[\mathcal{L}_{<k}(m)]_{d,1}&=&\lambda_{<k}(m),\label{theo1}\\
&=&(-1)^ml_k^d[m]^{d-1}D_m^{-d}l_{k-m}^{-dq^m}.\label{theo2}
\end{eqnarray}
\end{Theorem}
The equality \eqref{theo2} is a direct consequence of (\ref{identity}). We focus on (\ref{theo1}). Let us first study it in the case $m=1$.
Note that in this case we have
$$\lambda_{i,1}=l_i^{d(1-q)}f_d(\theta,\theta^{q^{i+1}},\theta^q).$$
Now
\begin{eqnarray}
[\mathcal{L}_{<k}(m)]_{d,1}&=&\Big[\sum_{k>j}\mathcal{L}_j\Big]_{d,1}
\nonumber\\
&=&\sum_{k>j}\Big[\dPsinv{d}{j}\Deltan_{x,z}(f_d)_{x=\theta,y=\theta^{q^{j+1}},z=\theta^q}\big(\dPs{d}{j}\big)^{(1)}\Big]_{d,1}\nonumber\\
&=&\sum_{k>j}\lambda_{j,1},\quad \forall k\geq 1.\label{casem=1}
\end{eqnarray}
To prove the identity (\ref{theo1}) in the case $m>1$ we will use induction, but we first need two preliminary lemmas. The next one deals in fact with an identity in $\ZZ[x,y,z,z']$ with independent indeterminates $x,y,z,z'$ that reduces modulo $p$. Here we need
induction over $d\geq 1$. In conformity with our convention, we indicate the dependence in the parameter $d$ 
by writing $\Deltad{d}_{x,z}$ instead of $\Deltan_{x,z}$, $\partiald{d}_x$ instead of $\partialn_x$ etc.
(it is the only place in the paper where we do this).

\begin{Lemma}\label{lemma1}
We have, for variables $x,y,z,z'$,
$$\Deltad{d}_{x,z}\big(f_d(x,y,z)\big)\begin{pmatrix}1 \\ z'-z\\ \vdots \\ (z'-z)^{d-1}\end{pmatrix}=\partiald{d}_x\big(f_d(x,y,z')\big).$$
\end{Lemma}
\begin{proof}
We proceed by induction on $d\geq 1$. The case $d=1$ is trivial as
$f_d=1$. Now, recall from the proof of Lemma \ref{lemmafS} that
$$f_d(x,y,z)=(x-y)f_{d-1}(x,y,z)+(z-y)^{d-1},\quad d>1.$$ Applying the operator $\Deltad{d}_{x,z}$ and using its compatibility formulas with $d_x$-
and $d_z$-matrices Lemma \ref{first-compatibility} we obtain:
$$\Deltad{d}_{x,z}(f_d(x,y,z))=\derd{d}_x(x-y)\Deltad{d}_{x,z}(f_{d-1}(x,y,z))+\leveln{e}_{d,1}\derd{d}_z((z-y)^{d-1}).$$ Note indeed that $\leveln{e}_{d,1}$, the elementary matrix of $K^{d\times d}$ having all the coefficients equal to zero except the coefficient on the $d$-th row and first column, is a $\Delta$-matrix as it equals $\Deltad{d}_{x,z}(1)$. Now we multiply the above identity on the right by the column vector $(1,z'-z,\ldots (z'-z)^{d-1})^\top$.
First observe that 
\begin{multline*}
\Deltad{d}_{x,z}(f_{d-1}(x,y,z))\begin{pmatrix}1 \\ z'-z \\ \vdots \\ (z'-z)^{d-1}\end{pmatrix}=\\=\begin{pmatrix} 0 \\ \Deltad{d-1}_{x,z}(f_{d-1}(x,y,z))\begin{pmatrix}1 \\ z'-z \\ \vdots \\ (z'-z)^{d-2}\end{pmatrix}\end{pmatrix}=\partiald{d}_x(f_{d-1}(x,y,z'))\end{multline*}
by the fact that $\Deltad{d}_{x,z}(f_{d-1}(x,y,z))$ has a block decomposition with the first row and the last column zero and 
the remaining coefficients corresponding to $\Deltad{d-1}_{x,z}(f_{d-1}(x,y,z))$ as a submatrix, and the induction hypothesis.
Secondly, observe that the top coefficient of
$$\derd{d}_z((z-y)^{d-1})\begin{pmatrix}1 \\ z'-z \\ \vdots \\ (z'-z)^{d-1}\end{pmatrix}$$
equals $(y-z')^{d-1}$ by an obvious consequence of Taylor's series expansions. Hence
$$\leveln{e}_{n,1}\derd{d}_z((z-y)^{d-1})\begin{pmatrix}1 \\ z'-z \\ \vdots \\ (z'-z)^{d-1}\end{pmatrix}=\begin{pmatrix}0 \\ \vdots \\ 0\\ (y-z')^{d-1}\end{pmatrix}=\partiald{d}_x((y-z')^{d-1})$$
and the lemma follows because we get
\begin{multline*}
\Deltad{d}_{x,z}(f_{d}(x,y,z))\begin{pmatrix}1 \\ z'-z \\ \vdots \\ (z'-z)^{d-1}\end{pmatrix}=\\=\partialn_x(f_{d-1}(x,y,z'))+\partiald{d}_x((y-z')^{d-1})=\partiald{d}_x((y-z')^{d-1})(f_d(x,y,z')).
\end{multline*}
\end{proof}
From the above lemma we deduce the next statement. We can again omit the symbols $(\cdot)^{\langle d\rangle}$, taking into account Lemma \ref{lemma1}, we can now fix the value of $d$ again.
\begin{Lemma}\label{lemma2} We have that
$$\mathcal{L}_iU_{i,j}^{(1)}=[j]^{(1-d)q}l_i^{-dq}\Gamma_i^{-1}\partialn_x\big(f_d(x,y,z')\big)_{\begin{smallmatrix}x=\theta\\ y=\theta^{q^{i+1}}\\ z'=\theta^{q^{j+1}}\end{smallmatrix}},$$
so that 
$$[\mathcal{L}_iU_{i,j}^{(1)}]_{d,1}=[j]^{(1-d)q}l_i^{d(1-q)}f_d (\theta,\theta^{q^{i+1}},\theta^{q^{j+1}})=\lambda_{i,j+1}.$$
\end{Lemma}

\begin{proof} 
We have
\begin{eqnarray*}
\mathcal{L}_iU_{i,j}^{(1)}&=&\Gamma_i^{-1}\Deltan_{x,z}(f_d)_{\begin{smallmatrix}x=\theta\\ y=\theta^{q^{i+1}}\\ z'=\theta^{q}\end{smallmatrix}}\dPsn{i}^{(1)}[j]^{(1-d)q}l_i^{-dq}\big(\dPsn{i}^{(1)}\big)^{-1}\begin{pmatrix} 1 \\ [j]^q \\ \vdots \\ [j]^{(d-1)q}
\end{pmatrix}\\
&=&\Gamma_i^{-1}\left(\Deltan_{x,z}(f_d(x,y,z))\begin{pmatrix}1 \\ z'-z \\ \vdots \\ (z'-z)^{d-1}\end{pmatrix}\right)_{\begin{smallmatrix}x=\theta\\ y=\theta^{q^{i+1}}\\ z=\theta^q\\ z'=\theta^{q^{j+1}}\end{smallmatrix}}[j]^{(1-d)q}l_i^{-dq}\\
&=&\Gamma_i^{-1}\partialn_x(f_d(x,y,z'))_{\begin{smallmatrix}x=\theta\\ y=\theta^{q^{i+1}}\\ z'=\theta^{q^{j+1}}\end{smallmatrix}}[j]^{(1-d)q}l_i^{-dq}
\end{eqnarray*}
by Lemma \ref{lemma1}.
\end{proof}

\begin{proof}[Proof of Theorem \ref{nathan-theorem}]
Without loss of generality, we suppose $m>1$. 
As a first step we begin by showing the formula (\ref{theo1}) in the case $k=m$ for all $m\geq1$. To do this we use induction on $m\geq 1$, knowing that the formula is true in the case $k=m=1$ by (\ref{casem=1}). Let us suppose that the formulas are true for the integer $m-1$, namely:
$$[\mathcal{L}_{<m-1}(m-1)]_{d,1}=\lambda_{<m-1}(m-1)=(-1)^{m-1}[m-1]^{d-1}l_{m-1}^dD_{m-1}^{-d}.$$
It is easy to show that, with the assumption of the second identity,
\begin{equation}\label{lambdamm}
\lambda_{<m}(m)=(-1)^{m}[m]^{d-1}l_{m}^dD_{m}^{-d}.
\end{equation}
 Indeed, by
(\ref{reclambda2}), 
\begin{eqnarray*}
\lambda_{<m}(m)&=&\lambda_{m-1,m}\lambda_{<m-1}(m-1)\\
&=&(-1)^{m-1}[m-1]^{q(1-d)}l_{m-1}^{d(1-q)}f_d(\theta,\theta^{q^m},\theta^{q^m})[m-1]^{d-1}l_{m-1}^dD_{m-1}^{-d}\\
&=&[m-1]^{q(1-d)}l_{m-1}^{d(1-q)}(\theta-\theta^{q^m})^{d-1}[m-1]^{(d-1)q}l_{m-1}^{dq}D_{m-1}^{-dq}\\
&=&l_{m-1}^{d}(\theta-\theta^{q^m})^{d-1}D_{m-1}^{-dq},
\end{eqnarray*}
where the third equality holds because $f_d(x,z,z)=(x-z)^{d-1}$ implying $f_d(\theta,\theta^{q^m},\theta^{q^m})=(\theta-\theta^{q^m})^{d-1}$. This implies (\ref{lambdamm}). We have, by (\ref{Reclambda2}), (\ref{identity}), (\ref{id2}) and (\ref{lambdamm}),
\begin{eqnarray*}
[\mathcal{L}_{<m}(m)]_{d,1}&=&(-1)^m[\Gamma_m^{-1}Q_m]_{d,1}\\
&=&(-1)^m[m]^{d-1}l_m^dD_m^{-d}\\
&=&\lambda_{<m}(m).
\end{eqnarray*}
This proves the identity 
\begin{equation}\label{idm=m}
[\mathcal{L}_{<m}(m)]_{d,1}=\lambda_{<m}(m),\quad m\geq 1
\end{equation}
and proves the theorem in the case $k=m\geq 1$.
It remains to show the general formula (\ref{theo1}):
$$[\mathcal{L}_{<k}(m)]_{d,1}=\lambda_{<k}(m),\quad k\geq m\geq 1.$$ Again, we can suppose that $m>1$ and proceed this time by double induction, for $m\geq 1$ as a primary parameter, and then, for each $m$ given, for $k\geq m$ as a secondary parameter. Hence we can suppose that the formula holds for 
all couples of integers $(k,m)$ with $k\geq m-1\geq 1$. We have seen in (\ref{idm=m}) that the formula also holds in the case $(k,m)=(m,m)$. By the induction hypothesis on the secondary parameter $k\geq m$ we can suppose that 
$[\mathcal{L}_{<k-1}(m)]_{d,1}=\lambda_{<k-1}(m)$. Apply  (\ref{Reclambda1}) and then compute $[\cdot]_{d,1}$. We get:
\begin{eqnarray*}
[\mathcal{L}_{<k}(m)]_{d,1}&=&[\mathcal{L}_{k-1}U_{k-1,k-1}^{(1)}]_{d,1}\lambda_{<k-1}(m-1)^{(1)}+[\mathcal{L}_{<k-1}(m)]_{d,1}\\
&=&[m-1]^{(1-d)q}l_{k-1}^{d(1-q)}f_d(\theta,\theta^{q^k},\theta^{q^m})\lambda_{<k-1}(m-1)^{(1)}+\lambda_{<k-1}(m)\\
&=&\lambda_{k-1,m}\lambda_{<k-1}(m-1)^{(1)}+\lambda_{<k-1}(m)\\
&=&\lambda_{<k}(m),
\end{eqnarray*}
where the second equality is a consequence of Lemma \ref{lemma2} and the last equality follows from (\ref{reclambda1}). This completes the proof of our theorem.
\end{proof}
\subsubsection*{Consequences for Carlitz multiple polylogarithms}
Write, for all $k\geq 0$,
$$\mathcal{L}_{<k}\begin{pmatrix} \theta^{m_1}&\ldots&\theta^{m_r}\\ n_1&\ldots& n_r\end{pmatrix}:=\frac{X^{q^k}}{l_k^{n_0}}\sum_{k>i_1>\cdots>i_r\geq 0}\frac{\theta^{m_1q^{i_1}+m_2q^{i_2}+\cdots+m_rq^{i_r}}}{l_{i_1}^{n_1}l_{i_2}^{n_2}\cdots l_{i_r}^{n_r}}$$
We deduce, from Theorem \ref{nathan-theorem}:
\begin{Corollary}\label{finite-version-imply-Theorem-C} For all $r\geq 1$ and for all $k\geq r$,
$$\tau^r\left(\mathcal{L}_{<k-r}\binom{1}{d}\right)=(-1)^r(\theta^{q^r}-\theta)^{1-d}D_r^{d}\sum_{\underline{m}}c_{\underline{m}}
\mathcal{L}_{<k}\begin{pmatrix} 1&\theta^{m_1q} & \cdots &\theta^{m_rq^r}\\ d&d(q-1) &\cdots& d(q-1)q^{r-1}\end{pmatrix}.$$
\end{Corollary}

\begin{proof}
Recall the definitions (\ref{lambdaij}) and (\ref{defi-lambda<}).
By Theorem \ref{nathan-theorem} we have, for all $i\geq r$, $$l_{i-r}^{-dq^i}=(-1)^r[r]^{1-d}D_r^{d}l_i^{-d}\lambda_{<i}(r),\quad i\geq r.$$ It is plain that
$$\lambda_{<i}(r)=\sum_{\underline{m}}c_{\underline{m}}\mathcal{L}_{<i}\begin{pmatrix} \theta^{m_1q} & \cdots & \theta^{m_rq^r} \\
d(q-1) & \cdots & d(q-1)q^{r-1}\end{pmatrix}.$$
The formula of our corollary follows immediately. 
\end{proof}

Multiplying by appropriate powers of $C$ and taking the limit for $k\rightarrow\infty$ allows us to deduce Theorem C.

\section{Motivic interpretation of identities}\label{S:Motivic Identities}
The identities we prove in the previous three sections, culminating in Corollary \ref{finite-version-imply-Theorem-C}, are all proved at the level of finite sums, meaning we get linear relations amongst multiple  sums which result in linear relations amongst Carlitz multiple polylogarithms when we take the limit. In this section, we explain how these identities can be seen as naturally coming from identities happening amongst elements of a $t$-motive, from which we recover the scalar identities using the $\delta_{1,\bz}^M$ ``cycle integration" map mentioned in the introduction. Thus, we view these identities of this section as giving a $q^s$-power analogue of the integral representation of Carlitz multiple polylogarithms. The fact that these ``integral representations" also give rise to new families of $K$-linear relations between Carlitz multiple polylogarithms (as well as multiple zeta values) suggests the natural question of whether our formulas can be viewed as an analogue of the integral shuffle relations for real-valued polylogarithms (and multiple zeta values). This is a subtle question, especially in view of the recent advances of \cite{IKND26}, and will be explored in future work.

\begin{Definition}\label{D:I_i def}
Let $b_0=1$ and set $b_k = (t-\theta)\dots (t-\theta^{q^{k-1}})$ and define
\[\bb_k(t) = \begin{pmatrix}
b_k(t)^d\\
b_k(t)^{d-1}b_{k+1}(t)\\
b_k(t)^{d-2}b_{k+1}(t)^2\\
\vdots\\
b_k(t)b_{k+1}(t)^{d-1}\\
\end{pmatrix}.\]
Also define
\[I_i(W) = \sum_{k=0}^i E_k(W) \bb_k(t)\in \C_\infty[[t]]^d,\]
for $W\in\mathcal{Z}(\CC_\infty)$ as in Section \ref{SS:Motivic Pairings}.
\end{Definition}

\begin{Definition}
Recall the definition of the $t$-motive $M$ associated with the $d$th tensor power of the Carlitz module from Definition \ref{D:t-motive}, which is a free $\C_\infty[t]$-module of rank $1$ and a free $\C_\infty[\tau]$-module of rank $d$. We define $$\hM:= M \otimes_{\C_\infty[t]}\TT \isom \TT$$ with diagonal $\tau$-action.
\end{Definition}

\begin{Lemma}
If $h\in \hM$, then $\tau (h)\in \hM$.
\end{Lemma}

\begin{proof}
Recall that $\tau( h) = (t-\theta)^d h\twist.$ Thus, the lemma follow from the fact that the Frobenius acts as a contraction on the unit ball in $\C_\infty$.

\end{proof}

In order to explain the origin of our motivic identities, we write $I_i(W)$ in two different ways, each time, collecting common Frobenius twists of $W$. First, writing in the formula for $E_k$ from Proposition \ref{P:E_ell properties}(5) gives
\begin{align*}
I_i(W) &= \sum_{k=0}^i \sum_{j=0}^k Q_j W\twistj P_{k-j}\twistj \bb_k\\
&= \sum_{k=0}^i \sum_{j=0}^k Q_j W\twistj b_j (P_{k-j} \bb_{k-j})\twistj\\
&= \sum_{k=0}^i \sum_{j=0}^k Q_j W\twistj \tau^j(P_{k-j} \bb_{k-j}),\\
\end{align*}
where in the last two lines we used the definition of $\tau^j$ acting on $M$ and the fact that $b_k = b_j\cdot b_{k-j}\twistj$ for all $0\leq j\leq k$. We switch the order of summation and collect common powers of $\tau$ to get
\begin{align}\label{E:I_i Q_m expression}
I_i(W) &=\sum_{m=0}^i \sum_{n=m}^i Q_m W\twistk{m} \tau^m(P_{n-m} \bb_{n-m})\nonumber\\
&=\sum_{m=0}^i  Q_m W\twistk{m} \tau^m\left(\sum_{n=0}^{i-m} P_{n} \bb_{n}\right ).
\end{align}

Now, compare this with the normalization for $E_k$ from Definition \eqref{E:F E relation} to get
\begin{align*}
I_i(W) &= \sum_{k=0}^i E_k(W) \bb_k\\
&= \sum_{k=0}^i \Gamma_k\cdot F_k(W)\cdot H_{\theta,\theta^{q^k}} \cdot\bb_k.
\end{align*}
We then substitute the non-commutative factorization for $F_k$ from Corollary \ref{C:non-commutative factorization} to get
\begin{align}\label{E:I_i decomposition}
&=\sum_{k=0}^i \Gamma_k\cdot \left[(1-L_{k-1})\circ (1-L_{k-2})\circ \dots \circ (1-L_1)\circ(1-L_0)(W)\right ]\cdot H_{\theta,\theta^{q^k}} \cdot\bb_k\nonumber\\
&=\sum_{k=0}^i \Gamma_k\sum_{j=0}^{k-1}(-1)^{j}\sum_{0\leq i_1<\cdots <i_j<k}  (L_{i_j}\circ \cdots \circ L_{i_1})(W)\cdot H_{\theta,\theta^{q^k}} \cdot\bb_k\nonumber\\
&=\sum_{m=0}^i (-1)^{m}\sum_{n=m}^{i-1}\Gamma_n\sum_{0\leq i_1<\cdots <i_m<n}  (L_{i_m}\circ \cdots \circ L_{i_1})(W)\cdot H_{\theta,\theta^{q^n}} \cdot\bb_n.
\end{align}
Notice that, for fixed $m$, the last line of the above equation includes an $m$-fold Frobenius twist of the variable $W$ coming from Definition \ref{D:L_k def}. So, the motivic identities discussed in this section come from equating the bottom coordinates of \eqref{E:I_i Q_m expression} and \eqref{E:I_i decomposition}, comparing common powers of the Frobenius. To recover the scalar identities of the previous sections, we evaluate both sides under the $\delta_{1,\bz}^M$ map. We begin by studying the identities arising from \eqref{E:I_i Q_m expression}.
Write:

\begin{equation}\label{D:W_0}
W_0 =  w e_{1,d}\in \mathcal{Z}(\CC_\infty),
\end{equation}
for $w\in \C_\infty$.

\begin{Proposition}\label{P:I_i bottom coord}
Let $W_0$ be as above. Then we have
\[\left ( \lim_{i\to\infty} I_i(W_0) \right )_d = \sum_{m=0}^\infty\frac{1}{D_m^d}[m]^{d-1}w\twistk{m} \tau^m\left(\frac{\tpi^d}{(t-\theta)}\left (\Omega\twistinv\right )^d \right )\in \hM.\]
\end{Proposition}

\begin{proof}
We start with the identity \eqref{E:I_i Q_m expression}. Using the formula for $P_k$ from Proposition \ref{proposition-Papanikolas-PQ}, we see that the bottom coordinate of $P_{n} \bb_{n}$ is
\[b_{n}^d \frac{1}{\ell_{n}\ell_{n-1}^{d-1}} + b_{n}^{d-1}b_{n+1} \frac{1}{\ell_{n}^2\ell_{n-1}^{d-2}} + \dots + b_{n}b_{n+1}^{d-1} \frac{1}{\ell_{n}^{d}}.\]
Finally, taking the limit as $i\to \infty$ and comparing with \cite[(5.3)]{GRE} gives
\[\lim_{i\to\infty}\sum_{n=0}^i b_{n}^d \frac{1}{\ell_{n}\ell_{n-1}^{d-1}} + b_{n}^{d-1}b_{n+1} \frac{1}{\ell_{n}^2\ell_{n-1}^{d-2}} + \dots + b_{n}b_{n+1}^{d-1} \frac{1}{\ell_{n}^{d}} = \frac{\tilde \pi^n}{(t-\theta)\omega_C^n},\]
where we interpret any term with a negative index as being 0.

We then compute that $Q_m W_0\twistk{m}$ is a matrix with all zeros, except the right-most column, that by \eqref{papanikolas} equals
\[
w\twistk{m}\begin{pmatrix}
\frac{1}{D_m^d}\\
\frac{[m]}{D_m^d}\\
\vdots\\
\frac{[m]^{d-1}}{D_m^d}\\
\end{pmatrix}.
\]
Putting all these formulas together, multiplying out these matrices and taking the limit as $i\to \infty$ then gives the bottom coordinate of $\lim_{i\to\infty} I_i(W_0)$ as the formal sum
\[\sum_{m=0}^\infty\frac{1}{D_m^d}[m]^{d-1}w\twistk{m} \tau^m\left(\frac{\tpi^d}{(t-\theta)}\left (\Omega\twistinv\right )^d \right ),\]
as claimed. Finally, we compute $1/(t-\theta)(\Omega\twistinv )^d = (t-\theta)^{d-1} \Omega^d$ by (\ref{omega-func-eq}), and it is proved in several sources that $\Omega\in \TT$ (see \cite[\S 3.3.4]{PAP0}).
\end{proof}

\begin{Lemma}\label{L:tau structure delta_1}
Recall the definition of $\delta_{1,\bz}^M$ from \eqref{D:delta_1 extended def}. For $h\in \hM$, if $\delta_{1,\bz}^M(h) = c\in \C_\infty$, then we have
\[\delta_{1,\bz}^M(\tau^k (h)) = c^{q^k}.\]
\end{Lemma}

\begin{proof}
Then if $\bz = (z_1,\dots,z_d)^\top$, we have that $\delta_{1,\bz}^M(a \tau^l(t-\theta)^{m}) = a z_{m+1}^{q^l}$ for $a\in \C_\infty$, for $l\geq 0$ and for $0\leq m \leq d-1$. So we have that 
\[\tau^k(a \tau^l(t-\theta)^m) = a^{q^k}\tau^{k+m}(t-\theta)^m = a^{q^k} (z_{m+1}^{q^{m}})^{q^{k}}.\]
The result then follows by extending using $\C_\infty$-linearity.
\end{proof}

\begin{Lemma}\label{L:delta1 of b_k}
Set $\bz = (0,\dots,0,X)^\top$ for $X\in \C_\infty$. Then for $0\leq j\leq d-2$ we have 
\[\delta_{1,\bz}^M(b_k(t)^{d-j}b_{k+1}(t)^j) = 0\]
and for $j=d-1$ it equals $X^{q^k}$.
\end{Lemma}

\begin{proof}
Recall the definition of the map $\delta_{1,\bz}^M$ from \eqref{D:delta_1^M} and of the $\C_\infty[\tau]$-basis elements $g_k$ from Definition \ref{D:t-motive}. Then, observe that 
\[b_k(t)^{d-j}b_{k+1}(t)^j = \tau^k(g_{j+1}),\]
for $k\geq 0$ and $0\leq j\leq d-1$. Thus we apply Lemma \ref{L:tau structure delta_1} to conclude
\[\delta_{1,\bz}^M(b_k(t)^{d-j}b_{k+1}(t)^j) = \delta_{1,\bz}^M(\tau^k(g_{j+1})) = \delta_{1,\bz}^M(g_{j+1})^{q^k},\]
and the lemma follows from \eqref{D:delta_1^M}, since $\bz = (0,\dots,0,X)^\top$.
\end{proof}

\begin{Proposition}\label{P:delta1 of Omega}

Let $\bz = (0,\dots,0,X)^\top$. We have
\[\delta_{1,\bz}^M\left ( \tau^k\left(\frac{\tpi^d}{(t-\theta)}\left (\Omega\twistinv\right )^d \right )\right ) =  \mathcal{L}\binom{X^{q^k}}{dq^k}.\]
\end{Proposition}

\begin{proof}
Recall the definition of $\mathbb M$ from \eqref{D:delta_1 extended def}. It follows from \cite[(5.4)]{GRE} that $$\tau^k\left(\frac{\tpi^d}{(t-\theta)}\left (\Omega\twistinv\right )^d \right ) \in \mathbb M.$$ Then, combining \cite[Theorem 5.10]{GRE} and Lemma \ref{L:tau structure delta_1} we get
\[\delta_{1,\bz}^M\left ( \tau^k\left(\frac{\tpi^d}{(t-\theta)}\left (\Omega\twistinv\right )^d \right )\right ) = (\text{Log}_{C^{\otimes d}}(\bz))_d^{q^k}\]
where we use that $\delta_{1,\bz}^M$ is $\C_\infty$-linear, and that $\Omega\twistinv = 1/\omega_C$. The theorem follows using the fact that the $d$th coordinate of the logarithm of the $d$th tensor power of Carlitz evaluated at $\bz$ is the polylogarithm $\mathcal{L}\binom{X}{d}$ (this can be seen from \cite[Remark, p. 172]{AND&THA}, for example) and Lemma \ref{L:tau structure delta_1}.
\end{proof}

Now we study the identities arising from \eqref{E:I_i decomposition}. Extend the definition of $\delta_{1,\bz}^M$ to act on $\hM^d$ by acting coordinatewise.

\begin{Proposition}\label{P:delta1 of L product}
Let $\bz = (0,\dots,0,X)^\top$. We have
\begin{align*}
\delta_{1,\bz}^M\bigg ( \lim_{i\to\infty}(-1)^{k}\sum_{n=k}^{i-1}&\Gamma_n\sum_{0\leq i_1<\cdots <i_k<n}  (L_{i_k}\circ \cdots \circ L_{i_1})(W_0)\cdot H_{\theta,\theta^{q^n}} \cdot\bb_n \bigg )_d\\
&= w^{q^k}(-1)^k\sum_{\underline{m}}c_{\underline{m}}
\mathcal{L}\begin{pmatrix} X&m_1q & \cdots &m_kq^k\\ d&d(q-1) &\cdots& d(q-1)q^{k-1}\end{pmatrix}.
\end{align*}

\end{Proposition}

We prove this proposition with a series of lemmas.

\begin{Lemma}\label{L:L_k in terms of Lambda}
We have
\[(L_{i_k}\circ \cdots \circ L_{i_1})(W)\cdot H_{\theta,\theta^{q^n}} = \mathcal{L}_{i_k}\cdot \mathcal{L}_{i_{k-1}}\twist \cdots \mathcal{L}_{i_1}\twistk{k-1}W\twistk{k} H_{\theta^{q^k},\theta^{q^n}}.\]
\end{Lemma}

\begin{proof}
This follows quickly from Definition $\ref{D:L_k def}$, Definition \ref{D:Lambda_k} and \eqref{groupoidH}.
\end{proof}

\begin{Lemma}\label{L:delta1 of H bb}
Let $\bb_n$ be as in Definition \eqref{D:I_i def} and $\bz = (0,\dots,0,X)^\top$. Then for any $a, b\in \ZZ$ and any $n\geq 0$ we have
\[\delta_{1,\bz}^M( H_{\theta^a,\theta^b}\cdot\bb_n) = (0,\dots,0,X^{q^n})^\top.\]
\end{Lemma}

\begin{proof}
This follows because $H_{\theta^a,\theta^b}$ is a lower triangular matrix with $1$'s along the diagonal and from Lemma \ref{L:delta1 of b_k}.
\end{proof}

\begin{proof}[Proof of Proposition \ref{P:delta1 of L product}]
Using Lemmas \ref{L:L_k in terms of Lambda} and \ref{L:delta1 of H bb} and the $\C_\infty$-linearity of $\delta_{1,\bz}^M$ we get
\begin{align*}
\delta_{1,\bz}^M&\bigg ( \lim_{i\to\infty}(-1)^{k}\sum_{n=k}^{i-1}\Gamma_n\sum_{0\leq i_1<\cdots <i_k<n}  (L_{i_k}\circ \cdots \circ L_{i_1})(W_0)\cdot H_{\theta,\theta^{q^n}} \cdot\bb_n \bigg )_d\\
&=\delta_{1,\bz}^M\bigg ( \lim_{i\to\infty}(-1)^{k}\sum_{n=k}^{i-1}\Gamma_n\sum_{0\leq i_1<\cdots <i_k<n}  \mathcal{L}_{i_k}\cdot \mathcal{L}_{i_{k-1}}\twist \cdots \mathcal{L}_{i_1}\twistk{k-1}W_0\twistk{k} H_{\theta^{q^k},\theta^{q^n}} \cdot\bb_n \bigg )_d\\
&= \bigg (\lim_{i\to\infty}(-1)^{k}\sum_{n=k}^{i-1}\Gamma_n\sum_{0\leq i_1<\cdots <i_k<n}  \mathcal{L}_{i_k}\cdot \mathcal{L}_{i_{k-1}}\twist \cdots \mathcal{L}_{i_1}\twistk{k-1}W_0\twistk{k} \delta_{1,\bz}^M\bigg (H_{\theta^{q^k},\theta^{q^n}} \cdot\bb_n \bigg )\bigg )_d\\&= \bigg (\lim_{i\to\infty}(-1)^{k}\sum_{n=k}^{i-1}\Gamma_n \mathcal{L}_{<n}(k)\cdot W_0\twistk{k} \cdot (0,\dots,0,X^{q^k})^\top\bigg )_d.
\end{align*}
The Proposition then follows by observing that the bottom right coordinate of $\Gamma_n$ equals $l_n^{-d}$ and from Theorem \ref{nathan-theorem}.
\end{proof}

\begin{Remark}\label{R:diagram}
We summarize the results of this section with a diagram illustrating how our identities occur in $\hM$, then are mapped down to identities in $\C_\infty$ using the cycle integration map $\delta_{1,\bz}^M$, with $\bz = (0,\dots,0,X)$:

\[\xymatrix@=2cm{\frac{1}{D_m^d}[m]^{d-1}\tau^m\left(\frac{\tpi^d}{(t-\theta)}\left (\Omega\twistinv\right )^d \right )\ar@{->}[d]^{\delta_{1,\bz}^M}\ar@{->}[r]^<{\quad\quad\quad\quad\quad\substack{\text{Noncommutative} \\ \text{Factorization}}} & \left ((-1)^{m}\Gamma_n\sum_{0\leq i_1<\cdots <i_m}  (L_{i_m}\circ \cdots \circ L_{i_1})(W_0)\cdot H_{\theta,\theta^{q^n}}\cdot\bb_n\right )_d \ar@{->}[d]^{\delta_{1,\bz}^M}\\
\frac{1}{D_m^d}[m]^{d-1}\mathcal{L}\binom{X^{q^m}}{dq^m} \ar@{->}[r]^-{\text{\textbf{Theorem C}}} & (-1)^m\sum_{\underline{m}}c_{\underline{m}} \mathcal{L} {\begin{pmatrix}X& m_1q & \cdots & m_kq^m \\ d & d(q-1) & \cdots & d(q-1)q^{m-1}  \end{pmatrix}}}\]
\end{Remark}

 \section{Identities for multiple zeta values and multiple polylogarithms at one}\label{applications-to-MZV-etc}
 
In this section we discuss applications of our results to identities (at the finite level) for 
Carlitz multiple polylogarithms evaluated at $(1,\ldots,1)$ and Thakur's multiple zeta values.
 
 \subsection{Carlitz multiple polylogarithms at one} We review notations, tools, definitions, and essential properties.
We call an $r$-tuple $\mathfrak{n}=(n_1,\ldots,n_r)$ with $r\geq1$ and with positive integers as coefficients an {\em array}. The integer $r$ is called the {\em depth} of $\mathfrak{n}$. We also assume by convention that there exists a unique array $\emptyset$ of depth zero. Given an array $\mathfrak{n}$ of depth $r>0$ as above, its {\em weight} is the positive integer $\sum_in_i$. The weight of $\emptyset$ is zero by definition. Consider an array $\mathfrak{n}$ of depth $r>0$.
 We write $\mathfrak{n}_1=(n_1)$ and $\mathfrak{n}_-=
 (n_2,\ldots,n_r)$, an array of depth $r-1$. If $r=1$ we set $\mathfrak{n}_-:=\emptyset$. To avoid complications arrays with depth one will be identified with positive integers. The {\em concatenation} $\mathfrak{mn}$ of two arrays $\mathfrak{m}=(m_1,\ldots,m_r)$
 and $\mathfrak{n}=(n_1,\ldots,n_s)$ is the array $\mathfrak{mn}=(m_1,\ldots,m_r,n_1,\ldots,n_s)$, of depth 
 the sum of the depths $r$ of $\mathfrak{m}$ and $s$ of $\mathfrak{n}$. To avoid ambiguities and in order not to confuse it with 
 the classical product of $\NN^*$, we will sometimes denote by $\cdot$ the concatenation operation between arrays.
 Define, inductively, with $n$ a positive integer and $\mathfrak{n}$ an array of depth $\geq 1$,
 \begin{multline*}
L_i(n):=l_i^{-n},\quad i\geq 0,\quad L_i(n):=0,\quad i<0,\\
 L_{<i}(n):=\sum_{j=0}^{i-1}L_j(n),\quad i\geq 0,\\
 L_i(\mathfrak{n}):=L_i(\mathfrak{n}_1)L_{<i}(\mathfrak{n}_-),\quad i\geq 0,\\
 L_{<i}(\mathfrak{n}):=\sum_{j=0}^{i-1}L_j(\mathfrak{n}),\quad i\geq 0.
 \end{multline*}
We have thus associated, to each array $\mathfrak{n}$ of depth $>0$, sequences $(L_i(\mathfrak{n}))_{i\geq 0}$ and $(L_{<i}(\mathfrak{n}))_{i\geq 0}$ in $K$. 
The {\em multiple polylogarithm at one} associated to $\mathfrak{n}$ is the well defined element of $K_\infty$ given by
$$\operatorname{Li}(\mathfrak{n}):=\lim_{i\rightarrow\infty}L_{<i}(\mathfrak{n})=\sum_{i=0}^\infty L_i(\mathfrak{n}).$$
We extend this formalism to handle linear combinations with coefficients in a field $F$ containing $\FF_p$. Denote by 
$\mathfrak{H}(F)$ the $F$-vector space generated by formal linear combinations
$$\sum_ih_i\mathfrak{m}_i,\quad h_i\in F,\quad \mathfrak{m}_i\text{ array}.$$ A basis of $\mathfrak{H}(F)$
is therefore given by the set $\{\mathfrak{n}:\mathfrak{n}\text{ array}\}$, which contains the empty array.
Letters in fraktur font are exclusively used to designate elements of $\mathfrak{H}(F)$. 
Extending the concatenation $F$-bilinearly endows $\mathfrak{H}(F)$ with the structure of a unital non-commutative $F$-algebra. The unit is represented by the unique empty array that we denote by $1$.
If $F$ is a field extension of $K$, we have well defined $F$-linear maps
$$L_i,L_{<i}:\mathfrak{H}(F)\rightarrow F,\quad \operatorname{Li}:\mathfrak{H}(K)\rightarrow K_\infty.$$
By a variant of Hoffman's Theorem on quasi-shuffle algebras \cite{HOF} (see also \cite[Theorem 1.150]{BUR&FRE})
 we can construct a quasi-shuffle product $*$ on $\mathfrak{H}(K)$ associated with the following {\em diamond operation} $\diamond$: $m\diamond n=m+n$ with $m,n\in\NN^*$ (it is extended $F$-linearly to the $F$-linear span of arrays of depth zero and one). This defines a 
 structure $(\mathfrak{H}(K),+,*)$ of unital $K$-algebra over $\mathfrak{H}(K)$ such that for all 
 $i$ and arrays $\mathfrak{m},\mathfrak{n}$,
 $$L_{<i}(\mathfrak{m}*\mathfrak{n})=L_{<i}(\mathfrak{m})L_{<i}(\mathfrak{n}),\quad \operatorname{Li}(\mathfrak{m}*\mathfrak{n})=\operatorname{Li}(\mathfrak{m})\operatorname{Li}(\mathfrak{n}).$$
 The product, often called {\em stuffle product}, is defined, inductively on depths, by 
 \begin{multline*}
 1*\mathfrak{m}=\mathfrak{m}*1=\mathfrak{m},\quad 
 \mathfrak{m}*\mathfrak{n}=(\mathfrak{m})_1(\mathfrak{m}_-*\mathfrak{n})+
 (\mathfrak{n})_1(\mathfrak{m}*\mathfrak{n}_-)+( (\mathfrak{m})_1\diamond(\mathfrak{n}_1))(\mathfrak{m}_-*\mathfrak{n}_-).
 \end{multline*}
 The diamond operation $\diamond$ itself can be extended to a product $\diamond$ over $\mathfrak{H}(K)$ in such a way that, for all $i\geq 0$ and $\mathfrak{m},\mathfrak{n}$ arrays,
 $$L_i(\mathfrak{m}\diamond\mathfrak{n})=L_i(\mathfrak{m})L_i(\mathfrak{n})\quad \operatorname{Li}(\mathfrak{m}\diamond\mathfrak{n})=\sum_{i\geq 0}L_i(\mathfrak{m})L_i(\mathfrak{n}).$$
 Explicitly, one defines, inductively on depths,
 \begin{equation*}
 \mathfrak{m}\diamond 1=1\diamond\mathfrak{m}=\mathfrak{m},\quad
 \mathfrak{m}\diamond\mathfrak{n}=((\mathfrak{m})_1\diamond(\mathfrak{n})_1)(\mathfrak{m}_-*\mathfrak{n}_-).
 \end{equation*}
 
 \subsection{Identities for Carlitz's multiple polylogarithms}
 Let us choose $d\geq 1$. Let us set:
 \begin{equation}\label{muij-def}
 \mu_{i,j}:=f_d(\theta,\theta^{q^{i+1}},\theta^{q^j})l_i^{-d(q-1)}.\end{equation}
 To begin with, we are going to show that, for all $j$, $$\mu_{i,j}=L_i(\mathfrak{k}_j),\quad i\geq0$$ for explicitly computable elements $\mathfrak{k}_j$ of
 $\mathfrak{H}(K)$, independent of $i\geq 0$. This is essential for the proof of the main result of this section, Theorem \ref{identity-scalar}. We have:
 
 \begin{Proposition}\label{lemma-muij}
 For all $j\geq1$ and for all $i\geq 0$ we have
 $$\mu_{i,j}=\sum_{k=0}^{d-1}\binom{d}{k}(\theta-\theta^q)^kf_{d-k}(\theta,\theta^q,\theta^{q^j})L_i(\mathfrak{c}_k),$$
 where 
 \begin{equation}\label{defi-c-k}
 \mathfrak{c}_k=\big((d-k)(q-1),(q-1)^{*k}\big).
 \end{equation}
  \end{Proposition}
  In (\ref{defi-c-k}), $\mathfrak{c}_k$ is the concatenation of the array $(d-k)(q-1)$ (depth one) and 
  the array $(q-1)^{*k}$ is the $k$-th power of the array of depth one $(q-1)$ for the multiplication $*$ (see Remark \ref{rem-stuffle}).
  Hence
  $$\mathfrak{k}_j=\sum_{k=0}^{d-1}\binom{d}{k}(\theta-\theta^q)^kf_{d-k}(\theta,\theta^q,\theta^{q^j})\mathfrak{c}_k.$$ Note that the arrays $\mathfrak{c}_k$ do not depend on $j$. 
One of the conventions of our paper is that empty products are equal to one. 
 In conformity with this we set $\mathfrak{m}^{*0}=1$ the empty array, for each array $\mathfrak{m}$. Therefore, if $k=0$,
 $\mathfrak{c}_k=d(q-1)$ (depth one). 
  
  The proof of Proposition \ref{lemma-muij} relies on the next elementary result.
  
  \begin{Lemma}\label{xyzt-f}
 For indeterminates $x,y,z,t$ and for all $d\geq 0$, we have the formula
 $$f_d(x,y,z)=\sum_{k=0}^{d-1}\binom{d}{k}(t-y)^kf_{d-k}(x,t,y).$$
 \end{Lemma}
 
 \begin{proof}
 From the definition:
 \begin{eqnarray*}
 f_d(x,y,z)&=&\frac{(z-y)^d-(x-y)^d}{z-x}\\
 &=&\frac{(z-t+t-y)^d-(x-t+t-y)^d}{z-x}\\
 &=&\sum_{k=0}^{d-1}\binom{d}{k}(t-y)^k\frac{(z-t)^{d-k}-(x-t)^{d-k}}{z-x}\\
 &=&\sum_{k=0}^{d-1}\binom{d}{k}(t-y)^kf_{d-k}(x,t,z).
 \end{eqnarray*}
 \end{proof}
 In particular, we have the following formula:
 \begin{equation}\label{fd-expansion}
 f_d(x,y,z)=\sum_{i=0}^{d-1}(-1)^i\binom{d}{i}\frac{x^{d-i}-z^{d-i}}{x-z}y^i.
 \end{equation}

  \begin{proof}[Proof of Proposition \ref{lemma-muij}] We apply the formula of Lemma \ref{xyzt-f}
  with $x=\theta,y=\theta^{q^{i+1}},z=\theta^{q^j},t=\theta^q$ to the definition of $\mu_{i,j}$:
  \begin{eqnarray*}
  \mu_{i,j}&=&f_d(\theta,\theta^{q^{i+1}},\theta^{q^j})l_i^{-d(q-1)}\\
  &=&\sum_{k=0}^{d-1}\binom{d}{k}f_{d-k}(\theta,\theta^q,\theta^{q^j})(\theta-\theta^{q^i})^{qk}l_i^{-d(q-1)}.
  \end{eqnarray*}
  Observe that
  \begin{eqnarray*}
  (\theta-\theta^{q^i})^{qk}l_i^{-d(q-1)}&=&(\theta-\theta^{q^i})^{qk}l_i^{-(d-k)(q-1)-kq+k}\\
  &=&l_i^{-(d-k)(q-1)}\Big(\frac{l_i}{l_{i-1}^q}\Big)^k.
  \end{eqnarray*}
  One deduces, from Theorem \ref{theorem-factor-complete} (case $d=1$, $m=1$) the formula
  \begin{equation}\label{crucial-formula}
  \frac{l_i}{l_{i-1}^q}=(\theta-\theta^q)L_{<i}(q-1),\quad i\geq0.
  \end{equation}
  Hence
  $$l_i^{-(d-k)(q-1)}\Big(\frac{l_i}{l_{i-1}^q}\Big)^k=(\theta-\theta^q)^kL_i\big(\mathfrak{c}_k\big)$$ which proves the formula of our statement.

  \end{proof}
   
   In the case $j=1$ the formula of Proposition \ref{lemma-muij} is simpler:
  \begin{equation}\label{special-formula-mui1}
  \mu_{i,1}=(\theta-\theta^q)^{d-1}\sum_{k=0}^{d-1}\binom{d}{k}L_i(\mathfrak{c}_k).
  \end{equation}
   This is due to the fact that $f_s(x,z,z)=(x-z)^{s-1}$ for all $s$.

 \begin{Remark}\label{rem-stuffle} Given a positive integer $n$, the powers $(n)^{*k}\in\mathfrak{H}(\FF_p)$ for the stuffle product can be computed explicitly by means of elementary combinatorics involving multinomial coefficients (recall our convention in the case $k=0$). We have the formula (with multinomial coefficients):
 \begin{equation} (n)^{*k}=\sum_{r\geq1}\sum_{(n_1,\ldots,n_r)\in (\NN^*)^r}\begin{pmatrix} k\\ n_1, \ldots, n_r\end{pmatrix}(nn_1,\ldots,nn_r)\label{multinomial}\in\mathfrak{H}(\FF_p),\end{equation} where the second sum runs over the arrays $(n_1,\ldots,n_r)\in(\NN^*)^r$ such that $n_1+\cdots+n_r=k$. 
\end{Remark}
 
 \begin{Remark} Consider the $K$-linear map $\mathfrak{H}(K)\xrightarrow{\operatorname{Li}}K_\infty$ which sends 
 an array $\mathfrak{n}$ to $\operatorname{Li}(\mathfrak{n})\in K_\infty$. Let $w$ be a positive integer. If $\mathfrak{H}(K)_w$ denotes the $K$-span in $\mathfrak{H}(K)$ of the arrays of weight $w$, the image of the restriction $\operatorname{Li}_w$  of $\operatorname{Li}$ on $\mathfrak{H}(K)_w$ has the basis $(\operatorname{Li}_w(\mathfrak{n}))_{\mathfrak{n}}$, where 
 $\mathfrak{n}=(n_1,\ldots,n_r)\in\mathfrak{H}(K)_w$ is such that $q\nmid n_i$ for all $1\leq i\leq r$. 
 This is one of the main results in both papers \cite{CHA&CHE&MIS,IKNLNDHP2}. Let us denote by $\mathcal{B}_w^{\operatorname{ND}}$ the set of the above arrays (\footnote{The exponent $\operatorname{ND}$ stands for {\em Ngo Dac} because the first occurrences of these important sets of arrays, later identified with bases of spaces of multiple zeta values, goes back to his paper \cite{NGO}.}).
 From (\ref{defi-c-k}) we see that for all $k<d$, $\mathfrak{c}_k\in\mathfrak{H}(K)_{d(q-1)}$.
 Suppose that $d<q$. Then, for every $k\leq d$, every array occurring in the expansion of $\mathfrak{c}_k\in\mathfrak{H}(K)_{d(q-1)}$ for $0\leq k\leq d-1$ belongs to $\mathcal{B}_{d(q-1)}^{\operatorname{ND}}$. In this case, the relations of Proposition \ref{lemma-muij} are uniquely determined by the parameters.
 \end{Remark}
  
 We are now in condition to prove our Theorem \ref{identity-scalar}.
 Consider two arrays $\mathfrak{m}$ and $\mathfrak{n}$ in $\mathfrak{H}(K)$. Following 
 \cite{IKNLNDHP}, we introduce the {\em triangle product}
 $$\mathfrak{m}\triangleright\mathfrak{n}:=\mathfrak{m}_1(\mathfrak{m}_-*\mathfrak{n}),$$
 extended in the obvious way in the case when one of the arrays has depth zero.
 This product can be used to compress the expression of the stuffle product:
 $$\mathfrak{m}*\mathfrak{n}=\mathfrak{m}\diamond\mathfrak{n}+\mathfrak{m}\triangleright\mathfrak{n}+\mathfrak{n}\triangleright\mathfrak{m}.$$ By \cite[Lemma 2.1]{NDNCP}, this is an associative operation. It is fundamentally involved in the construction of the co-product of an Hopf algebra, see \cite[Theorem B]{IKNLNDHP}. For us the main reason for which this operation is important is given by the next elementary result.

 \begin{Proposition}\label{nested-sum-more-general}
 Consider arrays $\mathfrak{k}_1,\ldots,\mathfrak{k}_m$ with positive depths.
 Then for all $k\geq 0$,
 $$L_{<k}(\mathfrak{k}_1\triangleright\mathfrak{k}_2\triangleright\cdots\triangleright\mathfrak{k}_{m-1}\triangleright\mathfrak{k}_m)=\sum_{k>i_1>\cdots>i_m\geq0}L_{i_1}(\mathfrak{k}_1)\cdots L_{i_m}(\mathfrak{k}_m).$$
 \end{Proposition}
 
 \begin{proof}
 We proceed by induction on $m\geq 1$. If $m=1$ the property is obvious. Suppose that 
 for $\mathfrak{k}'=\mathfrak{k}_2\triangleright\cdots\triangleright\mathfrak{k}_m$
 we have, for all $i_1$,
 $$L_{<i_1}(\mathfrak{k}')=\sum_{i_1>i_2>\cdots>i_m\geq0}L_{i_2}(\mathfrak{k}_2)\cdots L_{i_m}(\mathfrak{k}_m).$$
 Then 
 \begin{eqnarray*}
 \sum_{i_1=0}^{k-1}L_{i_1}(\mathfrak{k}_1)\sum_{i_1>i_2>\cdots>i_m\geq0}L_{i_2}(\mathfrak{k}_2)\cdots L_{i_m}(\mathfrak{k}_m)&=&\sum_{i_1=0}^{k-1}L_{i_1}(\mathfrak{k}_1)L_{<i_1}(\mathfrak{k}')\\
 &=&L_{<k}\big((\mathfrak{k}_1)_1\cdot((\mathfrak{k}_1)_-*\mathfrak{k}')\big).
 \end{eqnarray*}
 \end{proof}
  Define:
 \begin{equation}\label{definition-alphaij}
\alpha_{k,j}=\binom{d}{k}(\theta-\theta^q)^kf_{d-k}(\theta,\theta^q,\theta^{q^j})\in K^\times,\quad 0\leq k<d,\quad j\geq 1
\end{equation}
(it is easy to see that these elements of $K$ are all non-zero).
\begin{Theorem}\label{identity-scalar} For all $k\geq m\geq1$ the following formula holds
\begin{multline*}
L_{k-m}(dq^m)=\\=l_{k-m}^{-dq^m}=(-1)^mD_m
\sum_{h_1=0}^{d-1}\cdots\sum_{h_m=0}^{d-1}\alpha_{h_1,m}\alpha_{h_2,m-1}^q\cdots\alpha_{h_r,1}^{q^{m-1}}L_{k}\Big((d)\triangleright\mathfrak{c}_{h_1}\triangleright\mathfrak{c}_{h_2}^{*q}\triangleright\cdots\triangleright\mathfrak{c}_{h_m}^{*q^{m-1}}\Big),
\end{multline*} where the coefficients $\alpha_{i,j}$ have been defined in (\ref{definition-alphaij}) and the arrays $\mathfrak{c}_i$ have been introduced in Proposition \ref{lemma-muij}.
\end{Theorem}
 
 Decomposing as a linear combination of arrays 
 $(d)\triangleright\mathfrak{c}_{h_1}\triangleright\mathfrak{c}_{h_2}^{*q}\triangleright\cdots\triangleright\mathfrak{c}_{h_m}^{*q^{m-1}}=\sum_ic_i\mathfrak{k}_i$, we have, for all $i$,
 $$\mathfrak{k}_i=(d,x_1(q-1),\ldots,x_r(q-1))$$
 where $(x_1,\ldots,x_r)$ is an array of weight $d\frac{q^m-1}{q-1}$ hence the relations of Theorem \ref{identity-scalar} are non-trivial.
 The problem of the computation of $p$-powers of arrays (for the stuffle product) is studied in \cite[\S 2.3]{NDNCP}. In their Proposition
 2.7 the authors of this reference prove that, given an array 
$\mathfrak{m}=(m_1,\ldots,m_r)$, $\mathfrak{m}^{*p^k}=(m_1p^k,\ldots,m_rp^k)$. 
 
\begin{proof}[Proof of Theorem \ref{identity-scalar}]
First of all, note that, in the definition (\ref{defi-lambda<}) of $\lambda_{<k}(m)$, we can collect the common factor 
\begin{equation}\label{simplified1}
[m-1]^{-q(d-1)}[m-2]^{-q^2(d-1)}\cdots[1]^{-q^{m-1}(d-1)}=D_{m-1}^{-q(d-1)}.
\end{equation}
Hence we can write, taking into account the definition (\ref{muij-def}),
\begin{equation}\label{sum-lambda}
\lambda_{<k}(m)=D_{m-1}^{-q(d-1)}\sum_{k>i_1>\cdots>i_m}\mu_{i_1,m}\mu_{i_2,m-1}^q\cdots\mu_{i_m,1}^{q^{m-1}}.\end{equation}
By Proposition \ref{lemma-muij} we can write
\begin{equation}\label{simplified-id-muij}
\mu_{i,j}=\sum_{h=0}^{d-1}\alpha_{h,j}L_i(\mathfrak{c}_h),
\end{equation}
where 
$\mathfrak{c}_k$ is defined in (\ref{defi-c-k}).
Hence we have:
\begin{multline*}
\lambda_{<k}(m)=(D_{m-1}^q)^{-(d-1)}\sum_{h_1=0}^{d-1}\cdots\sum_{h_m=0}^{d-1}\alpha_{h_1,m}\alpha_{h_2,m-1}^q\cdots\alpha_{h_r,1}^{q^{m-1}}\times \\ \times\sum_{k>i_1>\cdots>i_m\geq0}L_{i_1}(\mathfrak{c}_{h_1})L_{i_2}(\mathfrak{c}_{h_2}^{*q})\cdots L_{i_m}(\mathfrak{c}_{h_m}^{*q^{m-1}}).\end{multline*}
Applying Proposition \ref{nested-sum-more-general} we come to the next assertion, recalling (\ref{definition-alphaij}). We have a decomposition:
$$\lambda_{<k}(m)=(D_{m-1}^q)^{-(d-1)}\sum_{h_1=0}^{d-1}\cdots\sum_{h_m=0}^{d-1}\alpha_{h_1,m}\alpha_{h_2,m-1}^q\cdots\alpha_{h_r,1}^{q^{m-1}}L_{<k}(\mathfrak{c}_{h_1}\triangleright\mathfrak{c}_{h_2}^{*q}\triangleright\cdots\triangleright\mathfrak{c}_{h_m}^{*q^{m-1}}).$$
By Theorem \ref{nathan-theorem} we also have $\lambda_{<k}(m)=(-1)^ml_k^d[m]^{d-1}D_m^{-d}l_{k-m}^{-dq^m}$.
Therefore
\begin{multline*}
L_{k-m}(dq^m)=l_{k-m}^{-dq^m}=\\=(-1)^mD_m
\sum_{h_1=0}^{d-1}\cdots\sum_{h_m=0}^{d-1}\alpha_{h_1,m}\alpha_{h_2,m-1}^q\cdots\alpha_{h_r,1}^{q^{m-1}}L_{k}\Big((d)\triangleright\mathfrak{c}_{h_1}\triangleright\mathfrak{c}_{h_2}^{*q}\triangleright\cdots\triangleright\mathfrak{c}_{h_m}^{*q^{m-1}}\Big).\end{multline*}
This finishes the proof.
 \end{proof}
 
 \begin{Remark} In Theorem \ref{identity-scalar}, the explicit expansion of 
 $$\mathfrak{k}:=(d)\triangleright\mathfrak{c}_{h_1}\triangleright\mathfrak{c}_{h_2}^{*q}\triangleright\cdots\triangleright\mathfrak{c}_{h_m}^{*q^{m-1}}$$ (depending on $h_1,\ldots,h_m$) in the basis $\mathcal{B}_{dq^m}^{\operatorname{ND}}$ can be reached inductively by iterated application of our formulas in a fashion similar to \cite{CHA&CHE&MIS,IKNLNDHP}, but we are unable to directly observe a closed formula. 
\end{Remark}
 
 \subsection{Multiple power sums}\label{multiple-power-sums}
 
 Our results are stated for sequences of multiple sums of the type $L_i(\mathfrak{m})$ or $L_{<i}(\mathfrak{n})$, but after \cite{NGO}, their $K$-span, which is an $\FF_p$-algebra, equals the $K$-span of sequences of multiple power sums. This is rooted in the notion of multiple zeta values of Thakur (see for example \cite{NGO2,THA3}). Given $(n_1,\ldots,n_r)\in(\NN^*)^r$ the associated multiple zeta value of Thakur is:
$$\zeta_A(n_1,\ldots,n_r):=\sum_{i_1>\cdots>i_r\geq 0}\sum_{\begin{smallmatrix} a_1,\ldots,a_r\in A\\
\text{monic}\\
\deg_\theta(a_j)=i_j,\\
j=1,\ldots,r\end{smallmatrix}}\frac{1}{a_1^{n_1}\cdots a_r^{n_r}}\in K_\infty,\quad n_1,\ldots,n_r\in\NN^*.$$
 Define, inductively, with $n$ a positive integer and $\mathfrak{n}$ an array of depth $\geq 1$,
 \begin{multline*}
S_i(n):=\sum_{\begin{smallmatrix}a\in A\\ \text{monic}\\ \deg_\theta(a)=i\end{smallmatrix}}a^{-n},\quad i\geq 0,\quad S_i(n)=0,\quad i<0,\\
 S_{<i}(n):=\sum_{j=0}^{i-1}S_j(n),\quad i\geq 0,\\
 S_i(\mathfrak{n}):=S_i(\mathfrak{n}_1)S_{<i}(\mathfrak{n}_-),\quad i\geq 0,\\
 S_{<i}(\mathfrak{n}):=\sum_{j=0}^{i-1}S_j(\mathfrak{n}),\quad i\geq 0.
 \end{multline*}
 The above are called {\em multiple power sums}. 
 They define sequences in $K$ and satisfy properties similar to those of $L_i(\mathfrak{n})$ and $L_{<i}(\mathfrak{n})$. In particular we can extend $S_i$ and $S_{<i}$ to linear maps $\mathfrak{H}(F)\rightarrow F$ (with $F$ a field containing $\FF_p$) and there is a unital algebra structure $(\mathfrak{H}(\FF_p),+,\shuffle)$ such that, over $\mathfrak{H}(K)$ and for all $\mathfrak{m},\mathfrak{n}$ arrays,
 $$S_{<i}(\mathfrak{m}\shuffle\mathfrak{n})=S_{<i}(\mathfrak{m})S_{<i}(\mathfrak{n}),\quad i\geq1.$$
 There also is a product governing the power sums $S_i$. We may denote it by 
 $\odot$ to distinguish it from $\diamond$. We have
 $$S_i(\mathfrak{m}\odot\mathfrak{n})=S_i(\mathfrak{m})S_i(\mathfrak{n}),\quad i\geq 0,\quad \mathfrak{m},\mathfrak{n}\text{ arrays},$$ see \cite[Theorem A]{IKNLNDHP}. By using Ngo Dac's \cite[Theorem A]{NGO}, Theorem \ref{identity-scalar} implies families of non-trivial linear relations
 between multiple power sums and ultimately, between Thakur's multiple zeta values. Indeed:
 $$\zeta_A(\mathfrak{n})=\lim_{i\rightarrow\infty}S_{<i}(\mathfrak{n})=\sum_{i=0}^\infty S_i(\mathfrak{n})\in K_\infty$$
 for every array $\mathfrak{n}$ of depth $\geq 1$.
 \subsection{The case $m=1$ and $d\leq q$ in Theorem \ref{identity-scalar}} This can be handled in a slightly different way and leads to explicit identities as we explain here. 
 Observe that by definition,
 $$\mu_{k,1}=\frac{(\theta-\theta^{q^{k+1}})^d-(\theta^q-\theta^{q^{k+1}})^d}{\theta-\theta^q}l_k^dl_k^{-dq}.$$
 Hence we can write:
 $$\mu_{k,1}=\frac{1}{\theta-\theta^q}\left(\Big(\frac{l_{k+1}}{l_k^q}\Big)^d-\Big(\frac{l_{k}}{l_{k-1}^q}\Big)^d\right).$$
 By (\ref{crucial-formula}) we can therefore write:
 \begin{equation}\label{depth-one}
 \mu_{k,1}=(\theta-\theta^q)^{d-1}\Big(L_{<k+1}(q-1)^d-L_{<k}(q-1)^d\Big),
 \end{equation}
  which implies
  \begin{equation}\label{formula-mui1}
  \sum_{i<k}\mu_{i,1}=(\theta-\theta^q)^{d-1}L_{<k}(q-1)^d.
  \end{equation}
 So far, we have not used the condition on $d$.
  A formula of Thakur \cite[\S 3.2.2]{THA} implies that
 $$(q-1)^{\shuffle d}=d(q-1),\quad d\leq q.$$ Coming back to
 (\ref{formula-mui1}), note that $L_{<i}(q-1)=S_{<i}(q-1)$. This ensures that
$$\sum_{i=0}^{k-1}\mu_{i,1}=(\theta-\theta^q)^{d-1}S_{<k}(d(q-1)).$$
Again by the fact that $d\leq q$, we have that $l_k^{-d}=S_{k}(d)$. Coming back to (\ref{sum-lambda}), with $m=1$ and $d\leq q$, by Theorem \ref{nathan-theorem}:
$$l_k^{-d}\lambda_{<k}(1)=(\theta-\theta^q)^{-1}S_{k-1}(dq).$$
We have reached Thakur's formula (see \cite[Theorem 5]{THA})
$$S_{k-1}(dq)=(\theta-\theta^q)^dS_k(d,d(q-1)),\quad k\geq0,\quad 1\leq d\leq q.$$

\end{document}